\documentclass[12pt]{amsart}
\usepackage{ amsmath, amsthm, amsfonts, amssymb, color}
 \usepackage{mathrsfs}
\usepackage{amsfonts, amsmath}
 \usepackage{amsmath,amstext,amsthm,amssymb,amsxtra}
 \usepackage{txfonts} 
 \usepackage[colorlinks, citecolor=blue,pagebackref,hypertexnames=false]{hyperref}
 \allowdisplaybreaks
 \usepackage{pgf,tikz}

\begin{document}

 \baselineskip 16.6pt
\hfuzz=6pt

\widowpenalty=10000

\newtheorem{cl}{Claim}
\newtheorem{theorem}{Theorem}[section]
\newtheorem{proposition}[theorem]{Proposition}
\newtheorem{coro}[theorem]{Corollary}
\newtheorem{lemma}[theorem]{Lemma}
\newtheorem{definition}[theorem]{Definition}
\newtheorem{assum}{Assumption}[section]
\newtheorem{example}[theorem]{Example}
\newtheorem{remark}[theorem]{Remark}
\renewcommand{\theequation}
{\thesection.\arabic{equation}}

\def\SL{\sqrt H}

\newcommand{\mar}[1]{{\marginpar{\sffamily{\scriptsize
        #1}}}}

\newcommand{\as}[1]{{\mar{AS:#1}}}

\newcommand\R{\mathbb{R}}
\newcommand\RR{\mathbb{R}}
\newcommand\CC{\mathbb{C}}
\newcommand\NN{\mathbb{N}}
\newcommand\ZZ{\mathbb{Z}}
\def\RN {\mathbb{R}^n}
\renewcommand\Re{\operatorname{Re}}
\renewcommand\Im{\operatorname{Im}}

\newcommand{\mc}{\mathcal}
\newcommand\D{\mathcal{D}}
\def\hs{\hspace{0.33cm}}
\newcommand{\la}{\alpha}
\def \l {\alpha}
\newcommand{\eps}{\varepsilon}
\newcommand{\pl}{\partial}
\newcommand{\supp}{{\rm supp}{\hspace{.05cm}}}
\newcommand{\x}{\times}
\newcommand{\lag}{\langle}
\newcommand{\rag}{\rangle}

\newcommand\wrt{\,{\rm d}}

\title[]{The Schr\"odinger equation   in $L^p$ spaces
for  operators with \\[2pt] heat  kernel satisfying Poisson type  bounds}

\author{Peng Chen, Xuan Thinh Duong, Zhijie Fan, Ji Li, Lixin Yan}

 \address{Peng Chen, Department of Mathematics, Sun Yat-sen
University, Guangzhou, 510275, P.R. China}
\email{chenpeng3@mail.sysu.edu.cn}

 \address{Xuan Thinh Duong, Department of Mathematics, Macquarie University, NSW 2109, Australia}
\email{xuan.duong@mq.edu.au}

 \address{Zhijie Fan, Department of Mathematics, Sun Yat-sen
University, Guangzhou, 510275, P.R. China}
\email{fanzhj3@mail2.sysu.edu.cn}

\address{Department of Mathematics, Macquarie University, NSW, 2109, Australia}
\email{ji.li@mq.edu.au}

\address{Department of Mathematics, Sun Yat-sen (Zhongshan) University, Guangzhou, 510275, P.R. China
 and Department of Mathematics, Macquarie University, NSW 2109, Australia}
\email{mcsylx@mail.sysu.edu.cn}

  \date{\today}

\begin{abstract}
Let $L$ be a non-negative self-adjoint operator acting on $L^2(X)$
where $X$ is a space of homogeneous type with a dimension $n$.
In this paper, we study  sharp endpoint $L^p$-Sobolev  estimates  for the solution of
the initial value  problem for the Schr\"odinger equation\,  $i  \partial_t u +  L u=0 $  and show that
  for all $f\in L^p(X), 1<p<\infty,$
 \begin{eqnarray*}
 \left\| e^{itL} (I+L)^{-{\sigma n}} f\right\|_{p} \leq  C(1+|t|)^{\sigma n} \|f\|_{p},
 \ \ \ t\in{\mathbb R}, \ \ \ \sigma\geq  \big|{1\over  2}-{1\over  p}\big|,
\end{eqnarray*}
where
the semigroup $e^{-tL}$ generated by $L$  satisfies  a Poisson type upper bound.
\end{abstract}

\maketitle


\section{Introduction}
\subsection{Background and main result}
We will consider the initial value  problem for the Schr\"odinger equation
\begin{eqnarray} \label{Sch}
\left\{
\begin{array}{ll}
    i { \partial_t u (x,t)  } +  Lu(x, t)=0 \quad x\in \mathbb R^n,\ \  t>0,  \\[5pt]
 u(x, 0)=f(x),
\end{array}
\right.
\end{eqnarray}
where $L $ is
a non-negative self-adjoint operator on $L^2(\mathbb R^n)$
with the heat kernel $h_t(x,y)$ of the semigroup $e^{-tL}$ satisfying the Poisson type upper bound: there exist positive
 constants $C$ and $m,a>0$ such that for all $x,y\in \mathbb R^n$, $t>0$,
\begin{eqnarray}\label{heat kernel bound}
|h_{t} (x,y)|\leq  C t^{-\frac{n}{m}}\left(1+\frac{|x-y|}{t^{1/m}}\right)^{-n-a}.
\end{eqnarray}
The aim of this paper is to  focus on the $L^p$ estimate for the solution of this Schr\"odinger
equation with the sharp index. Our main result in this paper is the following.

\begin{theorem}\label{main0}
Suppose that $L $ is
a non-negative self-adjoint operator on $L^2(\mathbb R^n)$ satisfying the above heat kernel upper bound \eqref{heat kernel bound}
with  $a>\big[\frac{n}{2}\big]+1$.
Then for any $p\in (1,\infty)$, there exists a
 constant $C_{p}>0$, independent of $t$ and $f$, such that
\begin{align}\label{Lp main0}
\|e^{itL}(I+L)^{-\sigma n}f\|_{L^{p}(\mathbb R^n)}\leq C_{p}(1+|t|)^{\sigma n}\|f\|_{L^{p}(\mathbb R^n)}\quad {\rm for\ any\ }
 \sigma\geq \sigma_{p}:=\Big|{1\over2} -{1\over p}\Big|.
\end{align}
\end{theorem}

We point out that when $L$ is the standard Laplacian on $\mathbb R^n$,
the sharp endpoint $L^p$-Sobolev estimate was first studied  by Miyachi \cite{M1,M2}. Later it has been extensively studied in
different types of Schr\"odinger equation in $\mathbb R^n$, see for example \cite{Br, Hi1, Hi2, Hi3}, where the key tool is Fourier transform.
Recently, Jensen and Nakamura \cite{JN94,Jensen} developed an idea on obtaining the $L^p$ estimates of the Sch\"odinger equation
 by using the commutator method. D'Ancola and Nicola \cite{Dancola} applied this method to prove  the uniform local $L^p$
 estimates for the solution of this Schr\"odinger problem,  that is, they obtained the sharp estimate of $\|e^{itL} \varphi(L) f\|_{L^p}$ with optimal growth in time and optimal regularity loss, where $\varphi$ is a compactly supported function in $C^\infty(\mathbb R)$ and the heat kernel upper bound of the
 operator $L$ can be relaxed from exponential decay to polynomial decay. See also \cite{BDN} for the extension of \cite{Dancola}
  to space of homogeneous type. However, their results cannot be applied to proving \eqref{Lp main0}.
In \cite{CDLY1}, the first, second, fourth and fifth authors  of this paper studied the Schr\"odinger equation
  for nonnegative
self-adjoint operator $L$ whose heat kernel satisfies  the Gaussian upper bound, and obtained the sharp estimate \eqref{Lp main0}.
The method in \cite{CDLY1} is to use the sharp maximal function estimate, which depends heavily on the important tool of function
calculus studied by Christ, Hebisch, McIntosh, Duong et., as well as Blunck and Kunstmann and so on, where the exponential decay
of the heat kernel  plays an essential role.

Our main result Theorem \ref{main0} reduces the kernel upper bound to pointwise Poisson bound, and gives the full range of the
 sharp $L^p$ estimate for the solution to the initial value problem for a Schr\"odinger equation.
To obtain this, we develop several new techniques comparing to the previous closely related results \cite{ CDLY1,Dancola,M1, M2}.
 We now explain these in the next subsection in details.

\subsection{Assumptions, framework and new techniques of the proof}

To show Theorem \ref{main0}, we will work on a more general setting in order to cover many important examples that are in the scope  of $\mathbb R^n$.
Throughout this paper, we assume that $X$ is a metric space, with a distance $d$ and  a nonnegative, Borel, doubling measure $\mu$  on $X$ satisfying $\mu(X)=+\infty$.

To be more precise, we first recall the basic setup for the metric space $X$.  Let $B(x,r)=\{y\in X:d(x,y)<r\}$ be the open ball
with center $x\in X$ and radius $r>0$ and let $V(x,r)=\mu(B(x,r))$, the volume of $B(x,r)$. We say that $(X,d,\mu)$ satisfies the doubling
property (see Chapter 3, \cite{Coifman}) if there exists a constant $C>0$ such that
\begin{align}\label{doubling0}
V(x,2r)\leq CV(x,r),\ \ \forall r>0,\ x\in X.
\end{align}
Then the doubling property implies the following strong homogeneity inequality,
\begin{align}\label{doubling}
V(x,\lambda r)\leq C\lambda^{n}V(x,r)
\end{align}
for some $C, n>0$ uniformly for all $\lambda\geq 1$ and $x\in X$.
In the Euclidean space with Lebesgue measure, $n$ is
the dimension of the space.
In our results, the critical index is always expressed in terms of the
homogeneous dimension $n$.
Note also that there  exist $c$ and $D, 0\leq D\leq n$ so that
\begin{equation}
V(y,r)\leq c\bigg( 1+{d(x,y)\over r}\bigg )^D V(x,r)
\label{e1.3}
\end{equation}
uniformly for all $x,y\in {  X}$ and $r>0$. Indeed, the
property (\ref{e1.3}) with
$D=n$ is a direct consequence of the triangle inequality with respect to the metric
$d$ and the strong homogeneity property. In the cases of Euclidean spaces
${\mathbb R}^n$ and Lie groups of polynomial growth, $D$ can be
chosen to be $0$.

Suppose that $L$ is  a non-negative self-adjoint operator on $L^2(X)$, one can formally define an Schr\"odinger group $e^{itL}$,
using the spectral theory for $L.$ Assume  that
 $L$ has a  spectral resolution:
$$
Lf=\int_0^{\infty}\lambda dE_L(\lambda) f, \ \ \ \ f\in L^2(X),
$$
where  $E_L(\lambda)$ is the projection-valued measure supported on the spectrum of $L$.
Then the  operator   $   e^{itL}$ is defined by
 \begin{equation}\label{e111}
 e^{itL}f =   \int_0^{\infty}    e^{it\lambda}dE_L(\lambda) f
   \end{equation}
 for $f\in L^2(X)$, and forms   the Schr\"odinger group.
   By the spectral theorem (\cite{Mc}),  the operator   $  e^{itL}$  is  continuous on $L^2(X)$.
 Assuming $f\in L^2(X)$, $u(x,t)=e^{itL}f$ solves  the following
  initial value  problem for the Schr\"odinger equation
\begin{eqnarray*}
\left\{
\begin{array}{ll}
    i { \partial_t u (x,t)  } +  Lu(x, t)=0 \quad x\in X,\ \  t>0, \\[5pt]
 u(x, 0)=f(x).
\end{array}
\right.
\end{eqnarray*}
 It is interesting to  investigate $L^p$-mapping properties  for
 the Schr\"odinger group $ e^{itL}$   on $L^p(X)$ for some $p, 1\leq p\leq \infty.$

We now introduce the polynomial off-diagonal estimate $({\rm PEV}_{p_{0},2}^{a,m})$ for $L$, which, when moving back to $\mathbb R^n$, is weaker than
the Poisson type decay \eqref{heat kernel bound}.

\begin{definition}\label{def PEV}
Let $L$ be a non-negative self-adjoint operator on $L^2(X)$ and $m> 0$, $a>0$. We say that
the semigroup $e^{-tL}$ generated by $L,$ satisfies the property $({\rm PEV}_{p_{0},2}^{a,m})$
  if there exists a constant $C>0$ such that for all  $x,y\in X$, $t>0$,
\begin{align*}\label{hh}
\tag{${\rm PEV}_{p_{0},2}^{a,m}$}\|\chi_{B(x,t^{1/m})}e^{-t L}V_{t^{1/m}}^{\sigma_{p_{0}}}
\chi_{B(y,t^{1/m})}\|_{p_{0}\rightarrow 2}\leq C\left(1+\frac{d(x,y)}{t^{1/m}}\right)^{-n-a},
\end{align*}
where $V_{t^{1/m}}^{\sigma_{p_{0}}}$  is  a pointwise multiplier operator defined by $V_{t^{1/m}}^{\sigma_{p_{0}}}f(x):=V(x,t^{1/m})^{\sigma_{p_{0}}}f(x)$.
\end{definition}

Note that if the semigroup $e^{-t L}$ has integral kernel $h_{t} (x,y)$
satisfying the following Poisson type upper bound:
\begin{eqnarray}\label{hh}
|h_{t} (x,y)|\leq  C V(x, t^{1/m})^{-1}\left(1+\frac{d(x,y)}{t^{1/m}}\right)^{-n-a}
\end{eqnarray}
 for all $t>0$ and all $x, y\in X$, then $L$ satisfies the property \eqref{inter} with $p=1$.
Based on the condition $({\rm PEV}_{p_{0},2}^{a,m})$, our main result is the following, which covers Theorem \ref{main0} when we
 restrict our $(X,d,\mu)$ to the setting of $\mathbb R^n$.

\begin{theorem}\label{main}
Suppose that $(X,d,\mu)$ is a space of homogeneous type with a dimension $n$ and that $L$ satisfies the property
$({\rm PEV}_{p_{0},2}^{\kappa,m})$ for some $1\leq p_{0}<2$, $m>0$ and $\kappa>\kappa_{0}:=\big[\frac{n}{2}\big]+1$. Then for any
$p\in (p_{0},p_{0}^{\prime})$, there exist a constant $C_{p}>0$, independent of $t$ and $f$, such that
\begin{align}\label{LL}
\|e^{itL}(I+L)^{-\sigma_{p}n}f\|_{L^{p}(X)}\leq C_{p}(1+|t|)^{\sigma_{p}n}\|f\|_{L^{p}(X)}.
\end{align}

 As a consequence, this estimate \eqref{LL} holds for all $1<p<\infty$ when the heat kernel of $L$
 satisfies a Poisson type upper bound \eqref{hh}  for $a>\kappa_{0}$.
\end{theorem}

\begin{remark}
(i). One can also consider the $L^{p}$ boundedness of Schr\"{o}dinger groups for the self-adjoint operator $L$ having a lower bound such that $L+M_{0}\geq 0$ as in \cite{BDN,Simon}, where $M_{0}\geq 0$. In this setting, if $L$ satisfies the off-diagonal estimate
\begin{align*}
\|\chi_{B(x,t^{1/m})}e^{-t L}V_{t^{1/m}}^{\sigma_{p_{0}}}
\chi_{B(y,t^{1/m})}\|_{p_{0}\rightarrow 2}\leq Ce^{tM_{0}}\left(1+\frac{d(x,y)}{t^{1/m}}\right)^{-n-\kappa},
\end{align*}
for some $1\leq p_{0}<2$, $m>0$ and $\kappa>\kappa_{0}:=\big[\frac{n}{2}\big]+1$, then for any $M>M_{0}$ and $p\in (p_{0},p_{0}^{\prime})$, by considering the non-negative self-adjoint operator $L_{M_{0}}:=L+M_{0}$ as in \cite{BDN}, we have
\begin{align*}
\|e^{itL}(M+L)^{-\sigma_{p}n}f\|_{L^{p}(X)}\leq C_{p}(1+|t|)^{\sigma_{p}n}\|f\|_{L^{p}(X)}.
\end{align*}

(ii). Under the assumption of $({\rm PEV}_{p_{0},2}^{\kappa,m})$, we can see by a similar argument in the proof of Theorem~\ref{main} (see also \cite[Theorem 4.1]{CDLY1}) that for any $p\in (p_{0},p_{0}^{\prime})$,
$
\|e^{itL}(I+tL)^{-\sigma_{p}}\|_{p\rightarrow p}\leq C.
$
This, together with Theorem \ref{spectralmultiplier}, implies the $L^{p}$ boundedness of local Schr\"{o}dinger flow $\|e^{itL} \varphi(L) f\|_{L^p}$, where $\varphi$ is a compactly supported function in $C^\infty(\mathbb R_+)$. That is,
\begin{align*}
\|e^{itL}\varphi(2^{-k}L)\|_{p\rightarrow p}\leq C_{p}(1+2^{k}|t|)^{\sigma_{p}}.
\end{align*}
This result is already proved in~\cite{Dancola} on $\R^n$ and in~\cite{BDN} on homogeneous spaces.
\end{remark}

To show Theorem \ref{main}, inspired by Miyachi's work \cite{M1,M2}, we first study a Hardy space $H_{L}^{q}(X)$ ($\frac{2n}{2\kappa_0+n}<q\leq1$) defined via a suitable Littlewood--Paley area function and then
we show that such a Hardy space allows molecular decomposition and complex interpolation. Hence, this theorem can be reduced to proving
the following $H^q(X)$ boundedness of $e^{itL}(I+L)^{-\sigma_{p}n}$ for $\frac{2n}{2\kappa_0+n}<q<1$.
\begin{proposition}\label{tool}
Suppose that $(X,d,\mu)$ is a space of homogeneous type with a dimension $n$ and that $L$ satisfies the property $({\rm PEV}_{2,2}^{\kappa,m})$
for some $m>0$ and $\kappa>\kappa_{0}:= \big[\frac{n}{2}\big]+1$. Then for any $\frac{2n}{2\kappa_0+n}<q<1$, there exists a constant $C>0$, independent of $t$ and $f$, such that
\begin{align*}
\|e^{itL}(I+L)^{-\sigma_q n}f\|_{H_{L}^{q}(X)}\leq C(1+|t|)^{\sigma_qn}\|f\|_{H_{L}^{q}(X)}.
\end{align*}
\end{proposition}

To obtain the above proposition,
 our main approach is to establish the following new off-diagonal estimates for
the oscillatory spectral multiplier $e^{itL}F(L)$.

\begin{proposition}\label{p2estimate22}
There exist constants $C, c_{0}>0$ such that for any ball $B\subset X$ with radius $r_{B}$ and for any $\lambda>0$, $j\geq6 $,
\begin{align*}
\big\|\chi_{U_{j}(B)}e^{itL}F(L)\chi_{B}f\big\|_{2}\leq C2^{-j\kappa_{0}}(\sqrt[m]{\lambda}r_{B})^{-\kappa_{0}+\frac{n}{2}}(1+\lambda|t|)^{\kappa_{0}}
(\sqrt[m]{\lambda}r)^{-c_{0}}\|\delta_{\lambda}F\|_{C^{\kappa_{0}+1}} \|f\|_{2}
\end{align*}
for all Borel functions $F$ such that supp$F\subseteq [-\lambda,\lambda]$, where $r= {\rm min}\{r_{B},\lambda^{-1/m}\}$.
\end{proposition}

We would like to point out that:
 \begin{itemize}
\item[ (1)] Under the assumption of the Gaussian upper bounds $({\rm GE}_{m})$ of an operator $L$, the Phragm\'{e}n-Lindel\"{o}f Theorem
is the central tool for the optimal extension of the $L^{2}-L^{2}$ off-diagonal estimates of the real semigroup $e^{-\tau L}$ for real
values $\tau\in\mathbb{R}_{+}$ to that for complex values $z\in \mathbb{C}_{+}$. By using complex semigroup $e^{-(i\tau-1)\lambda^{-1}L}$
to represent spectral multiplier $F(L)$, the authors in \cite{CDLY2} (see also \cite{Fan}) showed that for any $s\geq 0$, there exists a constant $C>0$ such that
for any $j\geq 2$,
\begin{align*}
\|\chi_{U_{j}(B)}F(L)\chi_{B}\|_{2\rightarrow 2}\leq C(\sqrt[m]{\lambda}2^{j}r_{B})^{-s}\|F(\lambda\cdot)\|_{B^{s}}
\end{align*}
for all balls $B\subset X$, and all Borel functions $F$ such that supp$F\subset [-\lambda,\lambda]$. It should be noted that due to the
appearance of the Besov norm $\|\cdot\|_{B^{s}}$, this inequality can be used to obtain a sharp $L^{2}-L^{2}$ off-diagonal estimate of
compactly supported spectral multiplier with an oscillatory term $e^{itL}$, which plays a crucial role in showing the boundedness on
$H_{L}^{1}$ for Schr\"{o}dinger groups.

However, under a mild decay assumption $({\rm PEV}_{p_{0},2}^{\kappa,m})$ or the Gaussian upper bounds $({\rm GE}_{m})$ for $m< 2$,
Phragm\'{e}n-Lindel\"{o}f Theorem cannot be applied to obtaining a suitable substitution for $e^{-zL}$ in a similar way.

\item[(2)] Without $({\rm GE}_{m})$, inspired by Davies's work (\cite{Da}), Duong and Robinson (\cite{DR}) used Poisson formula for
 subharmonic function to obtain the off-diagonal decay of the heat kernel $K_{e^{-zL}}(x,y)$ for the special case $X=\mathbb{R}^{n}$ and $p_{0}=1$.

However, such an estimate  is not enough to obtain the required off-diagonal estimates of $F(L)$, since the off-diagonal decay becomes
slower and disappears gradually, as the angle ${\rm arg} z$ increases from 0 to $\frac{\pi}{2}$.
\end{itemize}

Therefore, the previous methods cannot be expected to obtain off-diagonal estimates of $F(L)$.

To overcome this main difficulty, we develop a completely different method to obtain a suitable replacement by means of amalgam
blocks and commutators. To illustrate that, we split the whole process into three steps:

\smallskip
$\bullet$
\textbf{Step 1:}
Inspired by \cite{frame}, by representing the spectral multiplier $F(L)$ as $E(e^{-\frac{L}{\lambda}})e^{-\frac{L}{\lambda}}$ and then
studying the off-diagonal properties of the operators $E(e^{-\frac{L}{\lambda}})$ and $e^{-\frac{L}{\lambda}}$ separately, we obtain
 $L^{2}-L^{2}$ amalgam-type off-diagonal estimates of $F(L)$ (see Lemma \ref{pro2});
\smallskip

$\bullet$
\textbf{Step 2:}
Define $R_{\lambda}:=(I+\lambda^{-1}L)^{-1}$. Inspired by \cite{Dancola}, by representing the oscillatory spectral multiplier $e^{itL}F(L)$
as $R_{\lambda}^{2^{\kappa+1}-2}e^{itL}\widetilde{F}(L)$, where $\widetilde{F}$ is a compactly supported Borel function satisfying
 supp$\widetilde{F}\subset [-\lambda,\lambda]$, and then studying the off-diagonal property of the operator
 $R_{\lambda}^{2^{\kappa+1}-2}e^{itL}$, we obtain $L^{2}-L^{2}$ amalgam-type off-diagonal estimates of $e^{itL}F(L)$ (see Lemma \ref{oscillatory2});
\smallskip

$\bullet$
\textbf{Step 3: }
By embedding the set $U_{j-1}(B)$ into a countable union of amalgam block with a suitable size (which is different
from  the one chosen  in \cite{frame} and \cite{Dancola}, since there are two parameters that we
 need to consider:  the radius of $B$ and the size of supp$F$), we obtain $L^{2}-L^{2}$ off-diagonal estimates of $e^{itL}F(L)$
 (see Proposition \ref{p2estimate22}).

\smallskip

\subsection{Applications}
Our results,   Theorems~\ref{main0} and  \ref{main} can be applied to all examples which are discussed in \cite{CCO,
CDLY1, frame, Dancola, DOS}.

We now provide two more particular examples:

\smallskip
1. Fractional Schr\"odinger operator  with potentials on $\mathbb R^n$.

\smallskip

Let $n\geq 1$ and $W_1, W_2$ be  locally integrable non-negative functions on
${\Bbb R}^n$.
Consider the fractional Schr\"odinger operator  with potentials  $W_1$ and $W_2$:
$$
L=(-\Delta + W_1)^{\beta} +W_2(x), \ \  \ \beta\in (0,1].
$$
The particular case $\beta = 1/2$ is often referred to the relativistic Schr\"odinger operator.
The operator $L$ is self-adjoint as an operator associated with a well defined closed quadratic form. By the classical subordination
formula (see for example, \cite[Section 5.4]{Gr}) together with the Feynman-Kac formula it follows that  the semigroup kernel
$h_t(x,y)$ associated to $e^{-tL}$ satisfies the estimate
\begin{eqnarray*}
	0\leq  h_t(x,y) \leq Ct^{-{n/2\beta}} \left(1+t^{-{1\over 2\beta}}|x-y|\right)^{-(n+2\beta)}
\end{eqnarray*}
for all $t>0$ and $x,y\in{\Bbb R}^n$.
Hence,  estimate  \eqref{heat kernel bound} holds  for   $m=2\beta$ and  $\alpha=2\beta$.
 If $n=1$ and $\beta > 1/2$,   then  we apply  Theorem~\ref{main0}   to
  obtain the sharp   $L^p$-Sobolev   estimate \eqref{Lp main0} for the Schr\"odinger equation \eqref{Sch}  for the operator $L$.

\smallskip
2. Sub-Laplacian on certain Carnot--Carath\'edory spaces developed by Nagel and Stein \cite{NS04}.

\smallskip
We  first recall
   the background of this setting \cite{NS04}.
Let $M$ be a connected smooth manifold
and $\{\mathbb{X}_1, \cdots, \mathbb{X}_k\}$ are $k$ given smooth
real vector fields on $M$ satisfying H\"{o}rmander
condition of order $m$, i. e., these vector fields together with
their commutators of order $\leq m$ span the tangent space to $M$ at each point.

It was shown in \cite{NS04} that
there is a pseudo-metric $d$ on $M$
such that $d(x,y)$ is $C^{\infty}$ on $M\times M \backslash \{{\rm
diagonal}\}$, and for $x\neq y$
\begin{eqnarray*}
|\partial_X^K\partial_Y^Ld(x,y)|\lesssim d(x,y)^{1-K-L}.
\end{eqnarray*}
Here $\partial_X^K$ are products of $K$ vector fields
$\{X_1,\cdots X_k\}$ acting as derivatives on the $x$ variable, and
$\partial_Y^L$ are corresponding $L$ vector fields acting on the
$y$ variable. There is also a doubling measure on $M$, which was given in each specific
example of $M$.

Consider the sub-Laplacian $\mathcal {L}$ on $M$ in
self-adjoint form, given by
\begin{eqnarray*}
\mathcal {L}=\sum_{j=1}^k \mathbb{X}_j^{*}\mathbb{X}_j.
\end{eqnarray*}
Here $(\mathbb{X}_j^{*}\varphi, \psi)=(\varphi, \mathbb{X}_j\psi)$,
where $(\varphi,\psi)=\int\limits_{M} \varphi(x)\bar{\psi}(x)dx
$, and $\varphi, \psi\in C_0^{\infty}(M)$, the space of $C^{\infty}$
functions on $M$ with compact support. In general,
$\mathbb{X}_j^{*}= -\mathbb{X}_j+a_j$, where $a_j\in C^{\infty}(M)$.
The solution of the following initial value problem for the heat
equation
\begin{eqnarray*}
{{\partial u} \over {\partial s}}(x,s)+      \mathcal {L}_x u(x,s)=0
\end{eqnarray*}
with $u(x,0)=f(x)$ is given by $u(x,s)=H_s(f)(x)$, where $H_s$ is
the operator given via the spectral theorem by $H_s=e^{-s\mathcal
{L}}$, and an appropriate self-adjoint extension of the non-negative
operator $\mathcal {L}$ initially defined on $C_0^{\infty}(M)$. Nagel and Stein
proved that for $f\in L^2(X)$,
\begin{eqnarray*}
H_s(f)(x)=\int_M H(s,x,y)f(y)d\mu(y)
\end{eqnarray*}
and the heat kernel $H(s,x,y)$ satisfy the following property (see Proposition 2.3.1
in \cite{NS04} and Theorem 2.3.1 in \cite{NS01a}):

For every integer $N\geq 0$,
\begin{eqnarray*}
{|H(s,x,y)|}
 \lesssim  \frac{\displaystyle 1 }{\displaystyle
V(x,d(x,y))+V(x,\sqrt{s}) +V(y,\sqrt{s}) } \bigg(\frac{\displaystyle
\sqrt{s} }{\displaystyle d(x,y)+\sqrt{s} } \bigg)^{N\over 2},
\end{eqnarray*}
which implies \eqref{hh} obviously. Hence, we can apply  Theorem~\ref{main0}   to
  obtain the sharp   $L^p$-Sobolev   estimate \eqref{Lp main0} for the Schr\"odinger equation \eqref{Sch}  for the operator $\mathcal {L}$.

\subsection{Notation and structure of the paper}
For $1\leq p \leq+\infty$, we denote the norm of a function $f\in L^{p}(X,d\mu)$ by $\|f\|_{p}$. If $T$ is a bounded linear operator
from $L^{p}(X,d\mu)$ to $L^{q}(X,d\mu)$, $1\leq p,q\leq+\infty$, we write $\|T\|_{p\rightarrow q}$ for the operator norm of $T$. The
indicator function of a subset $E\subseteq X$ is denoted by $\chi_{E}$. Besides, let $\mathcal{D}(T)$ be the domain of an operator $T$.
Throughout the paper,   $V_{s}$ is  a pointwise multiplier operator defined by $V_{s}f(x):=V(x,s)f(x)$ and
$\delta_{r}F$ is   the dilation of a function $F$, defined   by $\delta_{r}F(x):=F(rx)$. Recall that $n$
is the dimension of the space $X$, we will write
\begin{eqnarray}\label{xx}
  \kappa_{0}=\left[\frac{n}{2}\right]+1,\ \ \ {\rm and}\ \ \
  \sigma_{p}=\left|\frac{1}{p}-\frac{1}{2}\right|, \ \ \ 1\leq p\leq \infty .
\end{eqnarray}
   Also,  for any given ball $B$ in $X$  we set
\begin{eqnarray}\label{xx}
   U_{0}(B):=B \ \ \  {\rm and}\ \ \
   U_{j}(B):=2^{j}B\backslash2^{j-1} B,\ \ j=1, 2,\ldots.
\end{eqnarray}

This paper is organised as follows. In Section 2 we provide the preliminaries, including the fundamental properties of the off-diagonal
estimate $({\rm PEV}_{p_{0},2}^{a,m})$ and the Hardy space associated with the operator $L$
satisfying $({\rm PEV}_{p_{0},2}^{a,m})$. In Section 3, we develop new techniques on the off-diagonal estimates for the compactly supported
  spectral multipliers (Proposition \ref{p2estimate}) and the oscillatory compactly supported spectral multipliers (Proposition \ref{p2estimate22}).
   In Section 4 we prove our main result Theorem \ref{main}.
In the last section we provide some results  on the molecular decomposition and interpolation
of the Hardy space $H^p_L(X)$.

\section{Preliminaries}
\setcounter{equation}{0}
\subsection{Basic properties of $({\rm PEV}_{p_{0},2}^{a,m})$}
In this subsection, we recall some basic properties of $({\rm PEV}_{p_{0},2}^{a,m})$, which were essentially discussed in \cite{frame}.
Assume that $L$ satisfies the property (${\rm PEV}_{p_{0},2}^{a,m}$) for some $1\leq p_{0}<2$ and $m>0$.
 By H\"{o}lder's inequality,  the property (${\rm PEV}_{p_{0},2}^{a,m}$) implies
   that for any $p_{0}\leq p\leq 2$,
\begin{align}\label{inter}
 \tag{${\rm PEV}_{p,2}^{a,m}$}\|\chi_{B(x,\lambda^{1/m})}e^{-\lambda L}V_{\lambda^{1/m}}^{\sigma_{p}}\chi_{B(y,\lambda^{1/m})}f\|_{2}
 &\leq C\left(1+\frac{d(x,y)}{\lambda^{1/m}}\right)^{-n-a}\|V_{\lambda^{1/m}}^{\frac{1}{p}-\frac{1}{p_{0}}}\chi_{B(y,\lambda^{1/m})}f\|_{p_{0}}\\
 &\leq C\left(1+\frac{d(x,y)}{\lambda^{1/m}}\right)^{-n-a}\|f\|_{p}\nonumber.
\end{align}
In particular, we have
$$
 \|\chi_{B(x,\lambda^{1/m})}e^{-\lambda L} \chi_{B(y,\lambda^{1/m})} \|_{2\to 2}
 \leq C\left(1+\frac{d(x,y)}{\lambda^{1/m}}\right)^{-n-a}.
$$

\smallskip


%

Next, we divide $X$ into countable partitions with different size parameters. For every $r>0$, we choose a sequence
 $\{x_{i}\}_{i=1}^{\infty}\in X$ such that $d(x_{i},x_{j})>\frac{r}{2}$ for $i\neq j$ and $\sup\limits_{x\in X}\inf\limits_{i}d(x,x_{i})\leq \frac{r}{2}$.
  Such sequence exists since $X$ is separable. Set
\begin{align*}
D= \bigcup_{i\in \mathbb{N}}B(x_i,r/4).
\end{align*}
Then we define the amalgam block $Q_{i}(r)$ by the formula
\begin{align*}
Q_i(r) =    B(x_i,r/4)\bigcup\left[B\left(x_i,r/2\right)\setminus
\left(\bigcup_{j<i}B\left(x_j,r/2\right)\cup D\right)\right],
\end{align*}
so that $\{Q_{i}(r)\}_{i}$ is a countable partition of $X$. Namely,
\begin{align}\label{divide}
X=\mathop{\bigcup}\limits_{i\in\mathbb{N}}Q_{i}(r),
\end{align}
where $Q_{i}(r)\cap Q_{j}(r)=\emptyset$ if $i\neq j$. We say that $x_{i}$ is the center of $Q_{i}(r)$ and $r$
is the diameter of $Q_{i}(r)$. Such a partition of $X$ is not unique. For a fixed partition, let $\mathcal{I}_{r}$ be a index
set consisting of all $i\in\mathbb{N}$ such that
\begin{align*}
i\in \mathcal{I}_{r} \Leftrightarrow x_{i}{\rm \ is\ the\ center\ of\ }Q_{i}(r).
\end{align*}

Observe that $Q_{i}(r)\subset B(x_{i},r)$ and there exists a uniform constant $C>0$ depending only on the doubling
constant such that $\mu(Q_{i}(r))\geq C\mu(B_{i})$.

The following lemma is a simple consequence of the estimate (${\rm PEV}_{p,2}^{a,m}$).
\begin{lemma}\label{heatbound}
Let $p_{0}\leq p\leq 2$, then there exists a constant $C>0$ such that for any $\lambda>0$, $0<r\leq \lambda^{-1/m}$, and $\beta\in\mathcal{I}_r$,
\begin{align*}
\big\|e^{-\frac{L}{\lambda}}V_{\lambda^{-1/m}}^{\sigma_{p}}\chi_{Q_{\beta}(r)}f\big\|_{2}\leq C(\sqrt[m]{\lambda}r)^{-n/2}\|f\|_{p}.
\end{align*}
\end{lemma}
\begin{proof}
The definition of amalgam block allows us to decompose $X=\mathop{\cup}\limits_{\alpha\in\mathbb{N}}Q_{\alpha}(r)$, and then
\begin{align*}
\big\|e^{-\frac{L}{\lambda}}V_{\lambda^{-1/m}}^{\sigma_{p}}\chi_{Q_{\beta}(r)}f\big\|_{2}
&=\left(\sum_{\alpha\in \mathcal{I}_{r}}
\big\|\chi_{Q_{\alpha}(r)}e^{-\frac{L}{\lambda}}V_{\lambda^{-1/m}}^{\sigma_{p}}\chi_{Q_{\beta}(r)}f\big\|_{2}^{2}
\right)^{\frac{1}{2}}\\
&\leq C\left(\sum_{\alpha\in \mathcal{I}_{r}}
\left(1+\sqrt[m]{\lambda}d(x_{\alpha},x_{\beta})\right)^{-2n-2a}\right)^{\frac{1}{2}}\|f\|_{p}\\
&\leq C(\sqrt[m]{\lambda}r)^{-n/2}\|f\|_{p},
\end{align*}
This ends the proof of Lemma \ref{heatbound}.
\end{proof}
As a corollary, we can show the following $L^{p}-L^{2}$ spectral multiplier theorem.
\begin{lemma}\label{global00}
Let $p_{0}\leq p\leq 2$, then there exists a constant $C>0$ such that for any $\lambda>0$, $0<r\leq \lambda^{-1/m}$, and $\beta\in\mathcal{I}_r$,
\begin{align*}
\|F(L)V_{\lambda^{-1/m}}^{\sigma_{p}}\chi_{Q_{\beta}(r)}f\|_{2}\leq C(\sqrt[m]{\lambda}r)^{-n/2}\|\delta_{\lambda}F\|_{C^{1}}\|f\|_{p}
\end{align*}
for all Borel functions $F$ such that supp$F\subseteq [-\lambda,\lambda]$.
\end{lemma}
\begin{proof}
Let $E(\tau):=F(-\lambda{\rm log}\tau)\tau^{-1}$ so that $F(L)=E(e^{-\frac{L}{\lambda}})e^{-\frac{L}{\lambda}}$.
 Then it follows from the Fourier inversion formula that
\begin{align*}
E(e^{-\frac{L}{\lambda}})=\int_{-\infty}^{+\infty}e^{i\xi e^{-\frac{L}{\lambda}}}\hat{E}(\xi)d\xi.
\end{align*}

This, together with the spectral theorem and Lemma \ref{heatbound}, yields that
\begin{align*}
\|F(L)V_{\lambda^{-1/m}}^{\sigma_{p}}\chi_{Q_{\beta}(r)}f\|_{2}
&=\Big\|\int_{-\infty}^{+\infty}e^{i\xi e^{-\frac{L}{\lambda}}}e^{-\frac{L}{\lambda}}V_{\lambda^{-1/m}}^{\sigma_{p}}\chi_{Q_{\beta}(r)}f \hat{E}(\xi)d\xi\Big\|_{2}\\
&\leq\int_{-\infty}^{+\infty}\big\|e^{i\xi e^{-\frac{L}{\lambda}}}e^{-\frac{L}{\lambda}}V_{\lambda^{-1/m}}^{\sigma_{p}}\chi_{Q_{\beta}(r)}f\big\|_{2} |\hat{E}(\xi)|d\xi\\
&\leq \|\hat{E}\|_{1}\big\|e^{-\frac{L}{\lambda}}V_{\lambda^{-1/m}}^{\sigma_{p}}\chi_{Q_{\beta}(r)}f\big\|_{2}\\
&\leq C(\sqrt[m]{\lambda}r)^{-n/2}\|\delta_{\lambda}F\|_{C^{1}}\|f\|_{p},
\end{align*}
where in the last inequality we used the fact that supp$\delta_{\lambda}F\subset [-1,1]$ implies that
\begin{align}\label{Cnorm}
\int_{-\infty}^{+\infty}|\hat{E}(\xi)|d\xi\leq C\|E\|_{H^{1}}\leq C \|\delta_{\lambda}F\|_{H^{1}}\leq C \|\delta_{\lambda}F\|_{C^{1}}.
\end{align}
This finishes the proof of Lemma \ref{global00}.
\end{proof}

\subsection{Preliminaries on Hardy space $H_{L}^{p}(X)$}
There were numerous number of references (see for example, \cite{ADM, DL, DY2, DY3, HLMMY, KU}) studying the theory of
the Hardy spaces associated with certain operators, especially those ones satisfying the Gaussian upper bounds $({\rm GE_{m}})$.
At the beginning of this section, for any $\frac{2n}{2\kappa_0+n}< p\leq 2$, we will extend some basic definitions to the Hardy space $H_{L}^{p}$
associated with operators which only satisfy the estimate (${\rm PEV}_{2,2}^{\kappa,m}$) for some $m>0$ and $\kappa>\kappa_{0}$.

In this article, we will define the Hardy space $H_{L}^{p}(X)$ associated with operators in terms of Littlewood-Paley type area
function instead of the semigroup factor $t^{m}Le^{-t^{m}L}$. Thanks to the compactly supported property of the Littlewood-Paley
function and the off-diagonal estimate (\ref{p2estimate}), we can see below that in our setting, this definition is much more
convenient to obtain the $L^{p}$ boundedness for Schr\"{o}dinger groups.

Usually, we should define the $L^{2}$ adapted Hardy space $\mathcal{H}^{2}(X):=\overline{R(L)}$, that is, the closure of the range of $L$ in $L^{2}(X)$. Then $L^{2}(X)$ is the orthogonal sum of $\mathcal{H}^{2}(X)$ and the null space $N(L)$. However in our setting, that is, if $L$ satisfies the property
$({\rm PEV}_{p_{0},2}^{\kappa,m})$, then $N(L)=\{0\}$. Indeed, it can be verified that for any $f\in N(L)$ and $t>0$, we have $e^{-tL}f=f$ and therefore, by the condition $({\rm PEV}_{p_{0},2}^{\kappa,m})$,
$$
\|f\|_{p_{0}^{\prime}}\leq \lim_{t\rightarrow +\infty}\sum_{i\in\mathcal{I}_{t^{1/m}}}\|\chi_{B(x,t^{1/m})}e^{-tL}\chi_{B(x_{i},t^{1/m})}f\|_{p_{0}^{\prime}}\leq C\lim_{t\rightarrow +\infty}V(x,t^{1/m})^{-\sigma_{p_{0}}}\|f\|_{2}=0.
$$
So in our setting, $\mathcal{H}^2(X)=L^2(X)$.

Next, let $\phi$ be a non-negative cut-off function on $C_{c}^{\infty}(\mathbb{R})$ such that supp$\phi\subset (1/4,1)$ and define
\begin{align*}
S_{L,\phi}(f):=\left(\int_{0}^{\infty}\int_{d(x,y)<\tau^{1/m}}|\phi(\tau L)f(y)|^{2}\frac{d\mu(y)}{V(x,\tau^{1/m})}\frac{d\tau}{\tau}\right)^{\frac{1}{2}}.
\end{align*}
\begin{definition}
{\rm Suppose that $L$ satisfies the estimate (${\rm PEV}_{2,2}^{\kappa,m}$) for some $m>0$ and $\kappa>\kappa_{0}$.
For any $\frac{2n}{2\kappa_0+n}< p\leq 2$, we define the Hardy space $H_{L}^{p}(X)$ associated with $L$ be the completion of the space
\begin{align*}
\{f\in \mathcal{H}^{2}(X):\ S_{L,\phi}(f)\in L^{p}(X)\}
\end{align*}
endowed with the norm
\begin{align*}
\|f\|_{H_{L}^{p}(X)}=\|S_{L,\phi}f\|_{L^{p}(X)}.
\end{align*}}

\end{definition}
\begin{remark}
{\rm
It can be seen from Lemmas \ref{moleculardecomposition}, \ref{cominter}, \ref{Hp} and \cite[Theorem 3.15,
 Proposition 4.4]{DL} that under the assumption that $L$ satisfies Davies-Gaffney estimates $({\rm DG}_{m})$ for $m\geq 2$,
  that is, there exist constants $C,c>0$ such that for all $\lambda>0$, and all $x,y\in X$,
\begin{align}
\tag{${\rm DG}_{m}$}
\|\chi_{B(x,\lambda^{1/m})}e^{-\lambda L}\chi_{B(y,\lambda^{1/m})}\|_{2\rightarrow 2}\leq C{\rm exp}
\left(-c\left(\frac{d(x,y)}{\lambda^{1/m}}\right)^{\frac{m}{m-1}}\right),
\end{align}
then the Hardy space $H_{L}^{p}(X)$
coincides with the Hardy space in some previous references (see for example, \cite{ADM, DL, DY2, DY3, HLMMY, KU}).
}
\end{remark}

It is easy to show, by combining Lemmas \ref{moleculardecomposition} and \ref{cominter}, that this definition is independent of the choice of $\phi$.


An important tool to study the Hardy space $H_{L}^{q}(X)$ for $\frac{2n}{2\kappa_0+n}<q\leq 1$ is the molecule decomposition.
To illustrate that, we need the following definition of $(q,2,M,\epsilon)$-molecule associated with the operator $L$.
\begin{definition}
{\rm Given $\epsilon>0$, $\frac{2n}{2\kappa_0+n}<q\leq 1$ and $M\in \mathbb{N}$.
We say that a function $a(x)\in L^{2}(X)$ is called a $(q,2,M,\epsilon)$-molecule associated with $L$
for some $\epsilon>0$ and $m\in\mathbb{N}$ if there exists a function $b\in \mathcal{D}(L^{M})$ and a ball $B=B(x_{B},r_{B})$ such that

\medskip

{ (i)}\ $a=L^M b$;

\medskip

{ (ii)}\  For every $k=0,1,2,\dots,M$ and $j=0,1,2,\dots$,
$$
\|(r_{B}^{m}L)^{k}b\|_{L^2(U_j(B))}\leq 2^{-j\epsilon} r_{B}^{mM}
\mu(2^jB)^{-(\frac{1}{q}-\frac{1}{2})},
$$
where we denote $\mathcal{D}(T)$ be the domain of an operator $T$.
}
\end{definition}
To continue, we define the molecular Hardy spaces associated with $L$ as follows.
\begin{definition}
{\rm Given $\epsilon>0$, $\frac{2n}{2\kappa_0+n}<q\leq 1$ and $M\in \mathbb{N}$.
Let $\{\lambda_{j}\}_{j=0}^{\infty}\in\ell^{q}$, $\{a_{j}\}$ is a sequence of $(q,2,M,\epsilon)$-molecule,
we say that $f=\sum_{j}\lambda_{j}a_{j}$, where the sum converges in $L^{2}(X)$, is a molecular $(q,2,M,\epsilon)$-representation of $f$. Set
\begin{align*}
\mathbb{H}_{L,mol,M,\epsilon}^{q}(X)=\{f:  f\ {\rm has \ a \ molecular}\ (q,2,M,\epsilon)-{\rm representation}\},
\end{align*}
endowed with the norm
\begin{align*}
\|f\|_{H_{L,mol,M,\epsilon}^{q}(X)}=\inf\left\{\sum_{j=0}^{\infty}|\lambda_{j}|:
f=\sum_{j=0}^{\infty}\lambda_{j}m_{j}\ {\rm is \ a \ molecular }\ (q,2,M,\epsilon)-{\rm representation}\right\}.
\end{align*}
We define the space $H_{L,mol,M,\epsilon}^{q}(X)$ as the completion of $\mathbb{H}_{L,mol,M,\epsilon}^{q}(X)$ with respect to this norm.
}
\end{definition}

Now, we point out that the Hardy space $H_{L}^{q}(X)$ allows molecular decomposition and complex interpolation.
To be precise, we prove the following three lemmas.
\begin{lemma}\label{moleculardecomposition}
Given $\epsilon>0$, $\frac{2n}{2\kappa_0+n}<q\leq 1$ and $M\in \mathbb{N}$. Suppose that $(X,d,\mu)$ is a space of homogeneous type with a dimension $n$ and that $L$ satisfies the property
 (${\rm PEV}_{2,2}^{\kappa,m}$) for some $m>0$ and $\kappa> \kappa_{0}$. Let $M$ be a sufficient large constant,
 then we have $H_{L,mol,M,\epsilon}^{q}(X)=H_{L}^{q}(X)$ with equivalent norm, that is
\begin{align*}
\|f\|_{H_{L,mol,M,\epsilon}^{q}(X)}\approx \|f\|_{H_{L}^{q}(X)}.
\end{align*}
\end{lemma}

\begin{lemma}\label{cominter}
Suppose that $(X,d,\mu)$ is a space of homogeneous type with a dimension $n$ and that $L$ satisfies the property
(${\rm PEV}_{2,2}^{\kappa,m}$) for some $m>0$ and $\kappa>\kappa_{0}$. Suppose {$\frac{2n}{2\kappa_0+n}< p_{1}<p_{2}<\infty$,} $0<\theta<1$,
 and $1/p=(1-\theta)/p_{1}+\theta/p_{1}$. Then
\begin{align*}
[H_{L}^{p_{1}}(X),H_{L}^{p_{2}}(X)]_{\theta}=H_{L}^{p}(X),
\end{align*}
where we recall that $[\cdot,\cdot]$ denotes the complex interpolation bracket.
\end{lemma}

\begin{lemma}\label{Hp}
Suppose that $(X,d,\mu)$ is a space of homogeneous type with a dimension $n$ and that $L$ satisfies the property
 (${\rm PEV}_{p_{0},2}^{\kappa,m}$) for some $1\leq p_{0}<2$, $m>0$ and $\kappa>\kappa_{0}$. Then
  for any $p\in (p_{0},2]$, we have $H_{L}^{p}(X)=L^{p}(X)$ with equivalent norm, that is
\begin{align*}
\|f\|_{H_{L}^{p}(X)}\approx \|f\|_{L^{p}(X)}.
\end{align*}
\end{lemma}

Proof of Lemmas~\ref{moleculardecomposition}, \ref{cominter} and  \ref{Hp} will be given in Section 5.

\medskip

\section{Off-diagonal estimates}
\setcounter{equation}{0}

This section is devoted to presenting some off-diagonal estimates for different kinds of spectral multipliers.

\subsection{Off-diagonal estimates for compactly supported spectral multipliers}
In the following two subsections, we will develop a new method to obtain off-diagonal estimates for compactly supported
spectral multipliers with oscillatory terms by means of the theory of amalgam block and some techniques related to commutators
(see Lemma \ref{pro2} and Lemma \ref{oscillatory2}). These estimates play crucial roles in obtaining the sharp boundedness for
 Schr\"{o}dinger groups. In the first step, we show off-diagonal estimates for compactly supported spectral multipliers without
  oscillatory term. That is, we will show the following proposition.
\begin{proposition}\label{p2estimate}
Let $p_{0}\leq p\leq 2$, then there exist constants $C, c_{0}>0$ such that for any ball $B\subset X$ with radius $r_{B}$ and for any $\lambda>0$, $j\geq5 $,
\begin{align}\label{p2es}
\|\chi_{U_{j}(B)}F(L)\chi_{B}\|_{p\rightarrow 2}\leq C2^{-j\kappa_{0}}\mu(B)^{-\sigma_{p}}
(\sqrt[m]{\lambda}r_{B})^{-\kappa_{0}+\frac{n}{2}}(\sqrt[m]{\lambda}r)^{-c_{0}}\|\delta_{\lambda}F\|_{C^{\kappa_{0}+1}}
\end{align}
for all Borel functions $F$ such that supp$F\subseteq [-\lambda,\lambda]$, where $r= {\rm min}\{r_{B},\lambda^{-1/m}\}$.
\end{proposition}

To begin with, we set $\Gamma(j,0)=1$ for $j\geq 1$, and we define $\Gamma(j,k)$ inductively by $\Gamma(j,k+1)
:=\sum_{\ell=k}^{j-1}\Gamma(\ell,k)$ for $1\leq k\leq j-1$. For a given $r>0$, we denote commutator inductively by
\begin{align*}
&{\rm Ad}_{\ell,r}^{0}(T):=T;\\
&{\rm Ad}_{\ell,r}^{k}(T):={\rm Ad}_{\ell,r}^{k-1}\left(\frac{d(x_{\ell},\cdot)}{r}T-T\frac{d(x_{\ell},\cdot)}{r}\right), k\geq 1.
\end{align*}

To continue, we recall a known formula for commutator of a Lipschitz function and an operator $T$ on $L^{2}(X)$.
\begin{lemma}\label{du}
Let $T$ be a self-adjoint operator on $L^{2}(X)$. Assume that for some $\eta\in {\rm Lip}(X)$, the commutator $[\eta, T]$,
 defined by $[\eta, T]f:=\eta Tf-T(\eta f)$, satisfies that for any $f\in \mathcal{D}(T)$, $\eta f\in \mathcal{D}(T)$ and
 that $[\eta, T]$ is bounded on $L^{2}(X)$. Then the following formula holds:
\begin{align*}
[\eta,e^{itT}]f=it\int_{0}^{1}e^{istT}[\eta, T]e^{i(1-s)tT}fds,  \ \ \forall t\in\mathbb{R},\ \forall f\in L^{2}(X).
\end{align*}
\end{lemma}
\begin{proof}
The proof was given in \cite{MM}.
\end{proof}

Next, we recall a criterion for $L^{p}-L^{q}$ boundedness for linear operators.
\begin{lemma}\label{criterion}
Let $T$ be a linear operator and $1\leq p\leq q\leq\infty$. For every $r>0$,
\begin{align*}
\|T\|_{p\rightarrow q}\leq \sup\limits_{j}\sum_{i}\|\chi_{Q_{i}(r)}T\chi_{Q_{j}(r)}\|_{p\rightarrow q}
+\sup\limits_{i}\sum_{j}\|\chi_{Q_{i}(r)}T\chi_{Q_{j}(r)}\|_{p\rightarrow q},
\end{align*}
where $\{Q_{i}(r)\}_{i}$ is a countable partition of $X$.
\end{lemma}
\begin{proof}
The proof was given in \cite[Lemma 2.1]{frame}.
\end{proof}
A direct consequence of this criterion is the following estimate.
\begin{lemma}\label{ppp1}
There exists a constant $C>0$, such that for any $\lambda, r>0$, $\xi\in\mathbb{R}$, $\ell\in \mathcal{I}_r$ and $0\leq k\leq \kappa_{0}$,
\begin{align*}
\big\|{\rm Ad}_{\ell,r}^{ k}(e^{i\xi e^{-\frac{L}{\lambda}}})\big\|_{2\rightarrow 2}\leq C(1+|\xi|)^{ k}(\sqrt[m]{\lambda}r)^{- k}.
\end{align*}
\end{lemma}
\begin{proof}
By Lemma \ref{du}, for all $\ell>0$,
\begin{align*}
{\rm Ad}_{\ell,r}(e^{i\xi e^{-\frac{L}{\lambda}}})f=i\xi\int_{0}^{1}e^{is\xi e^{-\frac{L}{\lambda}}}{\rm Ad}_{\ell,r}
(e^{-\frac{L}{\lambda}})e^{i(1-s)\xi e^{-\frac{L}{\lambda}}}fds.
\end{align*}
Note that $e^{-\frac{L}{\lambda}}$ is a bounded operator on $L^{2}(X)$. Repeatedly, it can be reduced to showing that
\begin{align}\label{comheat}
\big\|{\rm Ad}_{\ell,r}^{ k}(e^{-\frac{L}{\lambda}})\big\|_{2\rightarrow 2}\leq C(\sqrt[m]{\lambda}r)^{- k}, 0\leq k\leq \kappa_{0}.
\end{align}

By Lemma \ref{criterion}, it suffices to show that for every $0\leq  k \leq \kappa_{0}$, $\ell>0$, $\lambda>0$, $r\leq \lambda^{-1/m}$,
\begin{align}\label{ss1}
\sup\limits_{\alpha\in \mathcal{I}_{\lambda^{-1/m}}}\sum_{\beta\in \mathcal{I}_{\lambda^{-1/m}}}
\big\|\chi_{Q_{\beta}(\lambda^{-1/m})}{\rm Ad}_{\ell,r}^{ k}(e^{-\frac{L}{\lambda}})\chi_{Q_{\alpha}(\lambda^{-1/m})}
\big\|_{2\rightarrow 2}\leq C(\sqrt[m]{\lambda}r)^{- k}.
\end{align}
To show (\ref{ss1}), we note that
\begin{align*}
\chi_{Q_{\beta}(\lambda^{-1/m})}{\rm Ad}_{\ell,r}^{ k}(e^{-\frac{L}{\lambda}})\chi_{Q_{\alpha}(\lambda^{-1/m})}
=&\sum_{\gamma_{1}+\gamma_{2}+\gamma_{3}= k}\frac{ k!}{\gamma_{1}!\gamma_{2}!\gamma_{3}!}
\left(\frac{d(x_{\beta},x_{\ell})}{r}-\frac{d(x_{\alpha},x_{\ell})}{r}\right)^{\gamma_{1}}
\left(\frac{d(\cdot,x_{\ell})}{r}-\frac{d(x_{\beta},x_{\ell})}{r}\right)^{\gamma_{2}}\times\\
&\chi_{Q_{\beta}(\lambda^{-1/m})}e^{-\frac{L}{\lambda}}\chi_{Q_{\alpha}(\lambda^{-1/m})}
\left(\frac{d(x_{\alpha},x_{\ell})}{r}-\frac{d(\cdot,x_{\ell})}{r}\right)^{\gamma_{3}}.
\end{align*}
Observe that
\smallskip

$\bullet$
$\Big|\frac{d(x_{\beta},x_{\ell})}{r}-\frac{d(x_{\alpha},x_{\ell})}{r}\Big|\leq \frac{d(x_{\beta},x_{\alpha})}{r}$;

\smallskip

$\bullet$
$\Big|\frac{d(x,x_{\ell})}{r}-\frac{d(x_{\beta},x_{\ell})}{r}\Big|\chi_{Q_{\beta}(\lambda^{-1/m})}(x)\leq (\sqrt[m]{\lambda}r)^{-1}$;

\smallskip

$\bullet$
$\Big|\frac{d(x_{\alpha},x_{\ell})}{r}-\frac{d(y,x_{\ell})}{r}\Big|\chi_{Q_{\alpha}(\lambda^{-1/m})}(y)\leq (\sqrt[m]{\lambda}r)^{-1}$.

These, in combination with the estimate (\ref{inter}) with $a=\kappa$, yields
\begin{align*}
&\sum_{\beta\in \mathcal{I}_{\lambda^{-1/m}}}\big\|\chi_{Q_{\beta}(\lambda^{-1/m})}{\rm Ad}_{\ell,r}^{ k}
(e^{-\frac{L}{\lambda}})\chi_{Q_{\alpha}(\lambda^{-1/m})}\big\|_{2\rightarrow 2}\\
\leq & C\sum_{\gamma_{1}+\gamma_{2}+\gamma_{3}= k}\sum_{\beta\in \mathcal{I}_{\lambda^{-1/m}}}
\Big(\frac{d(x_{\beta},x_{\alpha})}{r}\Big)^{\gamma_{1}}(\sqrt[m]{\lambda}r)^{-\gamma_{2}-\gamma_{3}}
\big\|\chi_{Q_{\beta}(\lambda^{-1/m})}e^{-\frac{L}{\lambda}}\chi_{Q_{\alpha}(\lambda^{-1/m})}\big\|_{2\rightarrow 2}\\
\leq & C(\sqrt[m]{\lambda}r)^{- k}\sum_{\beta\in \mathcal{I}_{\lambda^{-1/m}}}
(1+\sqrt[m]{\lambda}d(x_{\beta},x_{\alpha}))^{\kappa_{0}}\big\|\chi_{Q_{\beta}(\lambda^{-1/m})}
e^{-\frac{L}{\lambda}}\chi_{Q_{\alpha}(\lambda^{-1/m})}\big\|_{2\rightarrow 2}\\
\leq & C(\sqrt[m]{\lambda}r)^{- k}
\end{align*}
for some constant $C>0$ independent of $\alpha$.

Hence, (\ref{ss1}) is proved.
\end{proof}

The most technical lemma in this subsection is the following amalgam type off-diagonal estimate.{
\begin{lemma}\label{pro2}
Let $p_{0}\leq p\leq 2$, then there exist constants $C, c_{0}>0$ such that for any ball
$B\subset X$ with radius $r_{B}$ and for any $\lambda>0$, $j\geq 3 $, $r= {\rm min}\{r_{B},\lambda^{-1/m}\}$,
\begin{align*}
\left(\sum_{\substack{\alpha\in \mathcal{I}_{r}\\
d(x_{\alpha},x_{B})\geq 2^{j}r_B}}\|\chi_{Q_{\alpha}(r)}F(L)\chi_{ B}f\|_{2}^{2}\right)^{\frac{1}{2}}
\leq C\mu(B)^{-\sigma_{p}}2^{-j\kappa_{0}}(\sqrt[m]{\lambda}r_{B})^{-\kappa_{0}+\frac{n}{2}}
(\sqrt[m]{\lambda}r)^{-c_{0}}\|\delta_{\lambda}F\|_{C^{\kappa_{0}+1}}\|f\|_{p}
\end{align*}
for all Borel functions $F$ such that supp$F\subseteq [-\lambda,\lambda]$.
\end{lemma}
\begin{proof}
Let $E(\tau):=F(-\lambda{\rm log}\tau)\tau^{-1}$ so that $F(L)=E(e^{-\frac{L}{\lambda}})e^{-\frac{L}{\lambda}}$. Then
\begin{align*}
&\left(\sum_{\substack{\alpha\in \mathcal{I}_{r}\\
d(x_{\alpha},x_{B})\geq 2^{j}r_B}}\|\chi_{Q_{\alpha}(r)}F(L)\chi_{ B}f\|_{2}^{2}\right)^{\frac{1}{2}}\\
\leq &\left(\sum_{\substack{\alpha\in \mathcal{I}_{r}\\
d(x_{\alpha},x_{B})\geq 2^{j}r_B}}\big\|\sum_{\substack{\gamma\in \mathcal{I}_{r}\\
\ d(x_{\gamma},x_{B})\leq 2^{j-1}r_B}}\chi_{Q_{\alpha}(r)}E(e^{-\frac{L}{\lambda}})
\chi_{Q_{\gamma}(r)}e^{-\frac{L}{\lambda}}\chi_{B}f\big\|_{2}^{2}\right)^{\frac{1}{2}}\\
+&\left(\sum_{\substack{\alpha\in \mathcal{I}_{r}\\
d(x_{\alpha},x_{B})\geq 2^{j}r_B}}\big\|\sum_{\substack{\gamma\in \mathcal{I}_{r}\\
\ d(x_{\gamma},x_{B})> 2^{j-1}r_B}}\chi_{Q_{\alpha}(r)}E(e^{-\frac{L}{\lambda}})\chi_{Q_{\gamma}(r)}e^{-\frac{L}{\lambda}}
\chi_{B}f\big\|_{2}^{2}\right)^{\frac{1}{2}}\\
=:& {\rm I} + {\rm II}.
\end{align*}

To estimate the first term ${\rm I}$, we note that if $x\in Q_{\alpha}(r)$, then
\begin{align*}
\frac{d(x,x_{\gamma})}{r}\geq \frac{d(x_{\alpha},x_{\gamma})}{r}-\frac{d(x,x_{\alpha})}{r}\geq
\frac{d(x_{\alpha},x_{\gamma})}{r}-1\geq \frac{1}{2}\frac{d(x_{\alpha},x_{\gamma})}{r}.
\end{align*}
Besides, since $d(x_{\alpha},x_{B})\geq 2^{j}r_B$ and $d(x_{\gamma},x_{B})\leq 2^{j-1}r_B$,
\begin{align*}
\frac{d(x_{\alpha},x_{\gamma})}{r}\geq \frac{d(x_\alpha,x_B)}{r}-\frac{d(x_\gamma,x_B)}{r}\geq \frac{2^{j-1}r_B}{r}.
\end{align*}
Hence,
\begin{align*}
{\rm I}&\leq \sum_{\substack{\gamma\in \mathcal{I}_{r}\\
\ d(x_{\gamma},x_{B})\leq 2^{j-1}r_B}}\left(\sum_{\substack{\alpha\in \mathcal{I}_{r}\\
d(x_{\alpha},x_{B})\geq 2^{j}r_B}}\big\|\chi_{Q_{\alpha}(r)}E(e^{-\frac{L}{\lambda}})
\chi_{Q_{\gamma}(r)}e^{-\frac{L}{\lambda}}\chi_{B}f\big\|_{2}^{2}\right)^{\frac{1}{2}}\\
&\leq \sum_{\gamma\in \mathcal{I}_{r}}\left(\sum_{\substack{\alpha\in \mathcal{I}_{r}\\
d(x_{\alpha},x_{\gamma})\geq 2^{j-1}r_B}}\big\|\chi_{Q_{\alpha}(r)}E(e^{-\frac{L}{\lambda}})
\chi_{Q_{\gamma}(r)}e^{-\frac{L}{\lambda}}\chi_{B}f\big\|_{2}^{2}\right)^{\frac{1}{2}}\\
&\leq C\sum_{\gamma\in \mathcal{I}_{r}}\left(\sum_{\substack{\alpha\in
\mathcal{I}_{r}\\ \ d(x_{\alpha},x_{\gamma})\geq 2^{j-1}r_{B}}}
\left(\frac{d(x_{\alpha},x_{\gamma})}{r}\right)^{-2\kappa_{0}}
\left\|\left(\frac{d(\cdot,x_{\gamma})}{r}\right)^{\kappa_{0}}\chi_{Q_{\alpha}(r)}E(e^{-\frac{L}{\lambda}})
\chi_{Q_{\gamma}(r)}e^{-\frac{L}{\lambda}}\chi_{B}f\right\|_{2}^{2}\right)^{\frac{1}{2}}\\
&\leq C\left(\frac{2^{j}r_B}{r}\right)^{-\kappa_0}\sum_{\gamma\in \mathcal{I}_{r}}\left(\sum_{\alpha\in
\mathcal{I}_{r}}
\left\|\left(\frac{d(\cdot,x_{\gamma})}{r}\right)^{\kappa_{0}}\chi_{Q_{\alpha}(r)}E(e^{-\frac{L}{\lambda}})
\chi_{Q_{\gamma}(r)}e^{-\frac{L}{\lambda}}\chi_{B}f\right\|_{2}^{2}\right)^{\frac{1}{2}}\\
&\leq C\left(\frac{2^{j}r_B}{r}\right)^{-\kappa_0}\sum_{\gamma\in \mathcal{I}_{r}}
\left\|\left(\frac{d(\cdot,x_{\gamma})}{r}\right)^{\kappa_{0}}E(e^{-\frac{L}{\lambda}})
\chi_{Q_{\gamma}(r)}e^{-\frac{L}{\lambda}}\chi_{B}f\right\|_{2}.
\end{align*}
Combining this estimate with the Fourier inversion formula
\begin{align}\label{fourier}
E(e^{-\frac{L}{\lambda}})=\int_{-\infty}^{+\infty}e^{i\xi e^{-\frac{L}{\lambda}}}\hat{E}(\xi)d\xi,
\end{align}
we obtain that
\begin{align*}
{\rm I} &\leq C\left(\frac{2^{j}r_B}{r}\right)^{-\kappa_0}\sum_{\gamma\in \mathcal{I}_{r}}
\left\|\int_{-\infty}^{+\infty}\left(\frac{d(\cdot,x_{\gamma})}{r}\right)^{\kappa_{0}}e^{i\xi e^{-\frac{L}{\lambda}}}
\chi_{Q_{\gamma}(r)}e^{-\frac{L}{\lambda}}\chi_{B}f \hat{E}(\xi)d\xi\right\|_{2}\\
&\leq C\left(\frac{2^{j}r_B}{r}\right)^{-\kappa_0}\sum_{\gamma\in \mathcal{I}_{r}}
\int_{-\infty}^{+\infty}\left\|\left(\frac{d(\cdot,x_{\gamma})}{r}\right)^{\kappa_{0}}e^{i\xi e^{-\frac{L}{\lambda}}}
\chi_{Q_{\gamma}(r)}e^{-\frac{L}{\lambda}}\chi_{B}f\right\|_{2} |\hat{E}(\xi)|d\xi.
\end{align*}
To continue, applying the following formula for commutators(see Lemma 3.1, \cite{Jensen}):
\begin{align*}
\bigg(\frac{d(\cdot,x_{\gamma})}{r}\bigg)^{\kappa_{0}}e^{i\xi e^{-\frac{L}{\lambda}}}
=\sum_{ k=0}^{\kappa_{0}}\Gamma(\kappa_{0}, k){\rm Ad}_{\gamma,r}^{ k}(e^{i\xi e^{-\frac{L}{\lambda}}})\bigg(\frac{d(\cdot,x_{\gamma})}{r}\bigg)^{\kappa_{0}- k},
\end{align*}
we have
\begin{align}\label{commutator01}
&\left\|\bigg(\frac{d(\cdot,x_{\gamma})}{r}\bigg)^{\kappa_{0}}e^{i\xi e^{-\frac{L}{\lambda}}}
\bigg(1+\frac{d(\cdot,x_{\gamma})}{r}\bigg)^{-\kappa_{0}}\right\|_{2\rightarrow 2}\nonumber\\
&\leq C\sum_{ k=0}^{\kappa_{0}}\left\|{\rm Ad}_{\gamma,r}^{ k}(e^{i\xi e^{-\frac{L}{\lambda}}})
\bigg(\frac{d(\cdot,x_{\gamma})}{r}\bigg)^{\kappa_{0}- k}\bigg(1+\frac{d(\cdot,x_{\gamma})}{r}\bigg)^{-\kappa_{0}}\right\|_{2\rightarrow 2}\nonumber\\
&\leq C\sum_{ k=0}^{\kappa_{0}}\big{\|}{\rm Ad}_{\gamma,r}^{ k}(e^{i\xi e^{-\frac{L}{\lambda}}})\big{\|}_{2\rightarrow 2}\nonumber\\
&\leq C(1+|\xi|)^{\kappa_{0}}(\sqrt[m]{\lambda}r)^{-\kappa_{0}},
\end{align}
where in the last inequality we applied Lemma \ref{ppp1}.
This, in combination with the Lemma \ref{heatbound} and the doubling condition (\ref{doubling}), yields
\begin{align}\label{estimate01}
{\rm I} &\leq C(2^{j}\sqrt[m]{\lambda}r_{B})^{-\kappa_{0}} \sum_{\gamma\in \mathcal{I}_{r}}
\left\|\bigg(1+\frac{d(\cdot,x_{\gamma})}{r}\bigg)^{\kappa_{0}}\chi_{Q_{\gamma}(r)}e^{-\frac{L}{\lambda}}\chi_{B}f
\right\|_{2}\int_{-\infty}^{+\infty}|\hat{E}(\xi)|(1+|\xi|)^{\kappa_{0}}d\xi\nonumber\\
&\leq C(2^{j}\sqrt[m]{\lambda}r_{B})^{-\kappa_{0}} \sum_{\gamma\in \mathcal{I}_{r}}
\left\|\chi_{Q_{\gamma}(r)}e^{-\frac{L}{\lambda}}\chi_{B}f
\right\|_{2}\|\delta_{\lambda}F\|_{C^{\kappa_{0}+1}}\nonumber\\
&\leq C(2^{j}\sqrt[m]{\lambda}r_{B})^{-\kappa_{0}} \sum_{\beta\in\mathcal{I}_{r}}\sum_{\gamma\in \mathcal{I}_{r}}
\left\|\chi_{Q_{\gamma}(r)}e^{-\frac{L}{\lambda}}\chi_{Q_{\beta}(r)\cap B}f
\right\|_{2}\|\delta_{\lambda}F\|_{C^{\kappa_{0}+1}},
\end{align}
where we used the fact that supp$\delta_{\lambda}F\subset [-1,1]$ implies that
\begin{align}\label{Cnorm}
\int_{-\infty}^{+\infty}|\hat{E}(\xi)|(1+|\xi|)^{\kappa_{0}}d\xi\leq C\|E\|_{H^{\kappa_{0}+1}}\leq C
 \|\delta_{\lambda}F\|_{H^{\kappa_{0}+1}}\leq C \|\delta_{\lambda}F\|_{C^{\kappa_{0}+1}}.
\end{align}
To continue, we use the condition $({\rm PEV}_{p_{0},2}^{\kappa,m})$ and the doubling condition \eqref{doubling} to see that
\begin{align}
&\left\|\chi_{Q_{\gamma}(r)}e^{-\frac{L}{\lambda}}\chi_{Q_{\beta}(r)\cap B}f
\right\|_{2}\nonumber\\
&\leq C(1+\sqrt[m]{\lambda}d(x_{\gamma},x_{\beta}))^{-n-\kappa}
\|\chi_{Q_{\beta}(r)\cap B}V_{\lambda^{-1/m}}^{-\sigma_{p}}f\|_{p}\nonumber\\
&\leq C\mu(B)^{-\sigma_{p}}(1+\sqrt[m]{\lambda}d(x_{\gamma},x_{\beta}))^{-n-\kappa}
(\sqrt[m]{\lambda}r_{B})^{n\sigma_{p}}(\sqrt[m]{\lambda}r)^{-n\sigma_{p}}\|\chi_{Q_{\beta}(r)\cap B}f\|_{p}.
\end{align}
Combining this with inequality \eqref{estimate01}, we conclude that
\begin{align*}
{\rm I}
&\leq C\mu(B)^{-\sigma_{p}}(\sqrt[m]{\lambda}r_{B})^{n\sigma_{p}}(\sqrt[m]{\lambda}r)^{-n\sigma_{p}}(2^{j}\sqrt[m]{\lambda}r_{B})^{-\kappa_{0}} \|\delta_{\lambda}F\|_{C^{\kappa_{0}+1}}\sum_{\beta\in\mathcal{I}_{r}}\sum_{\gamma\in \mathcal{I}_{r}}
(1+\sqrt[m]{\lambda}d(x_{\gamma},x_{\beta}))^{-n-\kappa}\|\chi_{Q_{\beta}(r)\cap B}f\|_{p}\\
&\leq C\mu(B)^{-\sigma_{p}}2^{-j\kappa_0}(\sqrt[m]{\lambda}r)^{-c}(\sqrt[m]{\lambda}r_{B})^{-\kappa_{0}+n\sigma_{p}} \|\delta_{\lambda}F\|_{C^{\kappa_{0}+1}}\sum_{\beta\in\mathcal{I}_{r}}
\|\chi_{Q_{\beta}(r)\cap B}f\|_{p}\\
&\leq C\mu(B)^{-\sigma_{p}}2^{-j\kappa_0}(\sqrt[m]{\lambda}r)^{-c}(\sqrt[m]{\lambda}r_{B})^{-\kappa_{0}+\frac{n}{2}} \|\delta_{\lambda}F\|_{C^{\kappa_{0}+1}}\|f\|_{p}
\end{align*}
for some $c>0$.

To estimate the term ${\rm II}$, we first note that
\begin{align*}
{\rm II}&=\left(\sum_{\substack{\alpha\in \mathcal{I}_{r}\\
d(x_{\alpha},x_{B})\geq 2^{j}r_B}}\Big\|\sum_{\substack{\gamma\in \mathcal{I}_{r}\\
\ d(x_{\gamma},x_{B})> 2^{j-1}r_B}}\chi_{Q_{\alpha}(r)}E(e^{-\frac{L}{\lambda}})\chi_{Q_{\gamma}(r)}e^{-\frac{L}{\lambda}}
\chi_{B}f\Big\|_{2}^{2}\right)^{\frac{1}{2}}\\
&\leq\Big\|\sum_{\substack{\gamma\in \mathcal{I}_{r}\\
\ d(x_{\gamma},x_{B})> 2^{j-1}r_B}}E(e^{-\frac{L}{\lambda}})\chi_{Q_{\gamma}(r)}e^{-\frac{L}{\lambda}}
\chi_{B}f\Big\|_{2}.
\end{align*}
This, in combination with (\ref{fourier}) and the spectral theorem, yields
\begin{align}\label{IIestimateA}
{\rm II} &=\left\|\int_{-\infty}^{+\infty}e^{i\xi e^{-\frac{L}{\lambda}}}\sum_{\substack{\gamma\in \mathcal{I}_{r}\\
\ d(x_{\gamma},x_{B})> 2^{j-1}r_B}}\chi_{Q_{\gamma}(r)}e^{-\frac{L}{\lambda}}
\chi_{B}f\hat{E}(\xi)d\xi\right\|_2\nonumber\\
&\leq\Bigg\|\sum_{\substack{\gamma\in \mathcal{I}_{r}\\
\ d(x_{\gamma},x_{B})> 2^{j-1}r_B}}\chi_{Q_{\gamma}(r)}e^{-\frac{L}{\lambda}}
\chi_{B}f\Bigg\|_2\int_{-\infty}^{+\infty}|\hat{E}(\xi)|d\xi\nonumber\\
&\leq\Bigg\|\sum_{\substack{\gamma\in \mathcal{I}_{r}\\
\ d(x_{\gamma},x_{B})> 2^{j-1}r_B}}\chi_{Q_{\gamma}(r)}e^{-\frac{L}{\lambda}}
\chi_{B}f\Bigg\|_2\|\delta_\lambda F\|_{C^{\kappa_{0}+1}}\nonumber\\
&\leq\Bigg\|e^{-\frac{L}{\lambda}}
\chi_{B}f\Bigg\|_{L^2(X\backslash B(x_B,2^{j-2}r_B))}\|\delta_\lambda F\|_{C^{\kappa_{0}+1}}\nonumber\\
&\leq \left(\sum_{k=j-1}^{\infty}\Bigg\|\chi_{U_k(B)}e^{-\frac{L}{\lambda}}
\chi_{B}f\Bigg\|_{2}^2\right)^{1/2}\|\delta_\lambda F\|_{C^{\kappa_{0}+1}}.
\end{align}
It follows from Lemma \ref{criterion} that for any $k\geq 3$,
\begin{align}\label{sssi}
\Bigg\|\chi_{U_k(B)}e^{-\frac{L}{\lambda}}
V_{\lambda^{-1/m}}^{\sigma_p}\chi_{B}\Bigg\|_{p\rightarrow 2}
&\leq \sup\limits_{\alpha\in\mathcal{I}_{r}}\sum_{\beta\in\mathcal{I}_{r}}\Bigg\|\chi_{Q_{\alpha}(r)\cap U_k(B)}e^{-\frac{L}{\lambda}}
V_{\lambda^{-1/m}}^{\sigma_p}\chi_{Q_{\beta}(r)\cap B}\Bigg\|_{p\rightarrow2}\nonumber\\
&+ \sup\limits_{\beta\in\mathcal{I}_{r}}\sum_{\alpha\in\mathcal{I}_{r}}\Bigg\|\chi_{Q_{\alpha}(r)\cap U_k(B)}e^{-\frac{L}{\lambda}}
V_{\lambda^{-1/m}}^{\sigma_p}\chi_{Q_{\beta}(r)\cap B}\Bigg\|_{p\rightarrow2}\nonumber\\
&\leq C\sup\limits_{\substack{\alpha\in\mathcal{I}_{r}\\Q_{\alpha}(r)\cap U_{k}(B)\neq \emptyset}}\sum_{\substack{\beta\in\mathcal{I}_{r}\\Q_{\beta}(r)\cap B\neq\emptyset}}(1+\sqrt[m]{\lambda}d(x_\alpha,x_\beta))^{-n-\kappa}\nonumber\\
&+ C\sup\limits_{\substack{\beta\in\mathcal{I}_{r}\\Q_{\beta}(r)\cap B\neq\emptyset}}\sum_{\substack{\alpha\in\mathcal{I}_{r}\\Q_{\alpha}(r)\cap U_{k}(B)\neq \emptyset}}(1+\sqrt[m]{\lambda}d(x_\alpha,x_\beta))^{-n-\kappa}\nonumber\\
&\leq C(1+2^{k}\sqrt[m]{\lambda}r_B)^{-\kappa_0}(\sqrt[m]{\lambda}r)^{-n}.
\end{align}
Therefore,
\begin{align*}
{\rm II}&\leq C\mu(B)^{-\sigma_{p}}(\sqrt[m]{\lambda}r_{B})^{n\sigma_{p}}(\sqrt[m]{\lambda}r)^{-c_1}\left(\sum_{k=j-1}^{\infty}(1+2^{k}\sqrt[m]{\lambda}r_B)^{-2\kappa_0}\right)^{1/2}\|\delta_\lambda F\|_{C^{\kappa_{0}+1}}\|f\|_{p}\\
&\leq C\mu(B)^{-\sigma_{p}}2^{-j\kappa_0}(\sqrt[m]{\lambda}r)^{-c_2}(\sqrt[m]{\lambda}r_{B})^{-\kappa_{0}+\frac{n}{2}} \|\delta_{\lambda}F\|_{C^{\kappa_{0}+1}}\|f\|_{p}
\end{align*}
for some $c_1,c_2>0$. This finishes the proof of Lemma \ref{pro2}.
\end{proof}

\medskip

\begin{proof}[Proof of Proposition \ref{p2estimate}]
We embed the set $(2^{j-1}B)^{c}$ into a countable union of amalgam block $\{Q_{\alpha}(r)\}$, that is,
\begin{align*}
(2^{j-1}B)^{c}\subseteq \bigcup_{\substack{\alpha\in \mathcal{I}_{r}\\ \ d(x_{\alpha},x_{B})\geq 2^{j-2}r_B}}Q_{\alpha}(r).
\end{align*}
  Hence, it follows from Lemma \ref{pro2} that
\begin{align*}
\|\chi_{U_{j}(B)}F(L)\chi_{B}f\|_{2}
\leq& \left(\sum_{\substack{\alpha\in \mathcal{I}_{r}\\
d(x_{\alpha},x_{B})\geq 2^{j-2}r_B}}\|\chi_{Q_{\alpha}(r)}F(L)\chi_{ B}f\|_{2}^{2}\right)^{\frac{1}{2}}\\
\leq & C2^{-j\kappa_{0}}\mu(B)^{-\sigma_{p}}(\sqrt[m]{\lambda}r_{B})^{-\kappa_{0}+\frac{n}{2}}
(\sqrt[m]{\lambda}r)^{-c_{0}}\|\delta_{\lambda}F\|_{C^{\kappa_{0}+1}} \|f\|_{p}.
\end{align*}
This implies the estimate (\ref{p2es}).
\end{proof}
}

\subsection{Off-diagonal estimates for oscillatory compactly supported spectral multipliers}
In the previous subsection, we obtain off-diagonal estimates for compactly supported spectral multipliers
with sufficient smoothness, but this estimate is not suitable for those multiplier function with oscillatory
term. Inspired by \cite{Dancola}, to overcome this difficulty, we will use commutators techniques again to
obtain much more subtle estimates for oscillatory compactly supported spectral multipliers. That is, we will the following
Proposition \ref{p2estimate2}, which is a general version of Proposition \ref{p2estimate22}.
\begin{proposition}\label{p2estimate2}
Let $p_{0}\leq p\leq 2$, then there exist constants $C, c_{0}>0$ such that for any ball $B\subset X$ with radius $r_{B}$ and for any $\lambda>0$, $j\geq6 $,
\begin{align*}
\big\|\chi_{U_{j}(B)}e^{itL}F(L)\chi_{B}f\big\|_{2}\leq C2^{-j\kappa_{0}}\mu(B)^{-\sigma_{p}}
(\sqrt[m]{\lambda}r_{B})^{-\kappa_{0}+\frac{n}{2}}(1+\lambda|t|)^{\kappa_{0}}(\sqrt[m]{\lambda}r)^{-c_{0}}\|\delta_{\lambda}F\|_{C^{\kappa_{0}+1}} \|f\|_{p}
\end{align*}
for all Borel functions $F$ such that supp$F\subseteq [-\lambda,\lambda]$, where $r= {\rm min}\{r_{B},\lambda^{-1/m}\}$.
\end{proposition}

To begin with, for every $\lambda>0$, we denote
$$R_{\lambda}:=\Big(I+\frac{L}{\lambda}\Big)^{-1}.$$
The key observation is the following lemma.
\begin{lemma}\label{adresolvent1}
For every $1\leq  k\leq\kappa_{0}$, $r, \lambda>0$ and $\ell\in\mathcal{I}_{r}$, the operator ${\rm Ad}_{\ell,r}(R_{\lambda}^{2^{ k+1}-2}e^{itL})$
 is given by a finite combination of operators of the following type:
$$R_{\lambda}^{\mu_{1}}R_{\lambda}^{2^{ k}-2}e^{itL}{\rm Ad}_{\ell,r}(R_{\lambda}^{\mu_{2}})R_{\lambda}^{\mu_{3}},\
 \ \mu_{1}, \mu_{2}, \mu_{3}\in\mathbb{N},$$
$$R_{\lambda}^{\mu_{1}}{\rm Ad}_{\ell,r}(R_{\lambda}^{\mu_{2}})R_{\lambda}^{2^{ k}-2}e^{itL}R_{\lambda}^{\mu_{3}},\
 \ \mu_{1}, \mu_{2}, \mu_{3}\in\mathbb{N},$$
$$\lambda t\int_{0}^{1}R_{\lambda}^{2^{ k}-2}e^{i\rho tL}{\rm Ad}_{\ell,r}(R_{\lambda})R_{\lambda}^{2^{ k}-2}e^{i(1-\rho)tL}d\rho.$$
\end{lemma}
\begin{proof}
We deduce this result by induction on $ k$. First of all, we point out that it is true for $ k=1$. Indeed, by the commutator formula and  Lemma \ref{du},
\begin{align*}
{\rm Ad}_{\ell,r}(R_{\lambda}e^{itL}R_{\lambda})&={\rm Ad}_{\ell,r}(R_{\lambda})e^{itL}R_{\lambda}
+R_{\lambda}e^{itL}{\rm Ad}_{\ell,r}(R_{\lambda})+R_{\lambda}{\rm Ad}_{\ell,r}(e^{itL})R_{\lambda}\\
&={\rm Ad}_{\ell,r}(R_{\lambda})e^{itL}R_{\lambda}+R_{\lambda}e^{itL}{\rm Ad}_{\ell,r}(R_{\lambda})
+i\lambda t\int_{0}^{1}e^{i\rho tL}R_{\lambda}{\rm Ad}_{\ell,r}(L/\lambda)R_{\lambda}e^{i(1-\rho)tL}d\rho.
\end{align*}
Observe that
\begin{align*}
R_{\lambda}{\rm Ad}_{\ell,r}(L/\lambda)R_{\lambda}=R_{\lambda}\frac{d(x_{\ell},\cdot)}{r}(R_{\lambda}^{-1}-I)
R_{\lambda}-R_{\lambda}(R_{\lambda}^{-1}-I)\frac{d(x_{\ell},\cdot)}{r}R_{\lambda}=-{\rm Ad}_{\ell,r}(R_{\lambda}).
\end{align*}
Hence, the result for $ k=1$ is verified.

Now, let us assume this lemma holds for $ k-1$ and compute
\begin{align*}
{\rm Ad}_{\ell,r}(R_{\lambda}^{2^{ k+1}-2}e^{itL})
&={\rm Ad}_{\ell,r}(R_{\lambda}^{2^{ k-1}}(R_{\lambda}^{2^{ k}-2}e^{itL})R_{\lambda}^{2^{ k-1}})\\
&={\rm Ad}_{\ell,r}(R_{\lambda}^{2^{ k-1}})R_{\lambda}^{2^{ k}-2}e^{itL}R_{\lambda}^{2^{ k-1}}+
R_{\lambda}^{2^{ k-1}}{\rm Ad}_{\ell,r}(R_{\lambda}^{2^{ k}-2}e^{itL})R_{\lambda}^{2^{ k-1}}+
R_{\lambda}^{2^{ k-1}}R_{\lambda}^{2^{ k}-2}e^{itL}{\rm Ad}_{\ell,r}(R_{\lambda}^{2^{ k-1}}).
\end{align*}
The first and the last term are of the desired form. The second one is also of the desired form by the inductive hypothesis, since
\begin{align*}
R_{\lambda}^{2^{ k-1}}R_{\lambda}^{2^{ k-1}-2}=R_{\lambda}^{2^{ k}-2}.
\end{align*}
This finishes the proof of Lemma \ref{adresolvent1}.
\end{proof}

\begin{lemma}\label{adresolvent2}
There exists a constant $C>0$, such that for any $\lambda,r>0$, $\ell\in\mathcal{I}_r$ and $0\leq k\leq \kappa_{0}$,
\begin{align*}
\|{\rm Ad}_{\ell,r}^{ k}(R_{\lambda})f\|_{2}\leq C(\sqrt[m]{\lambda}r)^{- k}\|f\|_{2}.
\end{align*}
\end{lemma}
\begin{proof}
Denote $K_{T}(x,y)$ be the distribution kernel of the operator $T$. Then by induction on $ k\in [0,\kappa_{0}]$,
we have
\begin{align*}
{\rm Ad}_{\ell,r}^{ k}(R_{\lambda})f(x)
&=\int_{X}\left(\frac{d(x_{\ell},x)}{r}-\frac{d(x_{\ell},y)}{r}\right)^{ k}K_{R_{\lambda}}(x,y)f(y)d\mu(y)\\
&=(\sqrt[m]{\lambda}r)^{- k}\int_{X}\left(\sqrt[m]{\lambda}d(x_{\ell},x)-\sqrt[m]{\lambda}d(x_{\ell},y)\right)^{ k}K_{R_{\lambda}}(x,y)f(y)d\mu(y).
\end{align*}
Applying the following representation formula
\begin{align*}
R_{\lambda}=\int_{0}^{\infty}e^{-\tau L/\lambda}e^{-\tau}d\tau,
\end{align*}
we see that
\begin{align*}
{\rm Ad}_{\ell,r}^{ k}(R_{\lambda})f(x)
&=(\sqrt[m]{\lambda}r)^{- k}\int_{0}^{\infty}\int_{X}\left(\sqrt[m]{\lambda}d(x_{\ell},x)-
\sqrt[m]{\lambda}d(x_{\ell},y)\right)^{ k}p_{\tau/\lambda}(x,y)f(y)d\mu(y)e^{-\tau}d\tau\\
&=(\sqrt[m]{\lambda}r)^{- k}\int_{0}^{\infty}{\rm Ad}_{\ell,(\tau/\lambda)^{1/m}}^{ k}
(e^{-\frac{\tau }{\lambda}L})f(x)\tau^{\frac{ k}{m}}e^{-\tau}d\tau.
\end{align*}

It follows from (\ref{comheat}) with $r$ replaced by $(\tau/\lambda)^{1/m}$ and $\lambda$ replaced by $\lambda/\tau$ that
\begin{align*}
\|{\rm Ad}_{\ell,(\tau/\lambda)^{1/m}}^{ k}(e^{-\frac{\tau }{\lambda}L})f\|_{2}\leq C\|f\|_{2},
\end{align*}
which means that
\begin{align*}
\|{\rm Ad}_{\ell,r}^{ k}(R_{\lambda})f\|_{2}\leq (\sqrt[m]{\lambda}r)^{- k}\int_{0}^{\infty}
\|{\rm Ad}_{\ell,(\tau/\lambda)^{1/m}}^{ k}(e^{-\frac{\tau }{\lambda}L})f(x)\|_{2}\tau^{\frac{ k}{m}}e^{-\tau}d\tau
\leq C(\sqrt[m]{\lambda}r)^{- k}\|f\|_{2}.
\end{align*}
This ends the proof of Lemma \ref{adresolvent2}.
\end{proof}
Now we apply Lemmas \ref{adresolvent1} and \ref{adresolvent2} to obtain the following crucial commutator estimate for Schr\"{o}dinger group.
\begin{lemma}\label{commutatorestimate}
There exist  constants $C,c>0$, such that for any $\lambda>0$, $r\leq \lambda^{-1/m}$, $\ell\in\mathcal{I}_r$ and $0\leq k\leq j\leq\kappa_{0}$,
\begin{align*}
\|{\rm Ad}_{\ell,r}^{ k}(R_{\lambda}^{2^{j+1}-2}e^{itL})\|_{2\rightarrow 2}\leq C(1+\lambda|t|)^{ k}(\sqrt[m]{\lambda}r)^{-c}.
\end{align*}
\end{lemma}
\begin{proof}
The result will be shown by induction on $j=0,\cdots,\kappa_{0}$.

By the spectral theorem, the result is true for $j=0$. Now we assume it holds for $j-1$, and write, for $ k\leq j$,
\begin{align*}
{\rm Ad}_{\ell,r}^{ k}(R_{\lambda}^{2^{j+1}-2}e^{itL})={\rm Ad}_{\ell,r}^{ k-1}({\rm Ad}_{\ell,r}(R_{\lambda}^{2^{j+1}-2}e^{itL})).
\end{align*}
By Lemma \ref{adresolvent1} and the formula
\begin{align*}
{\rm Ad}_{\ell,r}^{ k-1}(T_{1}\cdots T_{n})=\sum_{\alpha_{1}+\cdots+\alpha_{n}= k-1}
\frac{( k-1)!}{\alpha_{1}!\cdots \alpha_{n}!}{\rm Ad}_{\ell,r}^{\alpha_{1}}(T_{1})\cdots {\rm Ad}_{\ell,r}^{\alpha_{n}}(T_{n})
\end{align*}
as well as the inductive hypothesis and Lemma \ref{adresolvent2}, we obtain the inequality.
\end{proof}
Now, the result in the previous subsection can be extended to oscillatory compactly supported spectral multipliers.
{
\begin{lemma}\label{oscillatory2}
Let $p_{0}\leq p\leq 2$, then there exist constants $C,c_{0}>0$, such that for any ball $B\subset X$
 with radius $r_{B}$ and for any $\lambda>0$, $j\geq 3$, $r= {\rm min}\{r_{B},\lambda^{-1/m}\}$,
\begin{align*}
&\left(\sum_{\substack{\alpha\in \mathcal{I}_{r}\\
d(x_{\alpha},x_{B})\geq 2^{j}r_B}}\|\chi_{Q_{\alpha}(r)}e^{itL}F(L)\chi_{ B}f\|_{2}^{2}\right)^{\frac{1}{2}}\\
\leq  & C\mu(B)^{-\sigma_{p}}2^{-j\kappa_{0}}(\sqrt[m]{\lambda}r_{B})^{-\kappa_{0}+\frac{n}{2}}(1+\lambda|t|)^{\kappa_{0}}
(\sqrt[m]{\lambda}r)^{-c_{0}}\|\delta_{\lambda}F\|_{C^{\kappa_{0}+1}} \|f\|_{p}
\end{align*}
for all Borel functions $F$ such that supp$F\subseteq [-\lambda,\lambda]$.
\end{lemma}
\begin{proof}
To begin with, we note that
\begin{align*}
e^{itL}F(L)=R_{\lambda}^{2^{\kappa_{0}+1}-2}e^{itL}\delta_{\lambda^{-1}}G(L),
\end{align*}
where $G(L)=R_{1}^{-2^{\kappa_{0}+1}+2}\delta_{\lambda}F(L)$.

Hence,
\begin{align*}
&\left(\sum_{\substack{\alpha\in \mathcal{I}_{r}\\
d(x_{\alpha},x_{B})\geq 2^{j}r_B}}\|\chi_{Q_{\alpha}(r)}e^{itL}F(L)\chi_{ B}f\|_{2}^{2}\right)^{\frac{1}{2}}\\
\leq &\left(\sum_{\substack{\alpha\in \mathcal{I}_{r}\\
d(x_{\alpha},x_{B})\geq 2^{j}r_B}}\big\|\sum_{\substack{\gamma\in \mathcal{I}_{r}\\
\ d(x_{\gamma},x_{B})\leq 2^{j-1}r_B}}\chi_{Q_{\alpha}(r)}R_{\lambda}^{2^{\kappa_{0}+1}-2}e^{itL}
\chi_{Q_{\gamma}(r)}\delta_{\lambda^{-1}}G(L)\chi_{B}f\big\|_{2}^{2}\right)^{\frac{1}{2}}\\
+&\left(\sum_{\substack{\alpha\in \mathcal{I}_{r}\\
d(x_{\alpha},x_{B})\geq 2^{j}r_B}}\big\|\sum_{\substack{\gamma\in \mathcal{I}_{r}\\
\ d(x_{\gamma},x_{B})> 2^{j-1}r_B}}\chi_{Q_{\alpha}(r)}R_{\lambda}^{2^{\kappa_{0}+1}-2}e^{itL}\chi_{Q_{\gamma}(r)}\delta_{\lambda^{-1}}G(L)
\chi_{B}f\big\|_{2}^{2}\right)^{\frac{1}{2}}\\
=:& {\rm I} + {\rm II}.
\end{align*}

To estimate the first term ${\rm I}$, we note that if $x\in Q_{\alpha}(r)$, then
\begin{align*}
\frac{d(x,x_{\gamma})}{r}\geq \frac{d(x_{\alpha},x_{\gamma})}{r}-\frac{d(x,x_{\alpha})}{r}\geq
\frac{d(x_{\alpha},x_{\gamma})}{r}-1\geq \frac{1}{2}\frac{d(x_{\alpha},x_{\gamma})}{r}.
\end{align*}
Besides, since $d(x_{\alpha},x_{B})\geq 2^{j}r_B$ and $d(x_{\gamma},x_{B})\leq 2^{j-1}r_B$,
\begin{align*}
\frac{d(x_{\alpha},x_{\gamma})}{r}\geq \frac{d(x_\alpha,x_B)}{r}-\frac{d(x_\gamma,x_B)}{r}\geq \frac{2^{j-1}r_B}{r}.
\end{align*}
Hence,
\begin{align*}
{\rm I}&\leq \sum_{\substack{\gamma\in \mathcal{I}_{r}\\
\ d(x_{\gamma},x_{B})\leq 2^{j-1}r_B}}\left(\sum_{\substack{\alpha\in \mathcal{I}_{r}\\
d(x_{\alpha},x_{B})\geq 2^{j}r_B}}\big\|\chi_{Q_{\alpha}(r)}R_{\lambda}^{2^{\kappa_{0}+1}-2}e^{itL}
\chi_{Q_{\gamma}(r)}\delta_{\lambda^{-1}}G(L)\chi_{B}f\big\|_{2}^{2}\right)^{\frac{1}{2}}\\
&\leq \sum_{\gamma\in \mathcal{I}_{r}}\left(\sum_{\substack{\alpha\in \mathcal{I}_{r}\\
d(x_{\alpha},x_{\gamma})\geq 2^{j-1}r_B}}\big\|\chi_{Q_{\alpha}(r)}R_{\lambda}^{2^{\kappa_{0}+1}-2}e^{itL}
\chi_{Q_{\gamma}(r)}\delta_{\lambda^{-1}}G(L)\chi_{B}f\big\|_{2}^{2}\right)^{\frac{1}{2}}\\
&\leq C\sum_{\gamma\in \mathcal{I}_{r}}\left(\sum_{\substack{\alpha\in
\mathcal{I}_{r}\\ \ d(x_{\alpha},x_{\gamma})\geq 2^{j-1}r_{B}}}
\left(\frac{d(x_{\alpha},x_{\gamma})}{r}\right)^{-2\kappa_{0}}
\left\|\left(\frac{d(\cdot,x_{\gamma})}{r}\right)^{\kappa_{0}}\chi_{Q_{\alpha}(r)}R_{\lambda}^{2^{\kappa_{0}+1}-2}e^{itL}
\chi_{Q_{\gamma}(r)}\delta_{\lambda^{-1}}G(L)\chi_{B}f\right\|_{2}^{2}\right)^{\frac{1}{2}}\\
&\leq C\left(\frac{2^{j}r_B}{r}\right)^{-\kappa_0}\sum_{\gamma\in \mathcal{I}_{r}}\left(\sum_{\alpha\in
\mathcal{I}_{r}}
\left\|\left(\frac{d(\cdot,x_{\gamma})}{r}\right)^{\kappa_{0}}\chi_{Q_{\alpha}(r)}R_{\lambda}^{2^{\kappa_{0}+1}-2}e^{itL}
\chi_{Q_{\gamma}(r)}\delta_{\lambda^{-1}}G(L)\chi_{B}f\right\|_{2}^{2}\right)^{\frac{1}{2}}\\
&\leq C\left(\frac{2^{j}r_B}{r}\right)^{-\kappa_0}\sum_{\gamma\in \mathcal{I}_{r}}
\left\|\left(\frac{d(\cdot,x_{\gamma})}{r}\right)^{\kappa_{0}}R_{\lambda}^{2^{\kappa_{0}+1}-2}e^{itL}
\chi_{Q_{\gamma}(r)}\delta_{\lambda^{-1}}G(L)\chi_{B}f\right\|_{2}.
\end{align*}
To continue, applying the following formula for commutators (see Lemma 3.1, \cite{Jensen}):
\begin{align*}
\bigg(\frac{d(\cdot,x_{\gamma})}{r}\bigg)^{\kappa_{0}}R_{\lambda}^{2^{\kappa_{0}+1}-2}e^{itL}
=\sum_{ k=0}^{\kappa_{0}}\Gamma(\kappa_{0}, k){\rm Ad}_{\gamma,r}^{ k}(R_{\lambda}^{2^{\kappa_{0}+1}-2}e^{itL})\bigg(\frac{d(\cdot,x_{\gamma})}{r}\bigg)^{\kappa_{0}- k},
\end{align*}
we obtain that there exists a constant $c>0$ such that
\begin{align}\label{commutator02}
&\left\|\bigg(\frac{d(\cdot,x_{\gamma})}{r}\bigg)^{\kappa_{0}}R_{\lambda}^{2^{\kappa_{0}+1}-2}e^{itL}
\bigg(1+\frac{d(\cdot,x_{\gamma})}{r}\bigg)^{-\kappa_{0}}\right\|_{2\rightarrow 2}\nonumber\\
&\leq C\sum_{ k=0}^{\kappa_{0}}\left\|{\rm Ad}_{\gamma,r}^{ k}(R_{\lambda}^{2^{\kappa_{0}+1}-2}e^{itL})
\bigg(\frac{d(\cdot,x_{\gamma})}{r}\bigg)^{\kappa_{0}- k}\bigg(1+\frac{d(\cdot,x_{\gamma})}{r}\bigg)^{-\kappa_{0}}\right\|_{2\rightarrow 2}\nonumber\\
&\leq C\sum_{ k=0}^{\kappa_{0}}\big{\|}{\rm Ad}_{\gamma,r}^{ k}(R_{\lambda}^{2^{\kappa_{0}+1}-2}e^{itL})\big{\|}_{2\rightarrow 2}\nonumber\\
&\leq C(1+\lambda|t|)^{\kappa_{0}}(\sqrt[m]{\lambda}r)^{-c},
\end{align}
where in the last inequality we applied Lemma \ref{commutatorestimate}.
Applying the estimate (\ref{commutator02}), we conclude that there exists a constant $c>0$ such that
\begin{align}\label{estiII}
{\rm I }&\leq C\left(\frac{2^{j}r_B}{r}\right)^{-\kappa_0}\sum_{\gamma\in \mathcal{I}_{r}}\left\|\bigg(\frac{d(\cdot,x_{\gamma})}{r}\bigg)^{\kappa_{0}}R_{\lambda}^{2^{\kappa_{0}+1}-2}e^{itL}
\bigg(1+\frac{d(\cdot,x_{\gamma})}{r}\bigg)^{-\kappa_{0}}\right\|_{2\rightarrow 2}
\left\|\bigg(1+\frac{d(\cdot,x_{\gamma})}{r}\bigg)^{\kappa_{0}}
\chi_{Q_{\gamma}(r)}\delta_{\lambda^{-1}}G(L)\chi_{B}f\right\|_{2}\nonumber\\
&\leq C\left(\frac{2^{j}r_B}{r}\right)^{-\kappa_0}(1+\lambda|t|)^{\kappa_{0}}(\sqrt[m]{\lambda}r)^{-c}\sum_{\gamma\in \mathcal{I}_{r}}\left\|
\chi_{Q_{\gamma}(r)}\delta_{\lambda^{-1}}G(L)\chi_{B}f\right\|_{2}.
\end{align}
To continue, we write
\begin{align*}
\sum_{\gamma\in \mathcal{I}_{r}}\left\|
\chi_{Q_{\gamma}(r)}\delta_{\lambda^{-1}}G(L)\chi_{B}f\right\|_{2}
&\leq\sum_{\beta\in \mathcal{I}_{r}}\sum_{\gamma\in \mathcal{I}_{r}}\left\|
\chi_{Q_{\gamma}(r)}\delta_{\lambda^{-1}}G(L)\chi_{Q_{\beta}(r)\cap B}f\right\|_{2}\\
&\leq\sum_{\beta\in \mathcal{I}_{r}}\sum_{\substack{\gamma\in \mathcal{I}_{r}\\d(x_\gamma,x_\beta)\leq 2^{3}r}}\left\|
\chi_{Q_{\gamma}(r)}\delta_{\lambda^{-1}}G(L)\chi_{Q_{\beta}(r)\cap B}f\right\|_{2}\\
&+\sum_{\beta\in \mathcal{I}_{r}}\sum_{k=4}^\infty\sum_{\substack{\gamma\in \mathcal{I}_{r}\\2^{k-1}\leq d(x_\gamma,x_\beta)\leq 2^{k}r}}\left\|
\chi_{Q_{\gamma}(r)}\delta_{\lambda^{-1}}G(L)\chi_{Q_{\beta}(r)\cap B}f\right\|_{2}={\rm (A)}+{\rm (B)}.
\end{align*}
By Lemma \ref{global00},
\begin{align*}
\left\|
\chi_{Q_{\gamma}(r)}\delta_{\lambda^{-1}}G(L)\chi_{Q_{\beta}(r)\cap B}f\right\|_{2}\leq \mu(B)^{-\sigma_{p}}(\sqrt[m]{\lambda}r_B)^{n\sigma_p}
(\sqrt[m]{\lambda}r)^{-c}\|\delta_{\lambda}F\|_{C^{1}}\|\chi_{Q_{\beta}(r)\cap B}f\|_{p}.
\end{align*}
Therefore,
\begin{align*}
{\rm (A)}&\leq \mu(B)^{-\sigma_{p}}(\sqrt[m]{\lambda}r_B)^{n\sigma_p}
(\sqrt[m]{\lambda}r)^{-c}\|\delta_{\lambda}F\|_{C^{1}}\sum_{\beta\in \mathcal{I}_{r}}\|\chi_{Q_{\beta}(r)\cap B}f\|_{p}\\
&\leq \mu(B)^{-\sigma_{p}}(\sqrt[m]{\lambda}r_B)^{\frac{n}{2}}
(\sqrt[m]{\lambda}r)^{-c}\|\delta_{\lambda}F\|_{C^{1}}\|f\|_{p}.
\end{align*}
Next, similar to the proof of Proposition \ref{p2estimate}, we obtain that
\begin{align*}
{\rm (B)}&\leq \sum_{\beta\in \mathcal{I}_{r}}\sum_{k=4}^\infty 2^{\frac{kn}{2}}\left(\sum_{\substack{\gamma\in \mathcal{I}_{r}\\2^{k-1}\leq d(x_\gamma,x_\beta)\leq 2^{k}r}}\left\|
\chi_{Q_{\gamma}(r)}\delta_{\lambda^{-1}}G(L)\chi_{Q_{\beta}(r)\cap B}f\right\|_{2}^2\right)^{1/2}\\
&\leq \sum_{\beta\in \mathcal{I}_{r}}\sum_{k=2}^\infty 2^{\frac{kn}{2}}\|\chi_{U_{k}(B(x_\beta,r))}\delta_{\lambda^{-1}}G(L)\chi_{Q_{\beta}(r)\cap B}f\|_{2}\\
&\leq \sum_{\beta\in \mathcal{I}_{r}}\sum_{k=2}^\infty 2^{-(\kappa_0-\frac{n}{2})k}\mu(B)^{-\sigma_{p}}(\sqrt[m]{\lambda}r_B)^{n\sigma_p}
(\sqrt[m]{\lambda}r)^{-c}\|\delta_{\lambda}F\|_{C^{\kappa_{0}+1}}\|\chi_{Q_\beta(r)\cap B}f\|_{p}\\
&\leq \mu(B)^{-\sigma_{p}}(\sqrt[m]{\lambda}r_B)^{n\sigma_p}
(\sqrt[m]{\lambda}r)^{-c}\|\delta_{\lambda}F\|_{C^{\kappa_{0}+1}}\sum_{\beta\in \mathcal{I}_{r}}\|\chi_{Q_\beta(r)\cap B}f\|_{p}\\
&\leq \mu(B)^{-\sigma_{p}}(\sqrt[m]{\lambda}r_B)^{\frac{n}{2}}
(\sqrt[m]{\lambda}r)^{-c}\|\delta_{\lambda}F\|_{C^{\kappa_{0}+1}}\|f\|_{p}.
\end{align*}
Combining the estimates for (A) and (B), we see that
\begin{align*}
{\rm I }&\leq C\mu(B)^{-\sigma_{p}}2^{-j\kappa_{0}}(\sqrt[m]{\lambda}r_{B})^{-\kappa_{0}+\frac{n}{2}}(1+\lambda|t|)^{\kappa_{0}}
(\sqrt[m]{\lambda}r)^{-c_{0}}\|\delta_{\lambda}F\|_{C^{\kappa_{0}+1}} \|f\|_{p}
\end{align*}
for some constants $C,c>0$.

As for the second term ${\rm II}$, we first note that
\begin{align*}
{\rm II}&=\left(\sum_{\substack{\alpha\in \mathcal{I}_{r}\\
d(x_{\alpha},x_{B})\geq 2^{j}r_B}}\big\|\sum_{\substack{\gamma\in \mathcal{I}_{r}\\
\ d(x_{\gamma},x_{B})> 2^{j-1}r_B}}\chi_{Q_{\alpha}(r)}R_{\lambda}^{2^{\kappa_{0}+1}-2}e^{itL}\chi_{Q_{\gamma}(r)}\delta_{\lambda^{-1}}G(L)
\chi_{B}f\big\|_{2}^{2}\right)^{\frac{1}{2}}\\
&\leq\Bigg\|\sum_{\substack{\gamma\in \mathcal{I}_{r}\\
\ d(x_{\gamma},x_{B})> 2^{j-1}r_B}}R_{\lambda}^{2^{\kappa_{0}+1}-2}e^{itL}\chi_{Q_{\gamma}(r)}\delta_{\lambda^{-1}}G(L)
\chi_{B}f\Bigg\|_{2}\\
&\leq\Bigg\|\sum_{\substack{\gamma\in \mathcal{I}_{r}\\
\ d(x_{\gamma},x_{B})> 2^{j-1}r_B}}\chi_{Q_{\gamma}(r)}\delta_{\lambda^{-1}}G(L)
\chi_{B}f\Bigg\|_{2}.
\end{align*}
where in the last inequality we used the spectral theorem. To continue, we apply Proposition \ref{p2estimate} to obtain that
\begin{align}\label{IIestimateA22}
{\rm II} &\leq \|\delta_{\lambda^{-1}}G(L)
\chi_{B}f\|_{L^{2}(X\backslash B(x_B,2^{j-2}r_B))}\nonumber\\
&\leq \left(\sum_{k\geq 0}\|\chi_{U_{k+j-1}(B)}\delta_{\lambda^{-1}}G(L)\chi_B f\|_{2}^{2}\right)^{1/2}\nonumber\\
&\leq \left(\sum_{k\geq 0}2^{-2(j+k)\kappa_0}\right)^{1/2}\mu(B)^{-\sigma_{p}}
(\sqrt[m]{\lambda}r_{B})^{-\kappa_{0}+n\sigma_{p}}(\sqrt[m]{\lambda}r)^{-c_{0}}\|\delta_{\lambda}F\|_{C^{\kappa_{0}+1}}\|\chi_B f\|_{p}\nonumber\\
&\leq 2^{-j\kappa_0}\mu(B)^{-\sigma_{p}}
(\sqrt[m]{\lambda}r_{B})^{-\kappa_{0}+n\sigma_{p}}(\sqrt[m]{\lambda}r)^{-c_{0}}\|\delta_{\lambda}F\|_{C^{\kappa_{0}+1}}\|\chi_B f\|_{p}.
\end{align}
This ends the proof of Lemma \ref{oscillatory2}.
\end{proof}
}
\begin{proof}[Proof of Proposition \ref{p2estimate2}]
The proof of this proposition can be shown by a similar argument as in the proof
 of Proposition \ref{p2estimate}. We omit the details and leave it to the readers.
\end{proof}

\medskip

\section{Proof of Theorem~\ref{main}}
\setcounter{equation}{0}

\subsection{Proof of boundedness on $H_{L}^{q}(X)$ for $\frac{2n}{2\kappa_0+n}<q<1$}
In this subsection, with the help of Proposition \ref{p2estimate2} and Lemma \ref{moleculardecomposition},
we will borrow the ideas from \cite{CDLY2, HM} to show Proposition \ref{tool}.

\begin{proof}[Proof of Proposition \ref{tool}]
Choose $M$ be a sufficient large constant and assume that $a(x)$ is a $(q,2,M,\epsilon)$-molecule
associated to a ball $B=B(x_B,r_{B})$ and $a=L^{M}b$ such that  for every $k=0,1,2,\dots,M$ and $j=0,1,2,\dots$,
\begin{align}\label{molecule}
\|(r_{B}^mL)^{k}b\|_{L^2(U_j(B))}\leq 2^{-j\epsilon} r_{B}^{mM}
\mu(2^jB)^{-(\frac{1}{q}-\frac{1}{2})}.
\end{align}
By Lemma \ref{moleculardecomposition}  and a standard argument (see for example, \cite{DY3, HLMMY, HM, KU}),
it suffices to show  that
\begin{align}\label{goalsch}
\left\|\left(\int_{0}^{+\infty}\int_{d(x,y)<\tau^{1/m}}\big|\phi(\tau L)e^{itL}F(L)a(y)
\big|^{2}\frac{d\mu(y)}{V(x,\tau^{1/m})}\frac{d\tau}{\tau}\right)^{\frac{1}{2}}\right\|_{L^{q}(X)}\leq C(1+|t|)^{n(\frac{1}{q}-\frac{1}{2})},
\end{align}
where, for simplicity, we denote $F(\lambda)=(1+\lambda)^{-n(\frac{1}{q}-\frac{1}{2})}$.

Let us show \eqref{goalsch}. Following  \cite[Lemma 8.1]{HM}, we write
\begin{eqnarray*}
I&=& mr_{B}^{-m}\int_{r_{B}}^{\sqrt[m]{2}r_{B}}s^{m-1}ds \cdot I\nonumber\\
&=& mr_{B}^{-m} \int_{r_{B}}^{\sqrt[m]{2}r_{B}}s^{m-1}(I-e^{-s^mL})^Mds
 +  \sum_{\nu=1}^M C_{\nu,M}  r_{B}^{-m}
\int_{r_{B}}^{\sqrt[m]{2}r_{B}}
s^{m-1} e^{- \nu s^mL} ds
\end{eqnarray*}
\noindent
for some constant $C_{\nu, M}$ depending on $\nu$ and $M$ only. Besides, it follows by
 $\partial_se^{-\nu s^mL}=-m\nu s^{m-1}Le^{-\nu s^mL} $ that
\begin{align}\label{iter0}
m\nu L\int_{r_{B}}^{\sqrt[m]{2}r_{B}} s^{m-1} e^{-\nu s^mL} ds
=e^{-\nu r_{B}^mL}-e^{-2\nu r_{B}^mL}
=  e^{-\nu r_{B}^mL}
(I-e^{-r_{B}^mL})    \sum_{\mu=0}^{\nu -1}e^{-\mu r_{B}^mL}.
\end{align}
Then we  iterate the procedure above  $M$ times to conclude that for every $x\in X,$
\begin{eqnarray}\label{iterate}
\phi(\tau L)F(L) a(x)
&=&\sum_{k=0}^{M-1} r_{B}^{-m}  \int_{r_{B}}^{\sqrt[m] 2 r_{B}} s^{m-1}(I-e^{-s^mL})^{M}G_{k,r_{B},M}(L)\phi(\tau L)F(L) (r_{B}^{-mk}L^{M-k} b)ds\nonumber\\
 && + \sum_{\nu =1}^{(2M-1)M} C(\nu, k, M) e^{-\nu  r_{B}^mL}\phi(\tau L)F(L)(I-e^{-r_{B}^mL})^M(r_{B}^{-mM} b)(x)\nonumber\\
 &=:&\sum_{k=0}^{M-1}E_{k}(x)+E_{M}(x),
\end{eqnarray}
where
$$
G_{0,r_{B},M}(\lambda):=m^M\left(r_{B}^{-m}\int_{r_{B}}^{\sqrt[m] 2 r_{B}} s^{m-1}(I-e^{-s^m\lambda})^{M}ds\right)^{M-1},
$$
and  for $k=1,2,\cdots,M-1$,
$$
G_{k,r_{B},M}(\lambda):=(1-e^{-r_{B}^m\lambda})^k \left(r_{B}^{-m}\int_{r_{B}}^{\sqrt[m] 2 r_{B}} s^{m-1}(I-e^{-s^m\lambda})^{M}ds\right)^{M-k-1}
\sum_{\nu =1}^{(2M-1)k} C(\nu,k, M) e^{-\nu r_{B}^m\lambda}.
$$
To continue, we consider two cases: $k=0, 1, \cdots, M-1$ and $k=M.$

\medskip

\noindent
{\bf Case 1.} \ $k=0, 1, \cdots, M-1$. \ In this case, we see that{
 \begin{eqnarray}\hspace{0.2cm}
&&\hspace{-1cm}\left\|\left(\int_0^{\infty}\!\!\!\!\int_{\substack{  d(x,y)<\tau^{1/m}}}
\big|e^{itL}E_k(y)\big|^2 {d\mu(y)\over V(x,\tau^{1/m})}{d\tau\over \tau}\right)^{1/2}\right\|_{L^{q}}\nonumber\\
&\leq& C\sup_{s\in [r_{B},\sqrt[m] 2 r_{B}]}\left\|\left(\int_0^{\infty}\!\!\!\!\int_{\substack{  d(x,y)<\tau^{1/m}}}
\big|e^{itL}F_{\tau,s}(L) (r_{B}^{-mk}L^{M-k} b)(y)\big|^2 {d\mu(y)\over V(x,\tau^{1/m})}{d\tau\over \tau}\right)^{1/2}\right\|_{L^{q}}\nonumber\\
&\leq& C\left(\sum_{j\geq 0}\sup_{s\in [r_{B},\sqrt[m] 2 r_{B}]}\left\|\left(\int_0^{\infty}\!\!\!\!\int_{\substack{  d(x,y)<\tau^{1/m}}}
\big|e^{itL}F_{\tau,s}(L)\chi_{U_{j}(B)} (r_{B}^{-mk}L^{M-k} b)(y)\big|^2 {d\mu(y)\over V(x,\tau^{1/m})}{d\tau\over \tau}\right)^{1/2}\right\|_{L^{q}}^q\right)^{1/q}\nonumber\\
&=:& C\left(\sum_{j\geq 0}\sup_{s\in [r_{B},\sqrt[m] 2 r_{B}]} \|E(k, j, s)\|_{L^{q}(X)}^q\right)^{1/q},
\end{eqnarray}}
where $F_{\tau,s}(\lambda):=\phi_\tau(\lambda) F(\lambda)(1-e^{-s^m\lambda})^{M}G_{k,r_{B},M}(\lambda)$ and
$$
E(k, j, s)=\left(\int_0^{\infty}\!\!\!\!\int_{\substack{  d(x,y)<\tau^{1/m}}}
\big|e^{itL}F_{\tau,s}(L)\chi_{U_{j}(B)} (r_{B}^{-mk}L^{M-k} b)(y)\big|^2 {d\mu(y)\over V(x,\tau^{1/m})}{d\tau\over \tau}\right)^{1/2}.
$$
Let us estimate the term $\|E(k, j, s)\|_{L^{q}(X)}.$ Note that  $\|F\|_{\infty}+\|G_{k,r_{B},M}\|_{L^\infty}\leq C$.
We apply the estimate  (\ref{molecule}) and the $L^{2}$-boundedness of the square function to see that
 \begin{eqnarray*}
&&\hspace{-1cm}\|E(k, j, s)\|_{L^{2}(64(1+|t|)2^{j} B)}\nonumber\\
&\leq&C \left(\int_{0}^{\infty}\int_{X}\big|e^{itL}F_{\tau,s}(L)\chi_{U_{j}(B)}(r_{B}^{-mk}L^{M-k}b)(y)\big|^{2}
\int_{d(x,y)<\tau^{1/m}}d\mu(x)\frac{d\mu(y)}{V(y,\tau^{1/m})}\frac{d\tau}{\tau}\right)^{\frac{1}{2}}\nonumber\\
&\leq&C\left(\int_{0}^{\infty}\big\|e^{itL}F_{\tau,s}(L)\chi_{U_{j}(B)}(r_{B}^{-mk}L^{M-k}b)\big\|_{2}^{2}\frac{d\tau}{\tau}\right)^{\frac{1}{2}}\nonumber\\
&\leq&C\big\|(1-e^{-s^{m}L})^{M}G_{k,r_{B},M}(L)e^{itL}F(L)\chi_{U_{j}(B)}(r_{B}^{-mk}L^{M-k}b)\big\|_{2}\nonumber\\
&\leq&C\|r_{B}^{-mk}L^{M-k}b\|_{L^{2}(U_{j}(B))}
\leq  C 2^{-j \epsilon}\mu(2^{j}B)^{-(\frac{1}{q}-\frac{1}{2})}.
\end{eqnarray*}
Hence, it follows from the H\"{o}lder's inequality and  the doubling condition (\ref{doubling}) that
$$
\|E(k, j, s)\|_{L^{q}(64(1+|t|)2^{j} B)}
\leq \|E(k, j, s)\|_{L^{2}(64(1+|t|)2^{j} B)} \left(\frac{\mu(64(1+|t|)2^{j} B)}{\mu(2^{j}B)}\right)^{\frac{1}{q}-\frac{1}{2}}
\leq C 2^{-j \epsilon}(1+|t|)^{n(\frac{1}{q}-\frac{1}{2})}.
$$

Next we show that for some $\varepsilon'>0$,
\begin{align}\label{gggoal}
\|E(k, j, s)\|_{L^{q}((64(1+|t|)2^{j} B)^c)}   \leq  C2^{-j \varepsilon'}(1+|t|)^{n(\frac{1}{q}-\frac{1}{2})}.
\end{align}
To prove \eqref{gggoal}, we write
\begin{eqnarray*}
E(k, j, s)&=&\left(\int_0^{\infty}\!\!\!\!\int_{\substack{  d(x,y)<\tau^{1/m}}}
\big|e^{itL}F_{\tau,s}(L)\chi_{U_{j}(B)} (r_{B}^{-mk}L^{M-k} b)(y)\big|^2 {d\mu(y)\over V(x,\tau^{1/m})}{d\tau\over \tau}\right)^{1/2}\nonumber\\
&\leq& \sum_{\ell\in\ZZ}\left(\int_{2^{-\ell}}^{2^{-\ell+1}}\!\!\!\!\int_{\substack{  d(x,y)<\tau^{1/m}}}
\big|e^{itL}F_{\tau,s}(L)\chi_{U_{j}(B)} (r_{B}^{-mk}L^{M-k} b)(y)\big|^2 {d\mu(y)\over V(x,\tau^{1/m})}{d\tau\over \tau}\right)^{1/2}\nonumber\\
&=:& \sum_{\ell\in\ZZ} E(k, j, s, \ell).
\end{eqnarray*}
If $\ell> \frac{1}{m}$, then let $\nu_{0}^{+}\in\mathbb{Z_{+}}$ be a positive integer such that
\begin{align}\label{nu1}
8<2^{\nu_{0}^{+}+j-\ell(m-1)/m}r_{B}\leq 16,\ \ & {\rm if}\ 2^{j-\ell(m-1)/m}r_{B}\leq \frac{1}{8};\nonumber\\
\nu_{0}^{+}=7,\ \ & {\rm if}\ 2^{j-\ell(m-1)/m}r_{B}>\frac{1}{8}.
\end{align}
If $\ell\leq \frac{1}{m}$, then let $\nu_{0}^{-}\in\mathbb{Z_{+}}$ be a positive integer such that
\begin{align}\label{nu2}
8<2^{\nu_{0}^{-}+j+(\ell-1)/m}r_{B}\leq 16,\ \ & {\rm if}\ 2^{(\ell-1)/m+j}r_{B}\leq \frac{1}{8};\nonumber\\
\nu_{0}^{-}=7,\ \ & {\rm if}\ 2^{(\ell-1)/m+j}r_{B}>\frac{1}{8}.
\end{align}
Then{
\begin{eqnarray}\label{discuss}
 \|E(k, j, s)\|_{L^{q}((64(1+|t|)2^j B)^c)}
&\leq &\left(\sum_{\ell>1/m}\|E(k, j, s, \ell)\|_{L^{q}(B(x_B,8(1+|t|)2^{\ell(m-1)/m}))}^q\right)^{1/q}\nonumber\\
&&+
 \left(\sum_{\ell>1/m}\sum_{\nu\geq \nu^+_0} \|E(k, j, s, \ell)\|_{L^{q}({U_{\nu+j}((1+|t|)B)} )}^q\right)^{1/q}\nonumber\\
 &&+  \left(\sum_{\ell\leq 1/m}\|E(k, j, s, \ell)\|_{L^{q}(B(x_B,8(1+|t|)2^{-(\ell-1)/m}))}^q\right)^{1/q}\nonumber\\
 &&+\left(\sum_{\ell\leq 1/m}\sum_{\nu\geq \nu^-_0} \|E(k, j, s, \ell)\|_{L^{q}({U_{\nu+j}((1+|t|)B)} )}^q\right)^{1/q}\nonumber\\
&=:& {\rm I} (k, j, s) + {\rm II} (k, j, s) + {\rm III} (k, j, s) + {\rm IV} (k, j, s).
\end{eqnarray}}
Let us first estimate the terms ${\rm I} (k, j, s)$ and ${\rm II} (k, j, s)$.
Note that there is no term ${\rm I} (k, j, s)$ if $2^{j-\ell(m-1)/m}r_{B}> \frac{1}{8}$
and $\ell>\frac{1}{m}$. Besides, when $2^{j-\ell(m-1)/m}r_{B}\leq\frac{1}{8}$ and $\ell>\frac{1}{m}$,
we apply the estimate (\ref{molecule}) and the doubling condition (\ref{doubling}) to get that
\begin{eqnarray}\label{global}
 \|E(k, j, s, \ell)\|^2_{L^2(X)}
&\leq&C\int_{2^{-\ell}}^{2^{-\ell+1}}\big\|e^{itL}F_{\tau,s}(L)\chi_{U_{j}(B)} (r_{B}^{-mk}L^{M-k} b)\big\|_2^2{d\tau\over \tau}\nonumber\\
&\leq&C\int_{2^{-\ell}}^{2^{-\ell+1}} \|e^{it(\cdot)}F_{\tau,s}\|_{L^\infty}^2 \|r_{B}^{-mk}L^{M-k} b\|_{L^{2}(U_{j}(B))}^2{d\tau\over \tau}\nonumber\\
&\leq&C\min\{1, (2^{\ell/m} r_{B})^{2mM}\}2^{-2(\frac{1}{q}-\frac{1}{2})n\ell} 2^{-2j \epsilon} V(x_B,2^j r_{B})^{-2(\frac{1}{q}-\frac{1}{2})},
\end{eqnarray}
which, in combination with  the doubling condition (\ref{doubling}), yields that
\begin{eqnarray*}
{\rm I} (k, j, s)
&\leq& \left(\sum_{\ell\in\mathbb{Z}}\|E(k, j, s, \ell)\|_{L^{2}(B(x_B,8(1+|t|)2^{\ell(m-1)/m}))}^q V\big(x_B,8(1+|t|)2^{\ell(m-1)/m}\big)^{q(\frac{1}{q}-\frac{1}{2})}\right)^{1/q}\\
&\leq& C\left(\sum_{\ell\in\mathbb{Z}}\min\{1, (2^{\ell/m} r_{B})^{qmM}\} 2^{-q(\frac{1}{q}-\frac{1}{2})n\ell} 2^{-jq\epsilon}
\left(\frac{V(x_B,8(1+|t|)2^{\ell(m-1)/m})}{V(x_B,2^j r_{B})}\right)^{q(\frac{1}{q}-\frac{1}{2})}\right)^{1/q}\\
&\leq& C2^{-j (\epsilon+n(\frac{1}{q}-\frac{1}{2}))}\left(\sum_{\ell\in\mathbb{Z}}\min\{1, (2^{\ell/m} r_{B})^{qmM}\} (2^{\ell/m}r_{B})^{-q(\frac{1}{q}-\frac{1}{2})n}(1+|t|)^{qn(\frac{1}{q}-\frac{1}{2})}\right)^{1/q}\\
&\leq& C2^{-j (\epsilon+n(\frac{1}{q}-\frac{1}{2}))}(1+|t|)^{n(\frac{1}{q}-\frac{1}{2})}.
\end{eqnarray*}

Next we estimate the term ${\rm II} (k, j, s)$. It follows from (\ref{nu1}) that for $\tau\in[2^{-\ell},2^{-\ell+1}]$ and $\ell>\frac{1}{m}$, we have
$
\tau^{1/m}\leq 2^{1/m}2^{-\ell/m}\leq 2^{\ell(m-1)/m}\leq2^{\nu+j-3}(1+|t|)r_{B}.
$
Therefore, if $d(x,y)<\tau^{1/m}$ and $x\in U_{\nu+j}((1+|t|)B)$, then $y\in U_{\nu+j}^{\prime}((1+|t|)B)$, where
$$
U_{\nu+j}^{\prime}((1+|t|)B)=U_{\nu+j+1}((1+|t|)B)\cup U_{\nu+j}((1+|t|)B)\cup U_{\nu+j-1}((1+|t|)B).
$$
Then we have
\begin{eqnarray}\label{klkl}
 &&\hspace{-1.2cm}\|E(k, j, s, \ell)\|^2_{L^2({U_{\nu+j}((1+|t|)B)} )}\nonumber\\
&\leq&C\sum_{w=-1}^{1}\int_{2^{-\ell}}^{2^{-\ell+1}}\big\|\chi_{U_{\nu+j+w}((1+|t|)B)}e^{itL}F_{\tau,s}(L)\chi_{U_{j}(B)} (r_{B}^{-mk}L^{M-k} b)\big\|_{2}^2
{d\tau\over \tau}\nonumber\\
&\leq&C\sum_{w=-1}^{1}\int_{2^{-\ell}}^{2^{-\ell+1}} \big\|\chi_{U_{\nu+j+w}((1+|t|)B)}e^{itL}
F_{\tau,s}(L)\chi_{U_{j}(B)} \big\|_{2\rightarrow 2}^2 \|r_{B}^{-mk}L^{M-k} b\|_{L^{2}(U_{j}(B))}^{2}{d\tau\over \tau}.
\end{eqnarray}
To deal with the $L^{2}-L^{2}$ off-diagonal term involving the oscillatory semigroup $e^{itL}$,
we apply Proposition \ref{p2estimate2}, with $r_{B}$ replaced by $2^{j}r_{B}$ and $2^{j}$
replaced by $2^{\nu+j+w+v_{t}}$, where $v_{t}\in \mathbb{N}$ and $2^{v_{t}}\leq (1+|t|)<2^{v_{t}+1}$,
to see that there exist constants $C, c_{0}>0$ such that for $w\in \{-1,0,1\}$ and $\nu\geq 7$, $2^{-\ell}\leq\tau\leq 2^{-l+1}$,
\begin{align}\label{off222}
&\hspace{-1cm}\big\|\chi_{U_{\nu+j+w}((1+|t|)B)}e^{itL}F_{\tau,s}(L)\chi_{U_{j}(B)}\big\|_{2\rightarrow 2}\nonumber\\
\leq & C(2^{\nu}(1+|t|))^{-\kappa_{0}}(2^{\ell/m}2^{j}r_{B})^{-\kappa_{0}+\frac{n}{2}}(1+2^{\ell}|t|)^{\kappa_{0}}(2^{\ell/m}r_{j,\ell})^{-c_{0}}
\|\delta_{\tau^{-1}}F_{\tau,s}\|_{C^{\kappa_{0}+1}},
\end{align}
where $r_{j,\ell}={\rm min}\{2^{j}r_{B}, 2^{-\ell/m}\}$.
Note that the conditions supp$\phi\subset (1/4,1)$ and $r_{B}\leq s\leq \sqrt[m]{2}r_{B}$ imply that if $2^{-\ell}\leq\tau\leq 2^{-l+1}$, then
\begin{align}\label{llll}
\|\delta_{\tau^{-1}}F_{\tau,s}\|_{C^{\kappa_{0}+1}}\leq C\min\{1, (2^{\ell/m} r_{B})^{mM}\}2^{-\ell n (\frac{1}{q}-\frac{1}{2})}.
\end{align}
By a simple calculation, we can see that
\begin{align}\label{van1}
\min\{1, (2^{\ell/m} r_{B})^{mM}\}(2^{\ell/m}r_{j,\ell})^{-c_{0}}\leq C\min\{1, (2^{\ell/m} r_{B})^{mM-c_{0}}\}.
\end{align}
This, in combination with the estimates (\ref{molecule}), (\ref{klkl}), (\ref{off222}) and  (\ref{llll}), yields
\begin{eqnarray*}
&&\hspace{-1.2cm}\|E(k, j, s, \ell)\|_{L^2({U_{\nu+j}((1+|t|)B)} )}\\
&\leq &C(2^{\nu}(1+|t|))^{-\kappa_{0}}(2^{\ell/m}2^{j}r_{B})^{-\kappa_{0}+\frac{n}{2}}(1+2^{\ell}|t|)^{\kappa_{0}}
\min\{1, (2^{\ell/m} r_{B})^{mM-c_{0}}\}2^{-\ell n (\frac{1}{q}-\frac{1}{2})}2^{-j\epsilon }V(x_{B},2^{j}r_{B})^{-(\frac{1}{q}-\frac{1}{2})}.
\end{eqnarray*}
This, together the doubling condition (\ref{doubling}) and the definition of $\nu_{0}^{+}$, indicates that
\begin{align*}
&{\rm II} (k, j, s)\\
 &\leq  \left(\sum_{\ell>1/m}\sum_{\nu\geq \nu^+_0} \|E(k, j, s, \ell)\|_{L^{2}({U_{\nu+j}((1+|t|)B)} )}^qV(x_{B},2^{\nu+j}(1+|t|)r_{B})^{q(\frac{1}{q}-\frac{1}{2})}\right)^{1/q}\\
&\leq C2^{-j(\kappa_{0}-\frac{n}{2}+\epsilon) }(1+|t|)^{n(\frac{1}{q}-\frac{1}{2})-\kappa_{0}}\\
&\ \ \ \ \ \times\left(\sum_{\ell>1/m}
\sum_{\nu\geq \nu^+_0} 2^{-q\nu(\kappa_{0}-n(\frac{1}{q}-\frac{1}{2}))}(2^{\ell/m}r_{B})^{-q(\kappa_{0}-\frac{n}{2})}(1+2^{\ell}|t|)^{q\kappa_{0}}\min\{1, (2^{\ell/m} r_{B})^{qmM-qc_{0}}\}2^{-\ell nq(\frac{1}{q}-\frac{1}{2})}\right)^{1/q}\\
&\leq C2^{-j (n(\frac{1}{q}-\frac{1}{2})-\frac{n}{2}+\epsilon )}(1+|t|)^{n(\frac{1}{q}-\frac{1}{2})}\sum_{\ell>1/m} (2^{\ell/m}r_B)^{-q(n(\frac{1}{q}-\frac{1}{2})-\frac{n}{2} )}\min\{1,(2^{\ell/m}r_B)^{qmM-qc_0}\}\\
&\leq C2^{-j (n(\frac{1}{q}-\frac{1}{2})-\frac{n}{2}+\epsilon )}(1+|t|)^{n(\frac{1}{q}-\frac{1}{2})},
\end{align*}
where in the last two inequalities we used the condition that $\frac{2n}{2\kappa_0+n}<q<1$.

Consider the terms ${\rm III} (k, j, s)$ and ${\rm IV} (k, j, s)$.
Note that there is no term ${\rm III} (k, j, s)$ if $2^{(\ell-1)/m+j}r_{B}>\frac{1}{8}$ and $\ell\leq 1/m$. Therefore,
 when $2^{(\ell-1)/m+j}r_{B}\leq\frac{1}{8}$ and $\ell\leq 1/m$, similar to the proof of (\ref{global}), we obtain that
\begin{eqnarray*}
 \|E(k, j, s, \ell)\|^2_{L^2(X)}
&\leq&C\int_{2^{-\ell}}^{2^{-\ell+1}} \|e^{it(\cdot)}F_{\tau,s}\|_{L^\infty}^2 \|r_{B}^{-mk}L^{M-k} b\|_{L^{2}(U_{j}(B))}^2{d\tau\over \tau}\nonumber\\
&\leq&C\int_{2^{-\ell}}^{2^{-\ell+1}}\min\{1, (\tau^{-1/m} r_{B})^{2mM}\} 2^{-2j \epsilon} V(x_B,2^j r_{B})^{-2(\frac{1}{q}-\frac{1}{2})}{d\tau\over \tau}\nonumber\\
&\leq&C\min\{1, (2^{\ell/m} r_{B})^{2mM}\} 2^{-2j \epsilon} V(x_B,2^j r_{B})^{-2(\frac{1}{q}-\frac{1}{2})},
\end{eqnarray*}
which gives
\begin{eqnarray*}
{\rm III} (k, j, s)
&\leq& \left(\sum_{\ell\leq 1/m}\|E(k,j,s,\ell)\|_{L^{2}(B(x_{B},8(1+|t|)2^{-(\ell-1)/m}))}^qV(x_{B},8(1+|t|)2^{-(\ell-1)/m})^{q(\frac{1}{q}-\frac{1}{2})}\right)^{1/q}\\
&\leq&  C2^{-j(\epsilon+(\frac{1}{q}-\frac{1}{2})n))}\left(\sum_{\ell\leq 1/m}{\rm min}\{1,(2^{\ell/m}r_{B})^{qmM}\}(2^{\ell/m}r_{B})^{-q(\frac{1}{q}-\frac{1}{2})n}(1+|t|)^{qn(\frac{1}{q}-\frac{1}{2})}\right)^{1/q}\\
&\leq&  C2^{-j(\epsilon+(\frac{1}{q}-\frac{1}{2})n)}(1+|t|)^{n(\frac{1}{q}-\frac{1}{2})}.
\end{eqnarray*}
To estimate the term ${\rm IV} (k, j, s)$, we first note that it follows from (\ref{nu2}) that for $\tau\in [2^{-\ell},2^{-\ell+1}]$, we have
$
\tau^{1/m}\leq 2^{1/m}2^{-\ell/m}\leq 2^{\nu+j-2}(1+|t|)r_{B}.
$
Hence, if $d(x,y)<\tau^{1/m}$ and $x\in U_{\nu+j}((1+|t|)B)$, then $y\in U_{\nu+j}^{\prime}((1+|t|)B)$, where
\begin{align*}
U_{\nu+j}^{\prime}((1+|t|)B):=U_{\nu+j+1}((1+|t|)B)\cup U_{\nu+j}((1+|t|)B)\cup U_{\nu+j-1}((1+|t|)B).
\end{align*}
We write
\begin{eqnarray*}\label{pri0}
&&\hspace{-1cm}\|E(k, j, s, \ell)\|^2_{L^2({U_{\nu+j}((1+|t|)B)} )}\nonumber \\
&\leq&C\sum_{w=-1}^{1}\int_{2^{-\ell}}^{2^{-\ell+1}}\big\|\chi_{U_{\nu+j+w}((1+|t|)B)}e^{itL}F_{\tau,s}(L)\chi_{U_{j}(B)} (r_{B}^{-mk}L^{M-k} b)\big\|_{2}^2
{d\tau\over \tau}\nonumber\\
&\leq&C\sum_{w=-1}^{1}\int_{2^{-\ell}}^{2^{-\ell+1}}
\big\|\chi_{U_{\nu+j+w}((1+|t|)B)}e^{itL}F_{\tau,s}(L)\chi_{U_{j}(B)}
\big\|_{2\rightarrow 2}^2 \|r_{B}^{-mk}L^{M-k} b\|_{L^{2}(U_{j}(B))}^{2}{d\tau\over \tau}.
\end{eqnarray*}
Note that the conditions supp$\phi\subset (1/4,1)$ and $r_{B}\leq s\leq \sqrt[m]{2}r_{B}$ implies if $2^{-\ell}\leq\tau\leq 2^{-l+1}$, then
\begin{align}
\|\delta_{\tau^{-1}}F_{\tau,s}\|_{C^{\kappa_{0}+1}}\leq C\min\{1, (2^{\ell/m} r_{B})^{mM}\}.
\end{align}
This, together with the estimate (\ref{van1}), implies
\begin{align*}
\big\|\chi_{U_{\nu+j+w}((1+|t|)B)}e^{itL}F_{\tau,s}(L)\chi_{U_{j}(B)}\big\|_{2\rightarrow 2}
\leq C(2^{\nu}(1+|t|))^{-\kappa_{0}}(2^{\ell/m}2^{j}r_{B})^{-\kappa_{0}+\frac{n}{2}}(1+2^{\ell}|t|)^{\kappa_{0}}\min\{1, (2^{\ell/m} r_{B})^{mM-c_{0}}\},
\end{align*}
and thus
\begin{eqnarray*}
&&\hspace{-1.2cm}\|E(k, j, s, \ell)\|_{L^2({U_{\nu+j}((1+|t|)B)} )}\\
&\leq &C(2^{\nu}(1+|t|))^{-\kappa_{0}}(2^{\ell/m}2^{j}r_{B})^{-\kappa_{0}+\frac{n}{2}}(1+2^{\ell}|t|)^{\kappa_{0}}
\min\{1, (2^{\ell/m} r_{B})^{mM-c_{0}}\}2^{-j\epsilon }V(x_{B},2^{j}r_{B})^{-(\frac{1}{q}-\frac{1}{2})}.
\end{eqnarray*}
Hence,
\begin{align*}
&{\rm IV} (k, j, s)\\
&\leq \left(\sum_{\ell\leq 1/m}\sum_{\nu\geq \nu^-_0} \|E(k, j, s, \ell)\|_{L^2({U_{\nu+j}((1+|t|)B)} )}^qV(x_{B},2^{\nu+j}((1+|t|)B))^{q(\frac{1}{q}-\frac{1}{2})}\right)^{1/q}\\
&\leq 2^{-jq(\kappa_0-\frac{n}{2}+\epsilon)}(1+|t|)^{n(\frac{1}{q}-\frac{1}{2})}\left(\sum_{\ell\leq 1/m}\sum_{\nu\geq \nu^-_0}2^{-q\nu(\kappa_0-n(\frac{1}{q}-\frac{1}{2}))}(2^{\ell/m}r_B)^{-q(\kappa_0-\frac{n}{2})}\min\{1, (2^{\ell/m} r_{B})^{qmM-qc_{0}}\} \right)^{1/q}\\
&\leq C2^{-j((\frac{1}{q}-\frac{1}{2})n-\frac{n}{2}+\epsilon)}(1+|t|)^{n(\frac{1}{q}-\frac{1}{2})}\left(\sum_{\ell\leq 1/m}(2^{\ell/m}r_{B})^{-q((\frac{1}{q}-\frac{1}{2})n-\frac{n}{2})}\min\{1, (2^{\ell/m} r_{B})^{qmM-qc_{0}}\}\right)^{1/q}\\
&\leq C2^{-j((\frac{1}{q}-\frac{1}{2})n-\frac{n}{2}+\epsilon)}(1+|t|)^{n(\frac{1}{q}-\frac{1}{2})},
\end{align*}
where in the last two inequalities we used the condition $\frac{2n}{2\kappa_0+n}<q<1$.

Combining the estimates for ${\rm I} (k, j, s)$, ${\rm II} (k, j, s)$,
${\rm III} (k, j, s)$ and ${\rm IV} (k, j, s)$, we obtain the estimate (\ref{gggoal}) and therefore,
\begin{eqnarray*}
  \sum_{k=0}^{M-1} \left\|\left(\int_0^{\infty}\!\!\!\!\int_{\substack{  d(x,y)<\tau^{1/m}}}
\big|e^{itL}E_k(y)\big|^2 {d\mu(y)\over V(x,\tau^{1/m})}{d\tau\over \tau}\right)^{1/2}\right\|_{L^{q}}
&\leq&  C\sum_{k=0}^{M-1} \left(\sum_{j\geq 0}\sup_{s\in [r_{B},\sqrt[m] 2 r_{B}]} \|E(k, j, s)\|_{L^{q}(X)}^q\right)^{1/q}\\
   &\leq & C\left(\sum_{j\geq 0}2^{-j q\varepsilon'}\right)^{1/q}(1+t)^{n(\frac{1}{q}-\frac{1}{2})} \\
   &\leq& C (1+t)^{n(\frac{1}{q}-\frac{1}{2})}.
\end{eqnarray*}

\medskip

\noindent
{\bf Case 2.} \ $k=M$. \ Similarly to the proof of estimating $e^{itL}E_{k}$ for $k=1,2,\cdots,M-1$ as in {\bf Case 1}, we conclude that
\begin{align*}
\left\|\left(\int_{0}^{\infty}\int_{d(x,y)<\tau^{1/m}}\big|e^{itL}E_{M}(y)\big|^{2}
\frac{d\mu(y)}{V(x,\tau^{1/m})}\frac{d\tau}{\tau}\right)^{\frac{1}{2}}\right\|_{L^{q}}
\leq C(1+|t|)^{n(\frac{1}{q}-\frac{1}{2})}.
\end{align*}
This finishes the proof of (\ref{goalsch}) and then Proposition \ref{tool}.
\end{proof}

\subsection{Proof of boundedness on $L^{p}(X)$}
\setcounter{equation}{0}
This subsection will show how to apply the $H_{L}^{q}-H_{L}^{q}$ ($\frac{2n}{2\kappa_0+n}<q<1$) boundedness for Schr\"dinger
groups and the complex interpolation method we obtain in the appendix to get the $L^{p}$ boundedness for Schr\"odinger groups.

\begin{proof}[Proof of   Theorem~\ref{main}]
Inspired by \cite{CDLY2}, we consider the analytic family of operators
\begin{align*}
T_{z}:=e^{(1-z)^{2}}(1+|t|)^{-(1-z)(\frac{1}{q}-\frac{1}{2})n}(I+L)^{-(1-z)(\frac{1}{q}-\frac{1}{2})n}e^{itL},\ \ 0\leq {\rm Re}z\leq 1.
\end{align*}
Then $T_{z}$ is a holomorphic function of $z$ in the sense that
\begin{align*}
z\rightarrow \int_{X}T_{z}f(x)g(x)d\mu(x)
\end{align*}
for $f,g\in L^{2}(X)$.

By the spectral theorem,
\begin{align*}
\|T_{1+iy}f\|_{2}=e^{-y^{2}}\|(I+L)^{iy(\frac{1}{q}-\frac{1}{2})n}e^{itL}\|_{2}\leq C\|f\|_{2}.
\end{align*}

Besides, it follows from the Theorem \ref{spectralmultiplier} that for any $\frac{2n}{2\kappa_0+n}<q<1$,
\begin{align*}
\|(I+L)^{iy(\frac{1}{q}-\frac{1}{2})n}\|_{H_{L}^{q}\rightarrow H_{L}^{q}}\leq C(1+|y|)^{\kappa_{0}+1}.
\end{align*}
{This, in combination with Proposition \ref{tool} (note that the argument here also deduce the $H_L^q\rightarrow L^q$ boundedness of the same operator), yields
\begin{align*}
\|T_{iy}f\|_{L^q}&=e^{1-y^{2}}(1+|t|)^{-n(\frac{1}{q}-\frac{1}{2})}\|e^{itL}(I+L)^{-n(\frac{1}{q}-\frac{1}{2})}(I+L)^{iy(\frac{1}{q}-\frac{1}{2})n}f\|_{L^q}\\
&\leq Ce^{1-y^{2}}\|(I+L)^{iy(\frac{1}{q}-\frac{1}{2})n}f\|_{H_{L}^{q}}\\
&\leq C\|f\|_{H_{L}^{q}}.
\end{align*}
By complex interpolation between ${\rm Re} z=0$ and ${\rm Re} z=1$, we obtain that
for $\theta\in (0,1)$ and $p\in (q,2)$,
\begin{align*}
\|T_{\theta}(f)\|_{H_L^p}\leq C\|f\|_{[H_{L}^{q},H_{L}^{2}]_{\theta}}\leq C\|f\|_{H_{L}^{p}}.
\end{align*}
where the parameter $\theta$ satisfies $\frac{1}{p}=\frac{1-\theta}{q}-\frac{\theta}{2}$.
This, together with Lemma \ref{Hp}, shows that for any $p_{0}<p\leq 2$,
\begin{align*}
\|(I+L)^{-\sigma_{p}n}e^{itL}f\|_{p}=\|e^{-(1-\theta)^{2}}(1+|t|)^{(1-\theta)(\frac{1}{q}-\frac{1}{2})n}T_{\theta}f\|_{p}\leq C(1+|t|)^{\sigma_{p} n}\|f\|_{p}.
\end{align*}}
By duality, we can obtain the corresponding results for $2\leq p<p_{0}^{\prime}$. Then the $L^{p}$ boundedness for Schr\"{o}dinger groups is proven.
\end{proof}

\medskip

\section{ Results on  Hardy space $H_{L}^{q}(X):$  Proof of Lemmas~ \ref{moleculardecomposition},
\ref{cominter} and \ref{Hp} }

In this section, with the help of the estimate \eqref{p2estimate}, we will give the proof of Lemmas \ref{moleculardecomposition},
\ref{cominter}, \ref{Hp} and establish the spectral theorem on $H_{L}^{q}$ for $\frac{2n}{2\kappa_0+n}<q\leq1$.

\subsection{Tent space}
We recall some preliminaries on tent spaces of homogeneous type.

Let $F$ be a measurable function defined on $X\times (0,\infty)$. Denote
\begin{align*}
\mathcal{A}(F)(x)=\left(\int_{0}^{\infty}\int_{d(x,y)<\tau}|F(y,\tau)|^{2}\frac{d\mu(y)}{V(x,t)}\frac{d\tau}{\tau}\right)^{\frac{1}{2}}.
\end{align*}
Following \cite{CMS}, for any $0< p<\infty$, the tent space $T_{2}^{p}(X)$ is defined as the space of measurable
functions $F$ on $X\times (0,\infty)$ such that $\mathcal{A}(F)\in L^{p}(X)$, equipped with the (quasi-)norm:
\begin{align*}
\|F\|_{T_{2}^{p}(X)}=\|\mathcal{A}(F)\|_{L^{p}(X)}.
\end{align*}
Next, we recall the atomic decomposition theory for tent spaces, which was originally studied in \cite{CMS}.
\begin{definition}
{Given $0<q\leq 1$}, a measurable function $A(x,\tau)$ on $X\times (0,\infty)$ is said to be a $T_{2}^{q}$-atom if there exists a ball
 $B\subset X$ such that $A$ is supported in $\hat{B}$ and satisfies
\begin{align*}
\left(\int_{0}^{\infty}\int_{X}|A(x,\tau)|^{2}\frac{d\mu(x)d\tau}{\tau}\right)^{\frac{1}{2}}\leq \mu(B)^{-(\frac{1}{q}-\frac{1}{2})}.
\end{align*}
\end{definition}
The following atomic decomposition theorem for tent spaces was obtained in \cite{Russ}.
\begin{lemma}\label{tentde}
{Given $0<q\leq 1$, then for every $F\in T_{2}^{q}(X)$ there exists a constant $C>0$, a sequence $\{\lambda_{j}\}_{j=0}^{\infty}\in\ell^{q}$ and
 a sequence of $T_{2}^{q}$-atoms $\{A_{j}\}_{j=0}^{\infty}$ such that
\begin{align*}
F=\sum_{j=0}^{\infty}\lambda_{j}A_{j}\ \ in\ T_{2}^{q}(X)\ a.e.\ in\ X\times(0,\infty)
\end{align*}
and
\begin{align*}
\left(\sum_{j=0}^{\infty}|\lambda_{j}|^q\right)^{1/q}\leq C\|F\|_{T_{2}^{q}(X)}.
\end{align*}
In addition, if $F\in T_{2}^{q}(X)\cap T_{2}^{2}(X)$, then the summation also converges in $T_{2}^{2}(X)$.}
\end{lemma}
\subsection{Molecular decompositions for Hardy spaces}
This subsection is devoted to giving molecular decompositions for Hardy spaces $H_{L}^{q}$ for $\frac{2n}{2\kappa_0+n}<q\leq1$. To begin with,
we consider the operator: $\pi_{L}:T_{2}^{2}(X)\rightarrow L^{2}(X)$, given by
\begin{align*}
\pi_{L}(F)(x):=\int_{0}^{\infty}\phi(\tau^{m}L)(F(\cdot,\tau))(x)\frac{d\tau}{\tau},
\end{align*}
 where the improper integral converges weakly in $L^{2}(X)$.  By duality and the
boundedness of square function, it is not difficult to see that $\pi_{L}$ is bounded from $T_{2}^{2}(X)$ to $L^{2}(X)$.

\begin{lemma}\label{tentatom}
Given $\frac{2n}{2\kappa_0+n}<q\leq1$, for any $T_{2}^{q}(X)$-atom $A(y,\tau)$ associated to a ball $B$(or more precisely, to its tent $\hat{B}$), there is a
 uniform constant $C>0$ such that $C^{-1}\pi_{L}(A)$ is a $(q,2,M,\epsilon)$-molecule associated to $B$ for some $\epsilon>0$.
\end{lemma}
\begin{proof}

By the definition of $T_{2}^{q}(X)$-atom,
\begin{align*}
\left(\int_{X\times (0,\infty)}|A(y,\tau)|^{2}\frac{d\mu(y)d\tau}{\tau}\right)^{\frac{1}{2}}\leq \mu(B)^{-(\frac{1}{q}-\frac{1}{2})}.
\end{align*}
We write $a=\pi_{L}(A)=L^{M}b$, where $b(x)=\int_{0}^{\infty}L^{-M}\phi(\tau^{m}L)(A(\cdot,\tau))(x)\frac{d\tau}{\tau}$. Next
for any $\ell\geq 0$, $k=0,1,\ldots,M$, we estimate $\|(r_{B}^{m}L)^{k}b\|_{L^{2}(U_{\ell}(B))}$ by duality.
Consider $h\in L^{2}(U_{\ell}(B))$ such that $\|h\|_{L^{2}(U_{\ell}(B))}=1$. Then since $L$ is self-adjoint, we can
apply Proposition \ref{p2estimate} to conclude that when $\ell \geq 5$,
\begin{align*}
\left|\int_{X}(r_{B}^{m}L)^{k}b(x)h(x)d\mu(x)\right|
&\leq r_{B}^{mk}\int_{\hat{B}}|A(y,\tau)||L^{k-M}\phi(\tau^{m}L)h(y)|\frac{d\mu(y)d\tau}{\tau}\\
&\leq r_{B}^{mk}\|A\|_{T_{2}^{2}(X)}\left(\int_{\hat{B}}|(\tau^{m}L)^{k-M}\phi(\tau^{m}L)h(y)|^{2}\frac{d\mu(y)d\tau}{\tau^{2m(k-M)+1}}\right)^{\frac{1}{2}}\\
&\leq r_{B}^{mk} \mu(B)^{-(\frac{1}{q}-\frac{1}{2})}\left(\int_{0}^{r_{B}}\|\chi_{B}(\tau^{m}L)^{k-M}\phi(\tau^{m}L)\chi_{U_{\ell}(B)}h\|_{2}^{2}
\frac{d\tau}{\tau^{2m(k-M)+1}}\right)^{\frac{1}{2}}\\
&\leq C2^{-\ell\kappa_0}r_{B}^{mk} \mu(B)^{-(\frac{1}{q}-\frac{1}{2})}\left(\int_{0}^{r_{B}}\left(\frac{r_{B}}{\tau}\right)^{-2\kappa_{0}+n}
\frac{d\tau}{\tau^{2m(k-M)+1}}\right)^{\frac{1}{2}}\\
&\leq C2^{-\ell(\kappa_{0}-n(\frac{1}{q}-\frac{1}{2}))}r_{B}^{mM}\mu(2^{\ell}B)^{-(\frac{1}{q}-\frac{1}{2})}.
\end{align*}
Taking supremum over all $h\in L^{2}(U_{\ell}(B))$ satisfying $\|h\|_{L^{2}(U_{\ell}(B))}=1$ and choosing
$0<\epsilon<\kappa_{0}-n(\frac{1}{q}-\frac{1}{2})$, we obtain that for any $\ell\geq 0$, $k=0,1,\ldots,M$,
\begin{align*}
\|(r_{B}^{m}L)^{k}b\|_{L^{2}(U_{\ell}(B))}\leq C2^{-\ell(\kappa_{0}-n(\frac{1}{q}-\frac{1}{2}))}r_{B}^{mM}\mu(2^{\ell}B)^{-(\frac{1}{q}-\frac{1}{2})}
\leq C2^{-\ell\epsilon}r_{B}^{mM}\mu(2^{\ell}B)^{-(\frac{1}{q}-\frac{1}{2})}.
\end{align*}
This implies that $C^{-1}\pi_{L}(A)$ is a $(q,2,M,\epsilon)$-molecule.
\end{proof}

\begin{proof}[Proof of Lemma \ref{moleculardecomposition}]
Set $\mathcal{H}^{2}(X)=\overline{\{Lu\in L^{2}(X):u\in L^{2}(X)\}}$.
Recall that $H_{L,mol,M,\epsilon}^{q}(X)$  and $H_{L}^{q}(X)$ are the completions of $\mathbb{H}_{L,mol,M,\epsilon}^{q}(X)$
and $H_{L}^{q}(X)\cap \mathcal{H}^{2}(X)$, respectively. It suffices to show that $H_{L,mol,M,\epsilon}^{q}(X)$ and $H_{L}^{q}(X)$
have the same dense subset $\mathbb{H}_{L,mol,M,\epsilon}^{q}(X)=H_{L}^{q}(X)\cap \mathcal{H}^{2}(X)$ with equivalent norms.

\medskip

\noindent
\textbf{Step I}: $\mathbb{H}_{L,mol,M,\epsilon}^{q}(X)\subset (H_{L}^{q}(X)\cap \mathcal{H}^{2}(X))$.

By definition,   $\mathbb{H}_{L,mol,M,\epsilon}^{q}(X)\subset \mathcal{H}^{2}(X)$. Therefore, by
 a standard density argument, it will be enough to show that for every $(q,2,M,\epsilon)$-molecule $a(x)$ associated to a ball
 $B=B(x_{B},r_{B})$ of $X$, we have
\begin{align*}
\|S_{L,\phi}(a)\|_{L^{q}(X)}\leq C.
\end{align*}
Denote $F(y,\tau)=\phi(\tau^{m}L)a$. By a simple change of variable, it is enough to show that
\begin{align}\label{tent}
\|F\|_{T_{2}^{q}(X)}\leq C.
\end{align}
Now let $\eta_{0}=\chi_{2B\times (0,2r_{B})}$ and for all $j\geq 1$, define $\eta_{j}=\chi_{U_{j+1}(B)\times (0,r_{B})}$,
$\eta_{j}^{\prime}=\chi_{U_{j+1}(B)\times (r_{B},2^{j+1}r_{B})}$ and $\eta_{j}^{\prime\prime}=\chi_{2^{j}B\times(2^{j}r_{B},2^{j+1}r_{B})}$.
Then we decompose $F$ as follows.
\begin{align*}
F=\eta_{0}F+\sum_{j=1}^{\infty}\eta_{j}F+\sum_{j=1}^{\infty}\eta_{j}^{\prime}F+\sum_{j=1}^{\infty}\eta_{j}^{\prime\prime}F.
\end{align*}

Next, we will show that there exist constants  $C,\sigma>0$, such that

\smallskip

\noindent
(a) For any $j\geq 0$, $\|\eta_{j}F\|_{T_{2}^{2}(X)}\leq C2^{-j\sigma}\mu(2^{j}B)^{-(\frac{1}{q}-\frac{1}{2})}$;

\smallskip

\noindent
(b) For any $j\geq 1$, $\|\eta_{j}^{\prime}F\|_{T_{2}^{2}(X)}\leq C2^{-j\sigma}\mu(2^{j}B)^{-(\frac{1}{q}-\frac{1}{2})}$;

\smallskip

\noindent
(c) For any $j\geq 1$, $\|\eta_{j}^{\prime\prime}F\|_{T_{2}^{2}(X)}\leq C2^{-j\sigma}\mu(2^{j}B)^{-(\frac{1}{q}-\frac{1}{2})}$.

\smallskip

\noindent
Since each $\eta_{j}F$, $\eta_{j}^{\prime}F$, $\eta_{j}^{\prime\prime}F$ are supported in $\widehat{2^{j+2}B}$, these three
estimates will imply that $\frac{1}{C}2^{j\sigma}\eta_{j}F$, $\frac{1}{C}2^{j\sigma}\eta_{j}^{\prime}F$ and
 $\frac{1}{C}2^{j\sigma}\eta_{j}^{\prime\prime}F$ are atoms in $T_{2}^{q}(X)$, respectively, and thus the estimate (\ref{tent}) will be done.

Now we show the estimates (a),(b),(c). To show (a), we first apply the $L^{2}$ boundedness of the square function to obtain that
\begin{align*}
\|\eta_{0}F\|_{T_{2}^{2}(X)}\leq C\left(\int_{0}^{\infty}\int_{X}|\phi(\tau^{m}L)a(y)|^{2}\frac{d\mu(y)d\tau}{\tau}\right)^{\frac{1}{2}}
\leq C\|a\|_{L^{2}(X)}\leq C\mu(B)^{-(\frac{1}{q}-\frac{1}{2})}.
\end{align*}
For $j\geq 1$, we apply the formula (\ref{iterate}) to obtain that
\begin{eqnarray}\label{il0}
 \|\chi_{U_{j+1}(B)}\phi(\tau^{m}L)a\|_{2}
&\leq& \sum_{\ell=0}^{\infty}  {\rm I}(\ell,\tau) +  \sum_{\ell=0}^{\infty}  {\rm II}(\ell,\tau),
\end{eqnarray}
where
$$
{\rm I}(\ell,\tau) = \sum_{k=0}^{M-1}r_{B}^{-m}\int_{r_{B}}^{\sqrt[m]{2}r_{B}}s^{m-1}
\|\chi_{U_{j+1}(B)}(1-e^{-s^{m}L})^{M}G_{k,r_{B},M}(L)\phi(\tau^{m}L)\chi_{U_{\ell}(B)}(r_{B}^{-mk}L^{M-k}b)\|_{2}ds
$$
and
$$
{\rm II}(\ell,\tau) = \sum_{\nu=1}^{(2M-1)M}\|\chi_{U_{j+1}(B)}(1-e^{-s^{m}L})^{M}e^{-\nu r^{m}L}\phi(\tau^{m}L)
\chi_{U_{\ell}(B)}(r_{B}^{-mM}b)\|_{2}.
$$

We apply Proposition \ref{p2estimate} to conclude that when $\ell \leq j-5$,
\begin{align*}
{\rm I}(\ell,\tau) &\leq C\sum_{k=0}^{M-1}\sup\limits_{s\in [r_{B},\sqrt[m]{2}r_{B}]}
\|\chi_{U_{j+1}(B)}(1-e^{-s^{m}L})^{M}G_{k,r_{B},M}(L)\phi(\tau^{m}L)\chi_{U_{\ell}(B)}(r_{B}^{-mk}L^{M-k}b)\|_{2}\\
&\leq C2^{-j\kappa_0}\left(\frac{r_{B}}{\tau}\right)^{-\kappa_{0}+\frac{n}{2}}\|r_{B}^{-mk}L^{M-k}b\|_{L^{2}(U_{\ell}(B))}
\leq C2^{-\ell\epsilon}2^{-j\kappa_0}\left(\frac{r_{B}}{\tau}\right)^{-\kappa_{0}+\frac{n}{2}}\mu(2^{\ell}B)^{-(\frac{1}{q}-\frac{1}{2})}\\
&\leq C2^{-\ell\epsilon}2^{-j\kappa_0}\left(\frac{r_{B}}{\tau}\right)^{-\kappa_{0}+\frac{n}{2}}2^{(j-\ell)n(\frac{1}{q}-\frac{1}{2})}\mu(2^{j}B)^{-(\frac{1}{q}-\frac{1}{2})},
\end{align*}
and thus
\begin{align}\label{il1}
\sum_{\ell=0}^{j-5}\left(\int_{0}^{r_{B}}\big| {\rm I}(\ell,\tau)\big|^2 \frac{d\tau}{\tau}\right)^{\frac{1}{2}}
 \leq C\sum_{\ell=0}^{j-5}2^{-\ell(\epsilon+(\frac{1}{q}-\frac{1}{2})n)}2^{-j(\kappa_{0}-(\frac{1}{q}-\frac{1}{2})n)}\mu(2^{j}B)^{-(\frac{1}{q}-\frac{1}{2})}
 \leq C2^{-j(\kappa_{0}-(\frac{1}{q}-\frac{1}{2})n)}\mu(2^{j}B)^{-(\frac{1}{q}-\frac{1}{2})}.
\end{align}
Besides, the self-adjoint property of the operator $L$ allows us to apply Proposition \ref{p2estimate} to conclude that when $j\leq \ell-5$,
\begin{align*}
{\rm I}(\ell,\tau) &\leq C\sum_{k=0}^{M-1}\sup\limits_{s\in [r_{B},\sqrt[m]{2}r_{B}]}
\|\chi_{U_{j+1}(B)}(1-e^{-s^{m}L})^{M}G_{k,r_{B},M}(L)\phi(\tau^{m}L)\chi_{U_{\ell}(B)}(r_{B}^{-mk}L^{M-k}b)\|_{2}\\
&\leq C2^{-\ell\kappa_0}\left(\frac{r_{B}}{\tau}\right)^{-\kappa_{0}+\frac{n}{2}}\|r_{B}^{-mk}L^{M-k}b\|_{L^{2}(U_{\ell}(B))}\\
&\leq C2^{-\ell(\kappa_0+\epsilon)}\left(\frac{r_{B}}{\tau}\right)^{-\kappa_{0}+\frac{n}{2}}2^{(j-\ell)n(\frac{1}{q}-\frac{1}{2})}\mu(2^{j}B)^{-(\frac{1}{q}-\frac{1}{2})},
\end{align*}
and therefore
\begin{align}\label{il2}
\sum_{\ell=j+5}^{\infty}\left(\int_{0}^{r_{B}}\big| {\rm I}(\ell,\tau)\big|^2 \frac{d\tau}{\tau}\right)^{\frac{1}{2}}
&\leq C\sum_{\ell=j+5}^{\infty}2^{-\ell(\kappa_0+\epsilon)}2^{(j-\ell)n(\frac{1}{q}-\frac{1}{2})}\mu(2^{j}B)^{-(\frac{1}{q}-\frac{1}{2})}
\left(\int_{0}^{r_{B}}\left(\frac{r_{B}}{\tau}\right)^{-2\kappa_{0}+n}\frac{d\tau}{\tau}\right)^{\frac{1}{2}}\nonumber\\
&\leq C2^{-j(\kappa_{0}+\epsilon)}\mu(2^{j}B)^{-(\frac{1}{q}-\frac{1}{2})}.
\end{align}
Finally, by the $L^{2}$ boundedness of the square function, we have
\begin{eqnarray}\label{il3}
 \sum_{\ell=j-4}^{j+4}\left(\int_{0}^{r_{B}} \big| {\rm I}(\ell,\tau)\big|^2 \frac{d\tau}{\tau}\right)^{\frac{1}{2}}
&\leq & C\sum_{\ell=j-4}^{j+4}\left(\int_{0}^{\infty}\|\phi(\tau^{m}L)\chi_{U_{\ell}(B)}(r_{B}^{-mk}L^{M-k}b)\|_{2}^{2}
\frac{d\tau}{\tau}\right)^{\frac{1}{2}}\nonumber\\
&\leq & C\sum_{\ell=j-4}^{j+4}\|r_{B}^{-mk}L^{M-k}b\|_{L^{2}(U_{\ell}(B))}\nonumber\\
&\leq & C2^{-j\epsilon}\mu(2^{j}B)^{-(\frac{1}{q}-\frac{1}{2})}.
\end{eqnarray}
The term ${\rm II}(\ell,\tau)$ can be handled in a similar way. This, in combination with the estimates
 (\ref{il0}), (\ref{il1}), (\ref{il2}) and (\ref{il3}), imply that for any $j\geq 1$,
\begin{align*}
\|\eta_{j}F\|_{T_{2}^{2}}
\leq C\left(\int_{0}^{r_{B}}\|\chi_{U_{j+1}(B)}\phi(\tau^{m}L)a\|_{2}^{2}\frac{d\tau}{\tau}\right)^{\frac{1}{2}}
\leq C2^{-j\sigma}\mu(2^{j}B)^{-(\frac{1}{q}-\frac{1}{2})}
\end{align*}
for some $\sigma>0$.

Next, we turn to show the estimate (b).
We decompose $M=M_{0}+M_{1}$, where $M_{0}, M_{1}$ are two constants to be chosen large enough
later. Then, we iterate the formula (\ref{iter0}) $M_{0}$ times to conclude that
\begin{eqnarray}\label{jl0}
 \|\chi_{U_{j+1}(B)}\phi(\tau^{m}L)a\|_{2}
&\leq& \sum_{\ell=0}^{\infty}{\rm III}(\ell,\tau)+\sum_{\ell=0}^{\infty}{\rm IV}(\ell,\tau),
\end{eqnarray}
where
$$
{\rm III}(\ell,\tau)=\sum_{k=0}^{M_{0}-1}r_{B}^{-m}\int_{r_{B}}^{\sqrt[m]{2}r_{B}}s^{m-1}
\|\chi_{U_{j+1}(B)}(1-e^{-s^{m}L})^{M_{0}}G_{k,r_{B},M_{0}}(L)L^{M_{1}}\phi(\tau^{m}L)\chi_{U_{\ell}(B)}(r_{B}^{-mk}L^{M_{0}-k}b)\|_{2}ds
$$
and
$$
{\rm IV}(\ell,\tau)=\sum_{\nu=1}^{(2M_{0}-1)M_{0}}\|\chi_{U_{j+1}(B)}(1-e^{-s^{m}L})^{M_{0}}
e^{-\nu r^{m}L}L^{M_{1}}\phi(\tau^{m}L)\chi_{U_{\ell}(B)}(r_{B}^{-mM_{0}}b)\|_{2}.
$$
We apply Proposition \ref{p2estimate} to conclude that when $\ell \leq j-5$,
\begin{eqnarray*}
&&\hspace{-1.5cm}\sup\limits_{s\in [r_{B},\sqrt[m]{2}r_{B}]}\|\chi_{U_{j+1}(B)}(1-e^{-s^{m}L})^{M_{0}}
G_{k,r_{B},M_{0}}(L)L^{M_{1}}\phi(\tau^{m}L)\chi_{U_{\ell}(B)}\|_{2\rightarrow 2}\\
&\leq& C\tau^{-mM_{1}}2^{-j\kappa_0}\left(\frac{r_{B}}{\tau}\right)^{-\kappa_{0}+\frac{n}{2}}\min\{1, (\tau^{-1/m} r_{B})^{mM_{0}-c_{0}}\}
 \leq  C\tau^{-mM_{1}}2^{-j\kappa_0}\left(\frac{r_{B}}{\tau}\right)^{-\kappa_{0}+\frac{n}{2}}
\end{eqnarray*}
for some $ M_{0}> {c_{0}}/{m}$,
which implies that
\begin{align*}
{\rm III}(\ell,\tau)
&\leq C\tau^{-mM_{1}}2^{-j\kappa_0}\left(\frac{r_{B}}{\tau}\right)^{-\kappa_{0}+\frac{n}{2}}\|r_{B}^{-mk}L^{M_{0}-k}b\|_{L^{2}(U_{\ell}(B))}
\leq C2^{-\ell\epsilon}\tau^{-mM_{1}}2^{-j\kappa_0}\left(\frac{r_{B}}{\tau}\right)^{-\kappa_{0}+\frac{n}{2}}r_{B}^{mM_{1}}\mu(2^{\ell}B)^{-(\frac{1}{q}-\frac{1}{2})}\\
&\leq C2^{-\ell\epsilon}\tau^{-mM_{1}}2^{-j\kappa_0}\left(\frac{r_{B}}{\tau}\right)^{-\kappa_{0}+\frac{n}{2}}r_{B}^{mM_{1}}2^{(j-\ell)n(\frac{1}{q}-\frac{1}{2})}\mu(2^{j}B)^{-(\frac{1}{q}-\frac{1}{2})}.
\end{align*}
Hence,
\begin{align}\label{jl1}
\sum_{\ell=0}^{j-5}\left(\int_{r_{B}}^{2^{j+1}r_{B}}\big|{\rm III}(\ell,\tau)\big|^2 \frac{d\tau}{\tau}\right)^{\frac{1}{2}}
&\leq C\sum_{\ell=0}^{j-5}2^{-\ell\epsilon}r_{B}^{mM_{1}}2^{(j-\ell)n(\frac{1}{q}-\frac{1}{2})}\mu(2^{j}B)^{-(\frac{1}{q}-\frac{1}{2})}
2^{-j\kappa_0}\left(\int_{r_{B}}^{2^{j+1}r_{B}}\left(\frac{r_{B}}{\tau}\right)^{-2\kappa_{0}+n}\frac{d\tau}{\tau^{2mM_{1}+1}}\right)^{\frac{1}{2}}\nonumber\\
&\leq C\sum_{\ell=0}^{j-5}2^{-\ell\epsilon}2^{(j-\ell)n(\frac{1}{q}-\frac{1}{2})}\mu(2^{j}B)^{-(\frac{1}{q}-\frac{1}{2})}2^{-j\kappa_{0}}
 \leq C2^{-j(\kappa_{0}-n(\frac{1}{q}-\frac{1}{2}))}\mu(2^{j}B)^{-(\frac{1}{q}-\frac{1}{2})},
\end{align}
where in the next to the last inequality we choose $M_{1}$ sufficient large such that the integral can be bounded by a uniform constant $C>0$.
Similarly,
\begin{align}\label{jl2}
\sum_{\ell=j+5}^{\infty}\left(\int_{r_{B}}^{2^{j+1}r_{B}}\big|{\rm III}(\ell,\tau)\big|^2\frac{d\tau}{\tau}\right)^{\frac{1}{2}}
&\leq C \sum_{\ell=j+5}^{\infty}2^{-\ell\epsilon}r_{B}^{mM_{1}}2^{(j-\ell)n(\frac{1}{q}-\frac{1}{2})}\mu(2^{j}B)^{-(\frac{1}{q}-\frac{1}{2})}
2^{-\ell\kappa_0}\left(\int_{r_{B}}^{2^{j+1}r_{B}}\left(\frac{r_{B}}{\tau}\right)^{-2\kappa_{0}+n}\frac{d\tau}{\tau^{2mM_{1}+1}}\right)^{\frac{1}{2}}\nonumber\\
&\leq C\sum_{\ell=j+5}^{\infty}2^{-\ell(\kappa_{0}+n(\frac{1}{q}-\frac{1}{2})+\epsilon)}2^{jn(\frac{1}{q}-\frac{1}{2})}\mu(2^{j}B)^{-(\frac{1}{q}-\frac{1}{2})}
 \leq C2^{-j(\kappa_{0}+\epsilon)}\mu(2^{j}B)^{-(\frac{1}{q}-\frac{1}{2})}.
\end{align}
Also we have that
\begin{eqnarray}\label{jl3}
 \sum_{\ell=j-4}^{j+4}\left(\int_{r_{B}}^{2^{j+1}r_{B}} \big|{\rm III}(\ell,\tau)\big|^2\frac{d\tau}{\tau}\right)^{\frac{1}{2}}
&\leq&C\sum_{\ell=j-4}^{j+4}r_{B}^{-mM_{1}}\left(\int_{r_{B}}^{2^{j+1}r_{B}}\|(\tau^{m}L)^{M_{1}}
\phi(\tau^{m}L)\chi_{U_{\ell}(B)}(r_{B}^{-mk}L^{M_{0}-k}b)\|_{2}^{2}\frac{d\tau}{\tau}\right)^{\frac{1}{2}}\nonumber\\
&\leq&
\sum_{\ell=j-4}^{j+4}r_{B}^{-mM_{1}}\|r_{B}^{-mk}L^{M_{0}-k}b\|_{L^{2}(U_{\ell}(B))}\nonumber\\
&\leq &C2^{-j\epsilon}\mu(2^{j}B)^{-(\frac{1}{q}-\frac{1}{2})}.\nonumber
\end{eqnarray}
The term ${\rm IV}(\ell,\tau)$ can be handled in a similar way. This, in combination with the estimates
 (\ref{jl0}), (\ref{jl1}), (\ref{jl2}) and (\ref{jl3}), implies that for any $j\geq 1$,
\begin{align*}
\|\eta_{j}^{\prime}F\|_{T_{2}^{2}}
\leq C\left(\int_{r_{B}}^{2^{j+1}r_{B}}\|\chi_{U_{j+1}(B)}\phi(\tau^{m}L)a\|_{2}^{2}\frac{d\tau}{\tau}\right)^{\frac{1}{2}}
\leq C2^{-j\sigma}\mu(2^{j}B)^{-(\frac{1}{q}-\frac{1}{2})}
\end{align*}
for some $\sigma>0$.

Finally, it remains to show the estimate (c).
Similar to the proof of the estimate (b), we decompose $M=M_{0}+M_{1}$, where $M_{0}, M_{1}$
are two constants to be chosen large enough later. Then, we iterate the formula (\ref{iter0}) $M_{0}$ times to conclude that
\begin{eqnarray}\label{kl0}
 \|\chi_{2^{j}B}\phi(\tau^{m}L)a\|_{2}
&\leq& \sum_{\ell=0}^{\infty} {\rm V}(\ell,\tau) +\sum_{\ell=0}^{\infty}{\rm VI}(\ell,\tau),
\end{eqnarray}
where
$$
{\rm V}(\ell,\tau)=\sum_{k=0}^{M_{0}-1}r_{B}^{-m}\int_{r_{B}}^{\sqrt[m]{2}r_{B}}s^{m-1}
\|\chi_{2^{j}B}(1-e^{-s^{m}L})^{M_{0}}G_{k,r_{B},M_{0}}(L)L^{M_{1}}\phi(\tau^{m}L)\chi_{U_{\ell}(B)}(r_{B}^{-mk}L^{M_{0}-k}b)\|_{2}ds
$$
and
$$
{\rm VI}(\ell,\tau)= \sum_{\nu=1}^{(2M_{0}-1)M_{0}}\|\chi_{2^{j}B}(1-e^{-s^{m}L})^{M_{0}}e^{-\nu r^{m}L}L^{M_{1}}
\phi(\tau^{m}L)\chi_{U_{\ell}(B)}(r_{B}^{-mM_{0}}b)\|_{2}.
$$
It follows from the $L^{2}$ boundedness of square function that when $\ell\leq j+4$,
\begin{align*}
 \left(\int_{2^{j}r_{B}}^{2^{j+1}r_{B}}\big| {\rm V}(\ell,\tau)\big|^2\frac{d\tau}{\tau}\right)^{\frac{1}{2}}
\leq &C(2^{j}r_{B})^{-mM_{1}}\|r_{B}^{-mk}L^{M_{0}-k}\|_{L^{2}(U_{\ell}(B))}\\
\leq &C2^{-\ell\epsilon}2^{(j-\ell)n(\frac{1}{q}-\frac{1}{2})}2^{-jmM_{1}}\mu(2^{j}B)^{-(\frac{1}{q}-\frac{1}{2})}.
\end{align*}
Therefore,
\begin{align}\label{kl1}
\sum_{\ell=0}^{j+4}\left(\int_{2^{j}r_{B}}^{2^{j+1}r_{B}}\big| {\rm V}(\ell,\tau)\big|^2\frac{d\tau}{\tau}\right)^{\frac{1}{2}}
&\leq C\sum_{\ell=0}^{j+4}2^{-\ell\epsilon}2^{(j-\ell)n(\frac{1}{q}-\frac{1}{2})}2^{-jmM_{1}}\mu(2^{j}B)^{-(\frac{1}{q}-\frac{1}{2})}\nonumber\\
&\leq C2^{-j(mM_{1}-n(\frac{1}{q}-\frac{1}{2}))}\mu(2^{j}B)^{-(\frac{1}{q}-\frac{1}{2})}.
\end{align}
Besides, if we choose $M_{0}$ sufficient large, then the self-adjoint property of the operator $L$ allows us to apply
 Proposition \ref{p2estimate} to conclude that when $\ell\geq j+5$,
\begin{align*}
 \left(\int_{2^{j}r_{B}}^{2^{j+1}r_{B}} \big| {\rm V}(\ell,\tau)\big|^2\frac{d\tau}{\tau}\right)^{\frac{1}{2}}
\leq &C(2^{j}r_{B})^{-mM_{1}}2^{-\ell\kappa_0}\left(\int_{2^{j}r_{B}}^{2^{j+1}r_{B}}\left(\frac{r_{B}}{\tau}\right)^{-2\kappa_{0}+n}
\frac{d\tau}{\tau}\right)^{\frac{1}{2}}\|r_{B}^{-mk}L^{M_{0}-k}b\|_{L^{2}(U_{\ell}(B))}\\
\leq
&C2^{-\ell(\kappa_0+(\frac{1}{q}-\frac{1}{2})n+\epsilon)}2^{j(\kappa_0+(\frac{1}{q}-\frac{1}{2})n-\frac{n}{2})}2^{-jmM_{1}}\mu(2^{j}B)^{-(\frac{1}{q}-\frac{1}{2})}.
\end{align*}
Thus,
\begin{align}\label{kl2}
\sum_{\ell=j+5}^{\infty}\left(\int_{2^{j}r_{B}}^{2^{j+1}r_{B}}\big| {\rm V}(\ell,\tau)\big|^2\frac{d\tau}{\tau}\right)^{\frac{1}{2}}
&\leq C\sum_{\ell=j+5}^{\infty}2^{-\ell(\kappa_0+(\frac{1}{q}-\frac{1}{2})n+\epsilon)}2^{j(\kappa_0+(\frac{1}{q}-\frac{1}{2})n-\frac{n}{2})}2^{-jmM_{1}}\mu(2^{j}B)^{-(\frac{1}{q}-\frac{1}{2})}\nonumber\\
&\leq C2^{-j(mM_{1}+\frac{n}{2}+\epsilon)}\mu(2^{j}B)^{-(\frac{1}{q}-\frac{1}{2})}.
\end{align}
The term ${\rm VI}(\ell,\tau)$ can be handled in a similar way. This, in combination with the estimates (\ref{kl0}),
(\ref{kl1}) and (\ref{kl2}), implies that if we choose $M_{1}>\frac{n}{m}(\frac{1}{q}-\frac{1}{2})$, then for any $j\geq 1$,
\begin{align*}
\|\eta_{j}^{\prime\prime}F\|_{T_{2}^{2}}
\leq C\left(\int_{2^{j}r_{B}}^{2^{j+1}r_{B}}\|\chi_{2^{j}B}\phi(\tau^{m}L)a\|_{2}^{2}\frac{d\tau}{\tau}\right)^{\frac{1}{2}}
\leq C2^{-j\sigma}\mu(2^{j}B)^{-(\frac{1}{q}-\frac{1}{2})}
\end{align*}
for some $\sigma>0$. This finishes the proof of the estimate (c) and then the \textbf{Step I}.

\medskip

\noindent
\textbf{Step II}: $(H_{L}^{q}(X)\cap \mathcal{H}^{2}(X))\subset \mathbb{H}_{L,mol,M,\epsilon}^{q}(X)$.

Let $f\in H_{L}^{q}(X)\cap \mathcal{H}^{2}(X)$, we will establish a molecular $(q,2,M,\epsilon)$-representation for $f$.
 To this end, we modify the argument in \cite{DL} and set $F(x,\tau)=\phi(\tau^{m}L)f(x)$. Then the definition of
  $H_{L}^{q}(X)$ and the $L^{2}$ boundedness of square function imply that $F\in T_{2}^{q}(X)\cap T_{2}^{2}(X)$.
   Therefore, it follows from the Lemma \ref{tentde} that
\begin{align*}
F=\sum_{j}\lambda_{j}A_{j},
\end{align*}
where each $A_{j}$ is a $T_{2}^{q}(X)$-atom, the sum converges in both $T_{2}^{q}(X)$ and $T_{2}^{2}(X)$, and
\begin{align*}
\left(\sum_{j}|\lambda_{j}|^q\right)^{1/q}\leq C\|F\|_{T_{2}^{q}(X)}\leq C\|f\|_{H_{L}^{q}(X)}.
\end{align*}

Besides, by $L^{2}$-functional calculus, we have
\begin{align}\label{repres}
 f(x)=c\int_{0}^{\infty}\phi(\tau^{m}L)\phi(\tau^{m}L)f(x)\frac{d\tau}{\tau}=
c\pi_{L}(F)(x)=c\sum_{j}\lambda_{j}\pi_{L}(A_{j})(x),
\end{align}
where the last sum converges in $L^{2}(X)$ (see \cite[Lemma 3.22]{HLMMY}). Lemma \ref{tentatom} implies that up
to a harmless constant $C>0$, each $\pi_{L}(A_{j})$ is a $(q,2,M,\epsilon)$-molecule associated to $B$ for some
 $\epsilon>0$, which indicates that (\ref{repres}) gives a molecular $(q,2,M,\epsilon)$-representation of $f$
  so that $f\in$ $\mathbb{H}_{L,mol,M,\epsilon}^{q}(X)$.
This finishes the proof of Lemma \ref{moleculardecomposition}.
\end{proof}

\subsection{Interpolation}
The goal of this subsection is to establish the theory of complex interpolation for Hardy spaces.

\begin{proof}[Proof of Lemma \ref{cominter}]
For any $f\in H_{L}^{p}$, $\frac{2n}{2\kappa_0+n}< p <\infty$, we consider
\begin{align*}
Q_{\tau,L}f(x,\tau):=\phi(\tau^{m}L)f(x),\ \ \tau>0,\ x\in X.
\end{align*}
Then by the definition of $H_{L}^{p}(X)$, $Q_{\tau,L}$ embeds the Hardy space $H_{L}^{p}(X)$ isometrically into
the tent space $T_{2}^{p}(X)$ for $\frac{2n}{2\kappa_0+n}< p<\infty$.
Besides, from Lemma \ref{tentatom} we can easily see that the condition (${\rm PEV}_{2,2}^{\kappa,m}$) implies that
for any $\frac{2n}{2\kappa_0+n}< p<\infty$, $\pi_{L}$ is bounded from $T_{2}^{p}(X)$ to $H_{L}^{p}(X)$.
By the $L^{2}$-functional calculus, for any $f\in L^{2}(X)$, there exists a constant $c>0$ such that the following
Calder\'{o}n reproducing formula holds:
\begin{align*}
f(x)=c\pi_{L}(Q_{\tau,L}f)(x).
\end{align*}
Then Lemma \ref{cominter} can be shown by following a similar outline in \cite{HLMMY}.
\end{proof}

\begin{lemma}\label{boundtest}
Let $1\leq p_{0}<2$. Suppose that $T$ is a sublinear operator bounded on $L^{2}(X)$, and let $\{A_{r}\}_{r>0}$
be a family of linear operators acting on $L^{2}(X)$. Assume that there exists a constant $N>\frac{n}{2}$ such that for $j\geq 6$,
\begin{align}\label{test01}
\|\chi_{U_{j}(B)}T(I-A_{r_{B}})\chi_{B}f\|_{2}\leq C2^{-jN}\mu(B)^{-\sigma_{p_{0}}}\|f\|_{p_{0}}
\end{align}
and for $j\geq 0$,
\begin{align}\label{test02}
\|\chi_{U_{j}(B)}A_{r_{B}}\chi_{B}f\|_{2}\leq C2^{-jN}\mu(B)^{-\sigma_{p_{0}}}\|f\|_{p_{0}}
\end{align}
for all ball $B$ with $r_{B}$ the radius of $B$ and all $f$ supported in $B$. Then for any $p_{0}<p\leq2$, $T$ is bounded on $L^{p}$.
\end{lemma}
\begin{proof}
The proof is a slight modification of Theorem 2.1 in \cite{Aus}. We omit the details and leave it to the readers.
\end{proof}

Next, we give the proof of Lemma \ref{Hp}.

\begin{proof}[Proof of Lemma \ref{Hp}]
We first show that $L^{p}(X)\subset H_{L}^{p}(X)$,
 or equivalently, for any $p\in (p_{0},2]$, $S_{L,\phi}$ is bounded on $L^{p}(X)$. By Lemma \ref{boundtest}, it is enough to
  verify that there exists a sufficient large constant $M$ such that the operator $T=S_{L,\phi}$ and $A_{r_{B}}=I-(I-e^{-r_{B}^{m}L})^{M}$
  satisfy the estimates (\ref{test01}) and (\ref{test02}).
Indeed,
\begin{eqnarray*}
&&\hspace{-1.2cm}\|\chi_{U_{j}(B)}S_{L,\phi}(I-e^{-r_{B}^{m}L})^{M}\chi_{B}f\|_{2}\\
&\leq&\left(\int_{U_{j}(B)}\int_{0}^{2^{m(j-2)r_{B}^{m}}}\int_{d(x,y)<\tau^{1/m}}|\phi(\tau L)(I-e^{-r_{B}^{m}L})^{M}
\chi_{B}f(y)|^{2}\frac{d\mu(y)d\tau d\mu(x)}{V(x,\tau^{1/m})\tau}\right)^{\frac{1}{2}}\\
&+&\left(\int_{U_{j}(B)}\int_{2^{m(j-2)r_{B}^{m}}}^{\infty}\int_{d(x,y)<\tau^{1/m}}|\phi(\tau L)(I-e^{-r_{B}^{m}L})^{M}
\chi_{B}f(y)|^{2}\frac{d\mu(y)d\tau d\mu(x)}{V(x,\tau^{1/m})\tau}\right)^{\frac{1}{2}}\\
&=:& {\rm I}+ {\rm II}.
\end{eqnarray*}
Note that if $x\in U_{j}(B)$, $\tau<2^{m(j-2)}r_{B}^{m}$ and $d(x,y)<\tau^{1/m}$, then $y\in U_{j}^{\prime}(B)$, where
\begin{align*}
U_{j}^{\prime}(B):=U_{j-1}(B)\cup U_{j}(B)\cup U_{j+1}(B).
\end{align*}

Hence,
it follows from the doubling condition (\ref{doubling}) and Proposition \ref{p2estimate} that there exist constants $C,c>0$ such that
\begin{align*}
{\rm I}&\leq C\left(\int_{0}^{2^{m(j-2)r_{B}^{m}}}\|\chi_{U_{j}^{\prime}(B)}\phi(\tau L)(I-e^{-r_{B}^{m}L})^{M}\chi_{B}f\|_{2}^{2}
\frac{d\tau}{\tau}\right)^{\frac{1}{2}}\\
&\leq C2^{-j\kappa_{0}}\mu(B)^{-\sigma_{p_{0}}}\left(\int_{0}^{2^{m(j-2)r_{B}^{m}}}
(\tau^{-1/m}r_{B})^{-2\kappa_{0}+n}{\rm min}\{1,(\tau^{-1/m}r_{B})^{2mM-c_{0}}\}\frac{d\tau}{\tau}\right)^{\frac{1}{2}}\|f\|_{p_{0}}\\
&\leq C2^{-j\kappa_{0}}\mu(B)^{-\sigma_{p_{0}}}\|f\|_{p_{0}}.
\end{align*}
Next, we apply Lemma \ref{global00} and the doubling condition (\ref{doubling}) to obtain that
\begin{align*}
{\rm II}&\leq\left(\int_{2^{m(j-2)r_{B}^{m}}}^{\infty}\|\phi(\tau L)(I-e^{-r_{B}^{m}L})^{M}\chi_{B}f(y)\|_{2}^{2}\frac{d\tau}{\tau}\right)^{\frac{1}{2}}\\
&\leq C\left(\int_{2^{m(j-2)r_{B}^{m}}}^{\infty}{\rm min}\{1,(\tau^{-1/m}r_{B})^{2mM}\}
\|V_{\tau^{1/m}}^{-\sigma_{p_{0}}}\chi_Bf\|_{p_{0}}^{2}\frac{d\tau}{\tau}\right)^{\frac{1}{2}}\\
&\leq C\mu(B)^{-\sigma_{p_{0}}}\left(\int_{2^{m(j-2)r_{B}^{m}}}^{\infty}{\rm min}\{1,(\tau^{-1/m}r_{B})^{2mM}\}
\frac{d\tau}{\tau}\right)^{\frac{1}{2}}\|f\|_{p_{0}}\\
&\leq C2^{-jmM}\mu(B)^{-\sigma_{p_{0}}}\|f\|_{p_{0}}.
\end{align*}
Hence, (\ref{test01}) is proved after we choose $M>\frac{n}{2m}$.

Besides, observe that $[I-(I-e^{-r_{B}^{m}L})^{M}]$ is a finite combination of the terms $e^{-kr_{B}^{m}L}$, $k=1,\ldots, M$. Besides, similarly to the argument in \eqref{sssi}, we conclude that for any $j> 2$, the semigroup $e^{-kr_{B}^{m}L}$
  satisfies the following estimate
\begin{align*}
\|\chi_{U_{j}(B)}e^{-kr_{B}^{m}L}\chi_{B}f\|_{2}
&\leq C2^{-j\kappa_0}\mu(B)^{-\sigma_{p_{0}}}\|f\|_{p_{0}}.
\end{align*}
This, together with the Lemma \ref{heatbound}, shows (\ref{test02}).

Next, by the same argument as in the proof of \cite[Theorem 4.19]{Uhl}, our proof can be reduced to showing the $L^{p}$ boundedness of square function $G_{L,\phi}$, where $p\in(2,p_{0}^{\prime})$ and $G_{L,\phi}$ is defined by
$$G_{L,\phi}f(x):=\left(\int_{0}^{\infty}|\phi(tL)f(x)|^{2}\frac{dt}{t}\right)^{1/2},$$
while this boundedness can be obtained by verifying the condition in \cite[Theorem 2.2]{Aus} as verifying the estimates (\ref{test01}) and (\ref{test02}) (see also \cite[p.78]{Aus}). This ends the proof of Lemma \ref{Hp}.
\end{proof}
\subsection{Spectral multipliers theorem on the Hardy space $H_{L}^{q}(X)$ for $\frac{2n}{2\kappa_0+n}<q\leq1$}
Under the assumption that the operator $L$ satisfies the Gaussian upper bounds $({\rm GE_{m}})$, the spectral multipliers
theorem on the Hardy space $H_{L}^{1}(X)$ was shown in \cite{DY3}. Now, with the help of Proposition \ref{p2estimate} and
Lemma $\ref{moleculardecomposition}$, we can extend this result under a weaker assumption that $L$ satisfies the inequality
 (${\rm PEV}_{2,2}^{\kappa,m}$). Such a result is a helpful tool to obtain the boundedness on $L^{p}$ for Schr\"{o}dinger groups.
\begin{theorem}\label{spectralmultiplier}
Given $\frac{2n}{2\kappa_0+n}<q\leq1$. Suppose that $L$ satisfies the estimate (${\rm PEV}_{2,2}^{\kappa,m}$) for some $m>0$ and $\kappa>\kappa_{0}:= \big[\frac{n}{2}\big]+1$.
Assume in addition that $F$ is an even bounded Borel function such that ${\rm sup}_{R>0}\|\eta\delta_{R}F\|_{C^{\kappa_{0}+1}}<\infty$
and some nonzero cutoff function $\eta\in C_{c}^{\infty}(\mathbb{R}_{+})$. Then the operator $F(L)$ is bounded on $H_{L}^{q}(X)$. More precisely,
\begin{align*}
\|F(L)\|_{H_{L}^{q}(X)\rightarrow H_{L}^{q}(X)}\leq C\left(\sup\limits_{R>0}\|\eta\delta_{R}F\|_{C^{\kappa_{0}+1}}+F(0)\right).
\end{align*}
\end{theorem}
\begin{proof}
We will modify the proof of Proposition \ref{tool} and use the same notation as before (except that we will use $F(\lambda)$ to denote an even bounded Borel function instead of $(I+\lambda)^{-n(\frac{1}{q}-\frac{1}{2})}$) to show this theorem.

It suffices to show that
\begin{align}\label{goalhardy0}
\left\|\left(\int_{0}^{+\infty}\int_{d(x,y)<\tau^{1/m}}|\phi(\tau L)F(L)a(y)|^{2}\frac{d\mu(y)}{V(x,\tau^{1/m})}
\frac{d\tau}{\tau}\right)^{\frac{1}{2}}\right\|_{L^{q}(X)}\leq C\left(\sup\limits_{R>0}\|\eta\delta_{R}F\|_{C^{\kappa_{0}+1}}+F(0)\right).
\end{align}
We may assume in the sequel that $F(0)=0$. Otherwise, we may replace $F$ by $F-F(0)$.

By the formula (\ref{iterate}), the theorem can be reduced to showing that for $k=0,1,\cdots,M$,
\begin{align*}
\left\|\left(\int_{0}^{\infty}\int_{d(x,y)<\tau^{1/m}}|E_{k}(y)|^{2}\frac{d\mu(y)}{V(x,\tau^{1/m})}\frac{d\tau}{\tau}\right)^{\frac{1}{2}}\right\|_{L^{q}}
\leq C\sup\limits_{R>0}\|\phi\delta_{R}F\|_{C^{\kappa_{0}+1}}.
\end{align*}

\medskip

\noindent
{\bf Case 1.} \ $k=0, 1, \cdots, M-1$. \ In this case, we see that
 \begin{eqnarray}\label{e1}
&&\hspace{-1.2cm}\left\|\left(\int_0^{\infty}\!\!\!\!\int_{\substack{  d(x,y)<\tau^{1/m}}}
|E_k(y)|^2 {d\mu(y)\over V(x,\tau^{1/m})}{d\tau\over \tau}\right)^{1/2}\right\|_{L^{q}}\nonumber\\
&\leq& C\left(\sum_{j\geq 0}\sup_{s\in [r_{B},\sqrt[m] 2 r_{B}]}\left\|\left(\int_0^{\infty}\!\!\!\!\int_{\substack{  d(x,y)<\tau^{1/m}}}
|F_{\tau,s}(L)\chi_{U_{j}(B)} (r_{B}^{-mk}L^{M-k} b)(y)|^2 {d\mu(y)\over V(x,\tau^{1/m})}{d\tau\over \tau}\right)^{1/2}\right\|_{L^{q}}^q\right)^{1/q}\nonumber\\
&=:& C\left(\sum_{j\geq 0}\sup_{s\in [r_{B},\sqrt[m] 2 r_{B}]} \|E^{\prime}(k, j, s)\|_{L^{q}(X)}^q\right)^{1/q},
\end{eqnarray}
where
$$
E^{\prime}(k, j, s)=\left(\int_0^{\infty}\!\!\!\!\int_{\substack{  d(x,y)<\tau^{1/m}}}
|F_{\tau,s}(L)\chi_{U_{j}(B)} (r_{B}^{-mk}L^{M-k} b)(y)|^2 {d\mu(y)\over V(x,\tau^{1/m})}{d\tau\over \tau}\right)^{1/2}.
$$

Let us estimate the term $\|E^{\prime}(k, j, s)\|_{L^{q}(X)}.$ To begin with, let $\psi\in C_c^{\infty}(\mathbb{R})$ such
that ${\rm supp}\  \psi \subseteq ({1/8}, 2)$ and supp$\psi=1$ on $({1/4}, 1)$.
Noting that  $\|G_{k,r_{B},M}\|_{L^\infty}\leq C$,
we apply the estimate  (\ref{molecule}), the $L^{2}$-boundedness of the square function  and  the doubling condition (\ref{doubling}) to see that
\begin{eqnarray} \label{eddd}
&&\hspace{-1.2cm}\|E^{\prime}(k, j, s)\|_{L^{q}(64\cdot2^{j} B)}\nonumber\\
&\leq& \left\|\left(\int_0^{\infty}\!\!\!\!\int_{\substack{  d(x,y)<\tau^{1/m}}}
|F_{\tau,s}(L)\chi_{U_{j}(B)} (r_{B}^{-mk}L^{M-k} b)(y)|^2 {d\mu(y)\over V(x,\tau^{1/m})}{d\tau\over \tau}\right)^{1/2}
\right\|_{L^2(64\cdot2^{j} B)} \mu(64\cdot2^{j} B)^{\frac{1}{q}-\frac{1}{2}}\nonumber\\
&\leq&C\left(\int_{0}^{\infty}\|\psi_{\tau}(L)F_{\tau,s}(L)\chi_{U_{j}(B)}(r_{B}^{-mk}L^{M-k}b)\|_{2}^{2}
\frac{d\tau}{\tau}\right)^{\frac{1}{2}}\mu(64\cdot2^{j} B)^{\frac{1}{q}-\frac{1}{2}}\nonumber\\
&\leq& C  \left\|r_{B}^{-mk}L^{M-k}b\right\|_{L^{2}(U_{j}(B))}\mu(64\cdot2^{j} B)^{\frac{1}{q}-\frac{1}{2}}\sup\limits_{R>0}\|\phi\delta_{R}F\|_{L^\infty}\nonumber\\
&\leq& C 2^{-j \epsilon}\sup\limits_{R>0}\|\phi\delta_{R}F\|_{L^\infty}.
\end{eqnarray}

Next we show that for some $\varepsilon'>0$,
\begin{align*}
\|E^{\prime}(k, j, s)\|_{L^{q}((64\cdot2^{j} B)^c)}   \leq  C2^{-j \varepsilon'}\sup\limits_{R>0}\|\phi\delta_{R}F\|_{C^{\kappa_{0}+1}}.
\end{align*}
To begin with, we have a decomposition according to the frequency,
\begin{align}
E^{\prime}(k, j, s)=&\left(\int_0^{\infty}\!\!\!\!\int_{\substack{  d(x,y)<\tau^{1/m}}}
|F_{\tau,s}(L)\chi_{U_{j}(B)} (r_{B}^{-mk}L^{M-k} b)(y)|^2 {d\mu(y)\over V(x,\tau^{1/m})}{d\tau\over \tau}\right)^{1/2}\nonumber\\
\leq& \sum_{\ell\in\ZZ}\left(\int_{2^{-\ell}}^{2^{-\ell+1}}\!\!\!\!\int_{\substack{  d(x,y)<\tau^{1/m}}}
|F_{\tau,s}(L)\chi_{U_{j}(B)} (r_{B}^{-mk}L^{M-k} b)(y)|^2 {d\mu(y)\over V(x,\tau^{1/m})}{d\tau\over \tau}\right)^{1/2}\nonumber\\
=:&\sum_{\ell\in\ZZ} E^{\prime}(k, j, s, \ell).
\end{align}
Let $v_{0}\in\mathbb{Z_{+}}$ be a positive integer such that
\begin{align}\label{nu0}
8<2^{\nu_{0}+j+(\ell-1)/m}r_{B}\leq 16,\ \ & {\rm if}\ 2^{(\ell-1)/m+j}r_{B}\leq \frac{1}{8};\nonumber\\
\nu_{0}=7,\ \ & {\rm if}\ 2^{(\ell-1)/m+j}r_{B}>\frac{1}{8}.
\end{align}
Then it follows from (\ref{nu0}) that if $2^{(\ell-1)/m+j}r_{B}\leq \frac{1}{8}$,
\begin{align}\label{ti1}
\|E^{\prime}(k,j,s)\|_{L^{q}((64\cdot 2^{j}B)^{c})}
&\leq \left(\sum_{\ell\in\mathbb{Z}}\|E^{\prime}(k,j,s,\ell)\|_{L^{q}(B(x_{B}, 8\cdot2^{-(\ell-1)/m}))}^q\right)^{1/q}+\left(\sum_{\ell\in\mathbb{Z}}\sum_{\nu\geq\nu_{0}}\|E^{\prime}(k,j,s,\ell)\|_{L^{q}(U_{\nu+j}(B))}^q\right)^{1/q}\nonumber\\
&=:{\rm I}(k,j,s)+ {\rm II}(k,j,s).
\end{align}
Note that there is no term ${\rm I}(k,j,s)$ if $2^{(\ell-1)/m+j}r_{B}>\frac{1}{8}$. Besides, when $2^{(\ell-1)/m+j}r_{B}\leq\frac{1}{8}$,
similar to the procedure of dealing with the term ${\rm I}(k,j,s)$ defined by (\ref{discuss}), we can easily show that
\begin{align}\label{ti2}
{\rm I}(k,j,s)\leq C2^{-j\epsilon}\sup\limits_{R>0}\|\phi\delta_{R}F\|_{C^{\kappa_{0}+1}}.
\end{align}
By the estimates (\ref{ti1}) and (\ref{ti2}), it remains to estimate the second term ${\rm II}(k,j,s)$.
To estimate this term, we first note that (\ref{nu0}) implies that for $\tau\in [2^{-\ell},2^{-\ell+1}]$, we have $\tau^{1/m}\leq 2^{1/m}2^{-\ell/m}\leq 2^{\nu+j-2}r_{B}.$ Hence, if $d(x,y)<\tau^{1/m}$ and $x\in U_{\nu+j}(B)$, then $y\in U_{\nu+j}^{\prime\prime}(B)$, where
\begin{align*}
U_{\nu+j}^{\prime\prime}(B):=U_{\nu+j+1}(B)\cup U_{\nu+j}(B)\cup U_{\nu+j-1}(B).
\end{align*}
Then we have
\begin{eqnarray}\label{pri0}
&&\hspace{-1.2cm}\|E^{\prime}(k, j, s, \ell)\|^2_{L^2({U_{\nu+j}(B)} )}\nonumber \\
&\leq&C\int_{2^{-\ell}}^{2^{-\ell+1}}\int_{U_{\nu+j}^{\prime\prime}(B)}|F_{\tau,s}(L)
\chi_{U_{j}(B)} (r_{B}^{-mk}L^{M-k} b)(y)|^2\int_{\substack{  d(x,y)<\tau^{1/m}}}\,d\mu(x)
 {d\mu(y)\over V(y,\tau^{1/m})}{d\tau\over \tau}\nonumber\\
&\leq& C\sum_{w=-1}^{1}\int_{2^{-\ell}}^{2^{-\ell+1}}\|\chi_{U_{\nu+j+w}(B)}F_{\tau,s}(L)\chi_{U_{j}(B)} (r_{B}^{-mk}L^{M-k} b)\|_{2}^2
{d\tau\over \tau}\nonumber\\
&\leq& C\sum_{w=-1}^{1}\int_{2^{-\ell}}^{2^{-\ell+1}} \|\chi_{U_{\nu+j+w}(B)}F_{\tau,s}(L)\chi_{U_{j}(B)}
\|_{2\rightarrow 2}^2 \|r_{B}^{-mk}L^{M-k} b\|_{L^{2}(U_{j}(B))}^{2}{d\tau\over \tau}.
\end{eqnarray}
It follows from Proposition \ref{p2estimate} that there exist constants $C, c_{0}>0$ such that for
 $w\in \{-1,0,1\}$ and $\nu\geq 7$, $2^{-\ell}\leq\tau\leq 2^{-l+1}$,
\begin{align}\label{off22}
\|\chi_{U_{\nu+j+w}(B)}F_{\tau,s}(L)\chi_{U_{j}(B)}\|_{2\rightarrow 2}
\leq C2^{-\nu\kappa_{0}}(2^{\ell/m}2^{j}r_{B})^{-\kappa_{0}+\frac{n}{2}}(2^{\ell/m}r_{j,\ell})^{-c_{0}}\|\delta_{\tau^{-1}}F_{\tau,s}\|_{C^{\kappa_{0}+1}},
\end{align}
where $r_{j,\ell}={\rm min}\{2^{j}r_{B}, 2^{-\ell/m}\}$.
Note that the conditions supp$\phi\subset (1/4,1)$ and $r_{B}\leq s\leq \sqrt[m]{2}r_{B}$ implies if $2^{-\ell}\leq\tau\leq 2^{-l+1}$, then
\begin{align}\label{van0}
\|\delta_{\tau^{-1}}F_{\tau,s}\|_{C^{\kappa_{0}+1}}\leq C\min\{1, (2^{\ell/m} r_{B})^{mM}\}\sup\limits_{R>0}\|\phi\delta_{R}F\|_{C^{\kappa_{0}+1}}.
\end{align}
Combining the estimates (\ref{pri0}), (\ref{off22}), (\ref{van0}), (\ref{van1}) with (\ref{molecule}), we conclude that
\begin{eqnarray*}
&&\hspace{-1.2cm}\|E^{\prime}(k, j, s, \ell)\|_{L^2({U_{\nu+j}(B)} )}\\
&\leq& C2^{-j\epsilon}2^{-\nu\kappa_{0}}\mu(2^{j}B)^{-(\frac{1}{q}-\frac{1}{2})}(2^{\ell/m}2^{j}r_{B})^{-\kappa_{0}+\frac{n}{2}}
\min\{1, (2^{\ell/m} r_{B})^{mM-c_{0}}\}\sup\limits_{R>0}\|\phi\delta_{R}F\|_{C^{\kappa_{0}+1}}.
\end{eqnarray*}
This, in combination with the doubling condition (\ref{doubling}), yields{
\begin{align}\label{tii}
&{\rm II}(k,j,s)\nonumber\\&\leq C\left(\sum_{\ell\in\mathbb{Z}}\sum_{\nu\geq\nu_{0}}
\|E^{\prime}(k,j,s,\ell)\|_{L^{2}(U_{\nu+j}(B))}^q\mu(2^{\nu+j}B)^{q(\frac{1}{q}-\frac{1}{2})}\right)^{1/q}\nonumber\\
&\leq C2^{-j(\kappa_{0}-\frac{n}{2}+\epsilon)}\left(\sum_{\ell\in\mathbb{Z}}\sum_{\nu\geq\nu_{0}}
2^{-q\nu(\kappa_{0}-n(\frac{1}{q}-\frac{1}{2}))}(2^{\ell/m}r_{B})^{-q(\kappa_{0}-\frac{n}{2})}\min\{1, (2^{\ell/m} r_{B})^{qmM-qc_{0}}\}\right)^{1/q}\sup\limits_{R>0}\|\phi\delta_{R}F\|_{C^{\kappa_{0}+1}}\nonumber\\
&\leq  2^{-j(\kappa_0-\frac{n}{2}+\epsilon)}\left(\sum_{2^{\frac{\ell-1}{m}+j}r_B>\frac{1}{8}}(2^{\ell/m}r_B)^{-q(\kappa_0-\frac{n}{2})})\right)^{1/q}\sup\limits_{R>0}\|\phi\delta_{R}F\|_{C^{\kappa_{0}+1}}\nonumber\\
&+C2^{-j(n(\frac{1}{q}-\frac{1}{2})-\frac{n}{2}+\epsilon)}\left(\sum_{2^{\frac{\ell-1}{m}+j}r_B\leq \frac{1}{8}}
(2^{\ell/m}r_B)^{-q(n(\frac{1}{q}-\frac{1}{2})-\frac{n}{2})}\min\{1, (2^{\ell/m} r_{B})^{qmM-qc_{0}}\}\right)^{1/q}\sup\limits_{R>0}\|\phi\delta_{R}F\|_{C^{\kappa_{0}+1}}\nonumber\\
&\leq C2^{-j\epsilon}\sup\limits_{R>0}\|\phi\delta_{R}F\|_{C^{\kappa_{0}+1}},
\end{align}
}
It follows from the estimates (\ref{eddd}), (\ref{ti1}), (\ref{ti2}) and (\ref{tii})  that for $k=0,1,\cdots,M-1$,
\begin{align*}
\left\|\left(\int_{0}^{\infty}\int_{d(x,y)<\tau^{1/m}}|E_{k}(y)|^{2}\frac{d\mu(y)}{V(x,\tau^{1/m})}\frac{d\tau}{\tau}\right)^{\frac{1}{2}}\right\|_{L^{q}}
\leq C\sup\limits_{R>0}\|\phi\delta_{R}F\|_{C^{\kappa_{0}+1}}.
\end{align*}

\medskip

\noindent
{\bf Case 2.} \ $k=M$.
Similarly to the proof of estimating $E_{k}$ for $k=1,2,\cdots, M-1$ as in {\bf Case 1}, we conclude that
\begin{align*}
\left\|\left(\int_{0}^{\infty}\int_{d(x,y)<\tau^{1/m}}|E_{M}(y)|^{2}\frac{d\mu(y)}{V(x,\tau^{1/m})}\frac{d\tau}{\tau}\right)^{\frac{1}{2}}\right\|_{L^{q}}
\leq C\sup\limits_{R>0}\|\phi\delta_{R}F\|_{C^{\kappa_{0}+1}}.
\end{align*}
This finishes the proof of (\ref{goalhardy0}) and then Theorem \ref{spectralmultiplier}.
\end{proof}

\bigskip

 \noindent
 {\bf Acknowledgements:} The authors want to thank the referee for helpful comments. The authors want to thank Felipe Ponce-Vanegas for pointing out the mistakes made in the previous version. P. Chen was supported by NNSF of China 11501583, Guangdong Natural Science Foundation
2016A030313351. Duong is supported by the Australian Research Council (ARC) through the research grants  DP190100970.
Z.J. Fan was supported by International Program for Ph.D. Candidates from Sun Yat-Sen University.
J. Li is supported by the Australian Research Council (ARC) through the
research grant DP170101060 and by Macquarie University Research Seeding Grant.
 L. Yan was supported by the NNSF of China, Grant No.  ~11871480, and by the Australian Research
  Council (ARC) through the research grants  DP190100970.

\end{document}